\documentclass[10pt,reqno]{amsart}
\usepackage{fullpage}
\usepackage[T1]{fontenc}                                   
\usepackage{amsfonts}
\usepackage[utf8]{inputenc}                               
\usepackage{comment}                                       
\usepackage{mparhack}                                      
\usepackage{amsmath,amssymb,amsthm,mathrsfs,eucal,mathtools}                      
\usepackage{booktabs}                                                                          
\usepackage{graphicx,subfig}         
\usepackage{wrapfig}                                                                      
\usepackage[bookmarks=true,colorlinks=true]{hyperref}                      
\usepackage{textcomp}
\usepackage{tikz}
\usepackage{pgfplots}
\usepackage{dsfont}

\usepackage{multicol}

\newtheorem{defn}{Definition}[section]
\newtheorem{thm}{Theorem}[section]
\newtheorem{prop}{Proposition}[section]
\newtheorem{lem}{Lemma}[section]

\newtheorem{rem}{Remark}[section]

\DeclareMathOperator*{\argmin}{argmin}

\newcommand{\R}{\mathbb{R}}
\newcommand{\Rd}{{\mathbb{R}^{n}}}

\newcommand{\diag}{\mathrm{dg}}
\newcommand{\mF}{{\mathcal{F}}}

\newcommand{\mP}{{\mathscr{P}}}
\newcommand{\mpd}{{\mathscr{P}_{2}(\Rd)}}

\newcommand{\ddt}{\frac{d}{dt}}
\newcommand{\id}{\mathrm{id}}
\newcommand{\Xre}{\mathcal{X}_{R_\rho,R_\eta}}
\newcommand{\Kre}{\mathcal{K}_{R_\rho,R_\eta}}
\newcommand{\Hre}{\mathcal{H}_{R_\rho,R_\eta}}
\newcommand{\cmei}{cm_\eta^1}
\newcommand{\Xrre}{\mathcal{X}_{ R_\rho,L_\eta,R_\eta}}
\newcommand{\Krre}{\mathcal{K}_{R_\rho,L_\eta,R_\eta}}
\newcommand{\Hrre}{\mathcal{H}_{R_\rho,L_\eta,R_\eta}}

\pgfplotsset{compat=1.12}

\begin{document}
\author{S. Fagioli \and Y. Jaafra}
\address{ Simone Fagioli, Yahya Jaafra - DISIM - Department of Information Engineering, Computer Science and Mathematics, University of L'Aquila, Via Vetoio 1 (Coppito)
67100 L'Aquila (AQ) - Italy}
\email{simone.fagioli@univaq.it}
\email{yahyaya.jaafra@graduate.univaq.it}
\title[Multiple Patterns Formation for an Aggregation/Diffusion Predator-Prey System]{Multiple Patterns Formation for an Aggregation/Diffusion Predator-Prey System}
\date{}

\begin{abstract}
\noindent
We investigate existence of stationary solutions to an aggregation/diffusion system of PDEs, modelling a two species predator-prey interaction. In the model this interaction is described by non-local potentials that are mutually proportional by a negative constant $-\alpha$, with $\alpha>0$. Each species is also subject to  non-local self-attraction forces together with quadratic diffusion effects. The competition between the aforementioned mechanisms produce a rich asymptotic behaviour, namely the formation of steady states that are composed of multiple bumps, i.e. sums of Barenblatt-type profiles. The existence of such stationary states, under some conditions on the positions of the bumps and the proportionality constant $\alpha$, is showed for small diffusion, by using the functional version of the Implicit Function Theorem. We complement our results with some numerical simulations, that suggest a large variety in the possible strategies the two species use in order to interact each other.
\end{abstract}



\maketitle

\section{Introduction}
The mathematical modelling of the collective motion through aggregation/diffusion phenomena has raised a lot of interest in the recent years and it has been deeply studied for its application in several areas, such as biology \cite{boi,mogilner,topaz,lewis}, ecology \cite{krause,okubo,parrish}, animal swarming  \cite{BelHa,BelSol,Morale2005,tomkins_kolokolnikov} sociology and economics, \cite{ During, TT1, TT2,Toscani}. 
One of the common idea in this modelling approach is that a certain population composed by \emph{agents} evolves according to \emph{longe-range attraction} and \emph{short-range repulsion} forces between agents. We are interested in modelling the problem of predator-prey interactions, namely we consider two populations that attract (\emph{prey}) and repel (\emph{predators}) each others.  The pioneering works for this problem are the ones by  Lotka, \cite{lotka} and Volterra\cite{volterra},  which describe the predator-prey interaction via reaction terms in a set of differential equations, possibly combined with diffusion terms, see \cite{murray} and the references therein.

As in  \cite{DFrFag2}, in this paper we model predator-prey interactions  via a \emph{transport} terms rather than a reaction ones as follows: consider $N$ predators located at $X_1,\ldots,X_N\in\Rd$, and $M$ prey at $Y_1,\ldots,Y_M\in\Rd$ with masses $m_X^i>0$ and $m_Y^i>0$ respectively. We assume that each agent of the same species interacts under the effect of a radial non-local force that is  attractive in the long range and repulsive in the short range. Moreover, predators are attracted by the prey, while the latter are subject to a repulsive force from the predators, that is proportional to the previous one. This set of modelling assumptions leads to the following system of ODEs:
\begin{equation}\label{eq:particle_intro}
  \begin{cases}
    \displaystyle
  \dot{X}_i(t)=-\sum_{k=1}^N m_X^k\left(\nabla S_1^r(X_i(t)-X_k(t)) + \nabla S_1^a(X_i(t)-X_k(t))\right) - \sum_{h=1}^M m_Y^h \nabla K(X_i(t)-Y_h(t)),
  & \\
  \displaystyle
  \dot{Y}_j(t)=-\sum_{h=1}^M m_Y^h \left(\nabla S_2^r(X_i(t)-X_k(t)) +\nabla S_2^a(Y_j(t)-Y_h(t))\right) + \alpha\sum_{k=1}^M m_X^k \nabla K(Y_j(t)-X_k(t)),&
  \end{cases}
\end{equation}
with $i=1,\ldots,N$ and $j=1,\ldots,M$. The potentials $S_1^a$ and $S_2^a$ are called \emph{self-interaction} and model the long-range attraction among agents of the same species. The potential $K$ is responsible for the predator-prey interaction, and it is called \emph{cross-interaction} potential.  The coefficient $\alpha>0$ models the \emph{escape propensity} of prey from the predators. The short-range repulsion among particles of the same species is modelled by the non-local forces $S_1^r$ and $S_2^r$, this range usually scales with the number of particles, $S_i^r(z) =N^\beta S(N^{\beta/n}z)$ for a smooth functional $S$, see \cite{Morale2005}.

The formal limit when the number of particles tends to infinity leads to the following system of partial differential equations
\begin{equation}\label{eq:continuum_intro}
\begin{dcases}
\partial_{t}\rho = \mbox{div}\big( \rho \nabla\big( d_{1}\rho - S_\rho*\rho-K*\eta \big)\big),
\\
\partial_{t}\eta = \mbox{div}\big( \eta \nabla\big( d_{2}\eta - S_\eta*\eta+\alpha K*\rho \big)\big),
\end{dcases}
\end{equation} 
where $\rho$ and $\eta$ are the densities of predators and prey respectively. Through this limit the (non-local) short-range repulsion formally turns to a (local) nonlinear diffusion terms, being $d_1$ and $d_2$ positive constants modelling the spreading velocity, while the long-range attraction takes into account the non-local self-interactions. 
We can therefore lighten the notation by calling  $S_1^a=S_\rho$ and $S_2^a = S_\eta$.

The goal of this paper is to show that the model above catches one of the main features that occur in nature, namely the formation of  \emph{spatial patterns} where the predators are surrounded of empty zones and the prey aggregates around, that is usually observed in fish schools or in flock of sheeps, see \cite{Guo,MiYa}. In the fully aggregative case, namely system  \eqref{eq:continuum_intro} with $d_1=d_2=0$, the formation of these types of patterns has been studied in several papers, see \cite{chen_kolokolnikov,EvKol,DFrFag2,tomkins_kolokolnikov} and references therein.

Existence theory for solutions to system of the form \eqref{eq:continuum_intro} can be performed in the spirit of  \cite{CL1,DEF}. In particular, system \eqref{eq:continuum_intro} should be framed in the context of non symmetrizable systems, for which the Wasserstein gradient flow theory developed in \cite{AGS} and adapted to systems in \cite{DFrFag} does not work. In  \cite{CL1,DEF,DFrFag},  the authors consider different choices for the diffusion part (no diffusion in \cite{DFrFag}, diagonal nonlinear diffusion in \cite{CL1} and cross-diffusion with dominant diagonal in \cite{DEF}),  and the existence of solutions is proved  via an implicit-explicit version of the JKO scheme \cite{JKO}.

We reduce our analysis to the one-dimensional setting
\begin{equation}\label{main-equation}
\begin{dcases}
\partial_{t}\rho = \partial_x\big( \rho \partial_x\big( d_{1}\rho - S_\rho*\rho-K*\eta \big)\big)
\\
\partial_{t}\eta = \partial_x\big( \eta \partial_x\big( d_{2}\eta - S_\eta*\eta+\alpha K*\rho \big)\big)
\end{dcases}
\end{equation} 
Note that by a simple scaling argument we can always assume that $d_1=d_2=d$, indeed it is enough to multiply the first equation by $d_{2}/d_1$, and setting (by an abuse of notation) $d_2 = d$,  $S_\rho=\frac{d_2}{d_1}S_\rho$, $K=\frac{d_2}{d_1}K$ and $\alpha=\frac{d_1}{d_2}\alpha$.
We are interested in existence of stationary solutions to \eqref{main-equation}, that are solutions to the following system
\begin{equation}
\label{stationary system1}
\begin{dcases}
0 = \big( \rho \big( d\rho - S_\rho*\rho-K*\eta \big)_{x}\big)_{x},
\\
0 = \big( \eta \big( d\eta - S_\eta*\eta+\alpha K*\rho \big)_{x}\big)_{x},
\end{dcases}
\end{equation}
as well as their properties, e.g. symmetry, compact support, etc.


The stationary equation for the one species case, i.e.,
\[
\partial_{t}\rho = \partial_x\big( \rho \partial_x\big( d\rho - S*\rho \big)\big)
\]
is studied several papers, see \cite{Bed,BDFF,CaCaHo,ChFeTo} and therein references. In \cite{BDFF} the Krein-Rutman theorem is used in order to characterise the steady states as
eigenvectors of a certain non-local operator. The authors prove that a unique steady state with given mass and
centre of mass exists provided that $d < \|K\|_{L^1}$, and it exhibits a
shape similar to a Barenblatt profile for the porous medium equation; see \cite{Vazq} and \cite{DFJ} for the local stability analysis. Similar techniques are used in \cite{BFH} in order to  partly extend the result to more general nonlinear diffusion, see also \cite{Kaib}. This approach is used in \cite{BDFFS} in order to explore the formation of segregated stationary states for a system similar to \eqref{main-equation} but in presence of \emph{cross-diffusion}. Unfortunately, when dealing with systems, it is not possible to reproduce one of the major
issues solved in \cite{BDFFS}, namely the one-to-one correspondence between the
diffusion constant (eigenvalue) and the support of the steady state. A \emph{support-independent} existence result for small diffusion coefficient $d$ is obtained in \cite{BDFFS} by using the generalised version of the implicit function theorem, see also \cite{budif} where this approach is used in the one species case.

In this paper we apply the aforementioned approach in order to show that stationary solutions to \eqref{main-equation} are composed of multiple Barenblatt profiles. Let us introduce, for fixed $z_\rho, z_\eta >0$,  the following space
\[
 \mathcal{M} = \left\{(\rho,\eta)\in (L^\infty\cap L^1(\R))^2 : \rho,\, \eta\geq0,\, \|\rho\|_{L^1}=z_\rho,\, \|\eta\|_{L^1}=z_\eta\right\}.
\]

\begin{defn}\label{def:multi_bump}
We say that a pair $(\rho,\eta)\in  \mathcal{M} $ is a \emph{multiple bumps steady state} to \eqref{main-equation} if the pair $(\rho,\eta)$  solves \eqref{stationary system1} weakly and there exist two numbers $N_\rho,N_\eta\in \mathbb{N}$, and two families of intervals  $I_\rho^i=\left[l_\rho^i,r_\rho^i\right]$, for $i=1,...,N_\rho$, and  $I_\eta^h=\left[l_\eta^h,r_\eta^h\right]$, for $h=1,...,N_\eta$ such that
\begin{itemize}
    \item   $I_\rho^i\cap I_\rho^j=\emptyset$, for $i,j=1,...,N_\rho$, $i\neq j$ and $I_\eta^h\cap I_\eta^k=\emptyset$, for $h,k=1,...,N_\eta$, $h\neq k$,
    \item $\rho$ and $\eta$ are supported on 
    \[supp(\rho)=\bigcup_{i=1}^{N_\rho} I_\rho^i \quad\mbox{ and }\quad supp(\eta)=\bigcup_{i=1}^{N_\eta} I_\eta^i,
    \]
    respectively and
    \[
    \rho(x) = \sum_{i=1}^{N_\rho}\rho^i(x)\mathds{1}_{I_\rho^i}(x)\quad\mbox{ and }\quad \eta (x) = \sum_{h=1}^{N_\eta}\eta^h(x)\mathds{1}_{I_\eta^h}(x),
    \]
    where, for $i=1,...,N_\rho$ and $h=1,...,N_\eta$, $\rho^i$ and $\eta^h$ are even w.r.t the centres of $I_\rho^i$ and $I_\eta^h$ respectively, non-negative and $C^1$ functions supported on that intervals.
\end{itemize}
\end{defn}
In order to simplify the notations, in the following we will denote with $l\in\{\rho,\eta\}$ a generic index that recognise one of the two families.

We recall that an interaction potential $G$ is said to be radial if $G(x)=g(|x|)$ for some $g:[0,+\infty)\rightarrow \R$. In particular, $G$ is attractive if $g'(r)>0$ for $r>0$, while it is repulsive if $g'(r)<0$ for $r>0$. 
 Throughout the paper we shall require that all the kernels are smooth and attractive. More precisely:
\begin{itemize}
\item[(A1)] $S_\rho$, $S_\eta$ and $K$ are $C^2(\R)$.
\item[(A2)] $S_\rho$, $S_\eta$ and $K$ are radially symmetric and decreasing w.r.t. the radial variable.
\item[(A3)] $S_\rho$, $S\eta$ and $K$ are non-negative, with finite mass on $\R$.
\end{itemize}
The main result of the paper is the following
\begin{thm}\label{main_thm}
Assume that the interaction kernels are under the assumptions (A1), (A2) and (A3). 
Consider $N_\rho,N_\eta\in\mathbb{N}$ and let $z_l^i$ be fixed positive numbers  for $i=1,2,\cdots,N_l,$ and $l \in \{\rho,\eta\}$. Consider two families of real numbers $\{cm_\rho^i\}_{i=1}^{N_\rho}$ and $\{cm_\eta^i\}_{i=1}^{N_\eta}$ such that
    \begin{itemize}
        \item[(i)]  $\{cm_\rho^i\}_{i=1}^{N_\rho}$ and $\{cm_\eta^i\}_{i=1}^{N_\eta}$ are stationary solutions of the purely non-local particle system, that is, for $i=1,2,\cdots,N_l,$ for  $l,h \in \{\rho,\eta\}$ and $l\neq h$,
\begin{equation}\label{bli}
    B_{l}^i = \sum_{j=1}^{N_l} S'_l( cm_l^i - cm_l^j ) z_l^j+  \alpha_l\sum_{j=1}^{N_h} K'( cm_l^i - cm_h^j ) z_h^j=0,
\end{equation}
\item[(ii)] the following quantities 
\begin{equation}\label{dli}
 D_{l}^i = - \sum_{j=1}^{N_l} S''_l(cm_l^i - cm_l^j) z_l^j -\alpha_l \sum_{j=1}^{N_h}  K''(cm_l^i - cm_h^j ) z_h^j, 
\end{equation}
are strictly positive, for all $i=1,2,\cdots,N_l$,  $l,h \in \{\rho,\eta\}$ and $l\neq h$.
\end{itemize}

Then, there exists
a constant $d_0$ such that for all $d\in (0, d_0)$ the stationary equation \eqref{stationary system1} admits a unique solution in the sense of Definition \ref{def:multi_bump} of the form
\[
    \rho(x) = \sum_{i=1}^{N_\rho}\rho^i(x)\mathds{1}_{I_\rho^i}(x)\quad\mbox{ and }\quad \eta (x) = \sum_{h=1}^{N_\eta}\eta^h(x)\mathds{1}_{I_\eta^h}(x)
\]
where 
\begin{itemize}
\item each interval $I_l^i$ is symmetric around $cm_l^i$ for all  $i=1,2,\cdots,N_l$,  $l \in \{\rho,\eta\}$,
\item $\rho^i$ and $\eta^j$ are $C^1$, non-negative and even w.r.t the centres of $I_\rho^i$ and $I_\eta^j$ respectively, with masses $z_\rho^i$ and  $z_\eta^j$, for $i=1,...,N_\rho$ and $j=1,...,N_\eta$,
\item the solutions $\rho$ and $\eta$ have fixed masses 
\[ z_\rho=\sum_{i=1}^{N_\rho} z_\rho^i \mbox { and } z_\eta = \sum_{i=1}^{N_\eta} z_\eta^i,\]
respectively.
\end{itemize}
\end{thm}
The paper is structured as follows. In Section \ref{sec:preliminaries} we recall the basics notions on optimal transport and we introduce the $p$-Wasserstein distances in spaces of probability measures. Then, we  recall the strategy for proving existence to systems of the form \eqref{eq:continuum_intro}. The remaining part of the Section is devoted to the preliminary existence analysis of steady states  via the Krein-Rutman theorem of two particular types of stationary solutions that we call \emph{mixed} and \emph{separated}. In 
Section \ref{sec:implicit}, existence and uniqueness results for multiple bumps stationary solutions are proved in case
of small diffusion coefficient using the implicit function theorem. We conclude the paper with Section \ref{sec:numerics}, complementing our results with numerical simulations that also show interesting stability issues of the stationary states, namely transitions between states and others effects such as travelling waves profiles.

\section{Preliminary results}\label{sec:preliminaries}
\subsection{Tools in Optimal Transport} 
We start collecting preliminary concepts on the Wasserstein distance.

Let $\mP(\R^n)$ be the space of probability measures on $\R^n$ and fix $p \in [1,+\infty)$. The space of probability measures with finite $p$-moment is defined by
\begin{equation*}
     \mP_{p}(\R^n)=\left\{\mu \in \mP(\R^n): m_{p}(\mu)=\int_{\R^n}\left|x\right|^{p}d\mu(x)<\infty \right\}.
\end{equation*}
For a measure $\mu\in\mP(\R^n)$ and a Borel map $T:\R^n\rightarrow\R^k$, denote with $T_{\#}\mu\in\mP(\R^n)$ the push-forward of $\mu$ through $T$, defined by
\begin{equation*}
    \int_{\R^k}f(y)dT_{\#}\mu(y)=\int_{\R^n}f(T(x))d\mu(x) \qquad\mbox{for all $f$ Borel functions on $\R^k$.}
\end{equation*}
We endow the space $\mP_{p}(\R^n)$ with the Wasserstein distance, see for instance \cite{AGS,Sant,villani}
\begin{equation}\label{def:wasserstein}
     W_{p}^{p}(\mu,\nu)=\inf_{\gamma\in \Gamma(\mu,\nu)}\left\{ \int_{\R^n\times\R^n}|x-y|^{p}d\gamma(x,y)\right\},
\end{equation}
where $\Gamma(\mu_{1},\mu_{2})$ is the class of transport plans between $\mu$ and $\nu$, that is the class of measures $\gamma \in\mP(\R^n\times\R^n)$ such that, denoting by $\pi^{i}$ the projection operator on the $i$-th component of the product space, the marginality condition $\pi^{i}_{\#}\gamma=\mu_{i} $ $i=1,2$ is satisfied. 

Since we are working in a `multi-species' structure, we consider the product space $ \mP_{p}(\R^n)\times \mP_{p}(\R^n)$ endowed with a product structure. In the following we shall use bold symbols to denote elements in a product space.
For a $p\in [1,+\infty]$, we use the notation
\begin{equation*}
    \mathcal{W}_{p}^{p}(\bar{\mu},\bar{\nu})=W_{p}^{p}(\mu_{1},\nu_{1})+W_{p}^{p}(\mu_{2},\nu_{2}),
\end{equation*}
with $\bar{\mu}=(\mu_1,\mu_2),\bar{\nu}=(\nu_1,\nu_2) \in \mP_{p}(\R^n)\times \mP_{p}(\R^n)$.
In the one-dimensional case, given $\mu\in \mP (\R)$, we introduce the pseudo-inverse variable $u_\mu \in L^1([0,1];\R)$ as
\begin{equation}\label{eq:pseudoinverse}
u_\mu(z) \doteq \inf \bigl\{ x \in \R \colon \mu((-\infty,x]) > z \bigr\}, \quad z\in [0,1],
\end{equation}
see \cite{CarTos}.
\subsection{Weak solutions for the time-dependent system}
In the Introduction we already mention that the well-posedness of \eqref{main-equation} can be stated according to the results in \cite{CL1,DEF} in an arbitrary space dimension $n$. In these papers, the existence of weak solutions is provided using an implicit-explicit version of the Jordan-Kinderlehrer-Otto (JKO) scheme \cite{JKO,DFrFag}, that we sketch it in the following.  A key point in this approach is to associate to \eqref{main-equation} a \emph{relative energy functional}
\begin{align*}
\mF_{\left[\mu,\nu\right]}(\rho,\eta) & =\frac{d}{2}\int_{\R^n}\rho^2+\eta^2 dx - \frac{1}{2}\int_{\R^n}\rho S_\rho\ast \rho dx- \frac{1}{2}\int_{\R^n}\eta S_\eta\ast \eta dx\\
&- \int_{\R^n}\rho K \ast \mu dx+ \alpha\int_{\R^n}\eta K \ast \nu dx,
\end{align*}
for a fixed reference couple of measures $(\mu,\nu)$. We state our definition of weak measure solution for \eqref{main-equation}, in the space $\mP_2(\Rd)^2$.

\begin{defn}\label{defweaksolution}
A curve $\bar{\mu}=(\rho(\cdot),\eta(\cdot)):[0,+\infty)\longrightarrow\mpd^2$ is a weak solution to \eqref{main-equation} if 
\begin{itemize}
\item[(i)] $\rho,\,\eta\in L^{2}([0,T]\times \R^n)$ for all $T>0$, and $\nabla \rho,\,\nabla \eta \in L^2([0,+\infty)\times \R^n)$ for $i=1,2$,
\item[(ii)] for almost every $t\in [0,+\infty)$ and for all $\phi,\varphi\in C_c^{\infty}(\Rd)$, we have
\begin{align*}
  \ddt\int_\Rd\phi\rho dx & =-d\int_\Rd\rho\nabla\rho\cdot\nabla\phi\,dx+\int_\Rd\rho\left(\nabla S_\rho \ast \rho + \nabla K\ast\eta\right)\nabla \phi \,dx, \\ 
\ddt\int_\Rd\varphi\eta dx & =-d\int_\Rd\eta\nabla\eta\cdot\nabla\varphi\,dx+\int_\Rd\eta\left(\nabla S_\eta \ast \eta -\alpha \nabla K\ast\rho\right)\nabla \phi \,dx. 
\end{align*}

\end{itemize}
\end{defn}

\begin{thm}\label{exthm}
Assume that (A1)-(A3) are satisfied. 
Let $\bar{\mu}_0=(\rho_{1,0},\rho_{2,0})\in \mpd^2$ such that
\[\mF_{[\bar{\mu}_0]}\left(\bar{\mu}_0\right)<+\infty.\]
Then, there exists a weak solution to \eqref{main-equation} in the sense of Definition \ref{defweaksolution}.
\end{thm}
As already mentioned the proof of Theorem \ref{exthm} is a special case of the results in \cite{CL1,DEF} and consists in the following main steps:
\begin{enumerate}
    \item Let $\tau>0$ be a fixed time step and let $\bar{\mu}_0=\in\mpd^2$ be a fixed initial datum such that $\mF_{[\bar{\mu}_0]}\left(\bar{\mu}_0\right)<+\infty$. Define a sequence $\left\{\bar{\mu}_\tau^k\right\}_{k\in\mathbb{N}}$ recursively: $\bar{\mu}_\tau^0=\bar{\mu}_0$ and, for a given $\bar{\mu}_\tau^k\in\mpd^2$ with $n\geq 0$, we choose $\bar{\mu}_\tau^{k+1}$ as follows:
\begin{equation}\label{jko}
\bar{\mu}_\tau^{k+1}\in\argmin_{\bar{\mu}\in\mpd^2}\left\{\frac{1}{2\tau}\mathcal{W}_2^2(\bar{\mu}_\tau^k,\bar{\mu})+\mF_{[\bar{\mu}_\tau^k]}(\bar{\mu})\right\}.
\end{equation}
Let $N:=\left[\frac{T}{\tau}\right]$, set
$$
\bar{\mu}_\tau(t)=(\rho_\tau(t),\eta_\tau(t))=\bar{\mu}_\tau^k \qquad t\in((k-1)\tau,k\tau],
$$
with $\bar{\mu}_\tau^k$ defined in \eqref{jko}. 
\item There exists an absolutely continuous curve $\tilde{\bar{\mu}}: [0,T]\rightarrow\mpd^2$ such that the piecewise constant interpolation $\bar{\mu}_\tau$ admits a sub-sequence $\bar{\mu}_{\tau_h}$ narrowly converging to $\tilde{\bar{\mu}}$ uniformly in $t\in[0,T]$ as $k\rightarrow +\infty$. This is a standard result coming from the minimising condition \eqref{jko}.
\item There exist a constant $C>0$ such that
\begin{equation}\label{h1bound}
\int_0^T\left[||\rho_\tau(t,\cdot)||_{H^1(\Rd)}^2+||\eta_\tau(t,\cdot)^{\frac{m_2}{2}}||_{H^1(\Rd)}^2 \right]\,dt\le C(T,\bar{\mu}_0),
\end{equation}
and the sequence $\bar{\mu}_{\tau_{h}}:[0,+\infty[\longrightarrow\mpd^2$ converges to $\tilde{\bar{\mu}}$ strongly in 
$$L^{2}((0,T)\times\R^n)\times L^{2}((0,T)\times\R^n),$$ 
for every $T>0$. The estimate in \eqref{h1bound} can be deduced by using the so called \emph{Flow-interchange Lemma} introduced in \cite{MMCS}, see also \cite{DFM}. In order to deduce the strong convergence we use the extended version of the Aubin-Lions Lemma in \cite{RS}.
\item The approximating sequence $\bar{\mu}_{\tau_h}$ converges to a weak solution $\tilde{\bar{\mu}}$ to \eqref{main-equation}. This can be showed considering two consecutive steps in the semi-implicit JKO scheme \eqref{jko}, i.e. $\bar{\mu}_\tau^k$, $\bar{\mu}_\tau^{k+1}$, and perturbing  in the following way  
\begin{equation}\label{eq:perturbation}
 \bar{\mu}^\epsilon = (\rho^\epsilon,\eta^\epsilon)=(P_\#^\epsilon\rho_\tau^{k+1},\eta_\tau^k),
\end{equation}
where $P^\epsilon=\id+\epsilon\zeta$, for some $\zeta\in C_c^\infty(\Rd;\Rd)$ and $\epsilon\ge0$. From the minimizing property of $\bar{\mu}_\tau^{k+1}$ we have 
\begin{equation}\label{optimality}
0\leq\frac{1}{2\tau}\left[\mathcal{W}_2^2(\bar{\mu}_\tau^{k+1},\bar{\mu}^\epsilon)- \mathcal{W}_2^2(\bar{\mu}_\tau^{k}, \bar{\mu}_\tau^{k+1})\right]+\mF_{[\bar{\mu}_\tau^{k}]}(\bar{\mu}^\epsilon)-\mF_{[\bar{\mu}_\tau^{k}]}(\bar{\mu}^{\epsilon}).
\end{equation}
After some manipulations, sending first $\epsilon\to 0$ and then $\tau\to 0$ the inequality \eqref{optimality} leads to the first weak formulation in Definition \ref{defweaksolution}. Perturbing now on $\eta$ and repeating the same procedure we get the required convergence.
\end{enumerate}

\subsection{Stationary states for purely non-local systems}
The existence of weak solutions to the purely non-local systems, i.e., 
\begin{equation}\label{nonlocal-equation-main}
\begin{dcases}
\partial_{t}\rho = \mbox{div}\big( \rho \nabla\big( S_\rho*\rho+K_\rho*\eta \big)\big),
\\
\partial_{t}\eta = \mbox{div}\big( \eta \nabla\big( S_\eta*\eta+K_\eta*\rho \big)\big),
\end{dcases}
\end{equation} 
with generic cross-interaction kernels $K_\rho$ and $K_\eta$ is investigated in \cite{DFrFag}, whereas studies on the shape of stationary states can be found in \cite{Cica,EvKol}.  Concerning the predator-prey modelling and patterns formation, in \cite{chen_kolokolnikov,tomkins_kolokolnikov} a minimal version of \eqref{eq:particle_intro} has been considered with only one predator and arbitrarily many prey subject to (different) singular potentials. This model induces the formation of nontrivial patterns in some way to prevent the action of the predators.
In \cite{DFrFag2} the authors study existence and stability of stationary states for the purely aggregative version of system \eqref{main-equation}, namely equation \eqref{main-equation} with  $d=0$, 
\begin{equation}\label{nonlocal-equation}
\begin{dcases}
\partial_{t}\rho = \mbox{div}\big( \rho \nabla\big( S_\rho*\rho+K*\eta \big)\big),
\\
\partial_{t}\eta = \mbox{div}\big( \eta \nabla\big( S_\eta*\eta-\alpha K*\rho \big)\big),
\end{dcases}
\end{equation} 
and its relation with the particle system
\begin{equation}\label{eq:pred_pred_particle}
  \begin{cases}
  \displaystyle
  \dot{X_i}(t)= -\sum_{k=1}^N m_X^k \nabla S_\rho(X_i(t)-X_k(t)) -\sum_{k=1}^M m_Y^k \nabla K(X_i(t)-Y_k(t)), &  \\
  \displaystyle
  \dot{Y_j}(t)= -\sum_{k=1}^M m_Y^k \nabla S_\eta(Y_j(t)-Y_k(t)) +\alpha\sum_{k=1}^N m_X^k \nabla K(Y_j(t)-X_k(t)). &
  \end{cases}
\end{equation}
It is proved that stationary states of system \eqref{nonlocal-equation} are linear combinations of Dirac's deltas, namely $\bar{\rho},\bar{\eta}\in \mP(\R^n)$, with
\begin{equation}\label{delta_sum}		
 \left(\bar{\mu}_{1},\bar{\mu}_{2}\right)=\left(\sum_{k=1}^{N}m_X^{k}\delta_{\bar{X}_{k}}(x),\sum_{h=1}^{M}m_Y^{h}\delta_{\bar{Y}_{h}}(x)\right).
\end{equation}
where $\left\{\bar{X}_{k}\right\}_k$, $\left\{\bar{Y}_{h}\right\}_h$ are stationary solutions of system \eqref{eq:pred_pred_particle}, i.e.,
\begin{equation}\label{eq:pred_particle_steady}
\begin{cases}
\displaystyle
  0 = \sum_{k=1}^N \nabla S_\rho(\bar{X}_k-\bar{X}_i)m_X^k + \sum_{h=1}^M \nabla K(\bar{Y}_h-\bar{X}_i)m_Y^h & \\
  \displaystyle
   0 = \sum_{h=1}^M \nabla S_\eta(\bar{Y}_h-\bar{Y}_j)m_Y^h -\alpha \sum_{k=1}^N \nabla K(\bar{X}_k-\bar{Y}_j)m_X^k & \\
\end{cases}.
\end{equation}
for $i=1,...,N$ and $j=1,...,M$, see also \cite{felrao1,felrao2} for a symilar result in the one-species case. As pointed out in \cite{DFrFag2}, system \eqref{eq:pred_particle_steady} is not enough to determine a unique steady state, since the linear combination of the first $N$ equations, weighted with  $\alpha m_X^i$, and the final $M$ equations weighted with coefficients $-m_Y^j$ get the trivial identity $0=0$. System \eqref{eq:pred_particle_steady} should be coupled with the quantity
\begin{equation}\label{c_alpha}
    C_\alpha=\alpha\sum_{i=1}^{N}m_{X}^{i}X_{i}-\sum_{j=1}^{M}m_{Y}^{j}Y_{j}
\end{equation}
that is a conserved quantities, and therefore one would like to produce a unique steady state once the quantity $C_\alpha$ is prescribed. Solutions to system \eqref{eq:pred_particle_steady} will play a crucial role in the proof of the main Theorem \ref{main_thm}.

\subsection{Existence of some stationary states via Krein-Rutman Theorem} We now prove the existence of two possible shapes of steady states, that will be prototype examples for the general case. The first one is what we can call \emph{mixed steady state}, that identifies the case in which the predators can catch the prey, see \figurename~\ref{fig:mixed_state}.
\begin{figure}[htbp]
\begin{tikzpicture}
\draw[->] (-3,0) -- (6,0);
\draw[->] (1.5,-0.5) -- (1.5,4.2) ;
\draw[scale=1,domain=-0.3:3.3,smooth,variable=\x,red] plot ({\x},{(3.25-(\x-1.5)*(\x-1.5)))});
\draw[scale=1,domain=-1.25:4.25,smooth,variable=\x,blue] plot ({\x},{(1.5-0.2*(\x-1.5)*(\x-1.5)))});

\node[scale=1] at (-0.3,-0.3) {$L_\rho$};
\node[scale=1] at (-1.25,-0.3) {$L_\eta$};
\node[scale=1] at (3.3,-0.3) {$R_\rho$};
\node[scale=1] at (4.25,-0.3) {$R_\eta$};

\node[red,scale=1] at (1.8/2,3.1){$\rho$};
\node[blue,scale=1] at (2.6/2,1.7){$\eta$};

\end{tikzpicture}
\caption{Example of mixed stationary state. Note that by symmetry $L_\rho=-R_\rho$ and $L_\eta=-R_\eta$.}
\label{fig:mixed_state}
\end{figure}
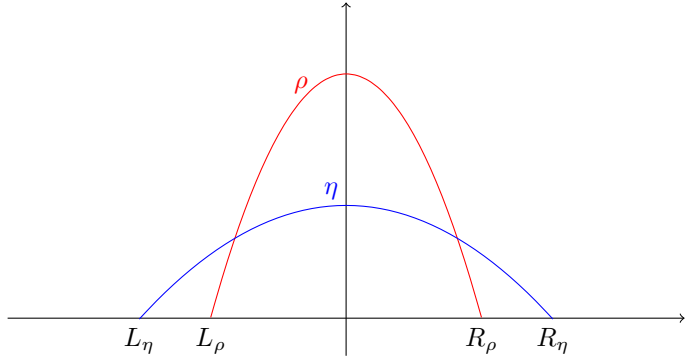

The other steady state, called \emph{separated steady state}, corresponds to the case where the prey win, namely the corresponding densities are supported on disjoint intervals and a vacuum region is formed around predators, see \figurename~\ref{fig:separated_state}.  

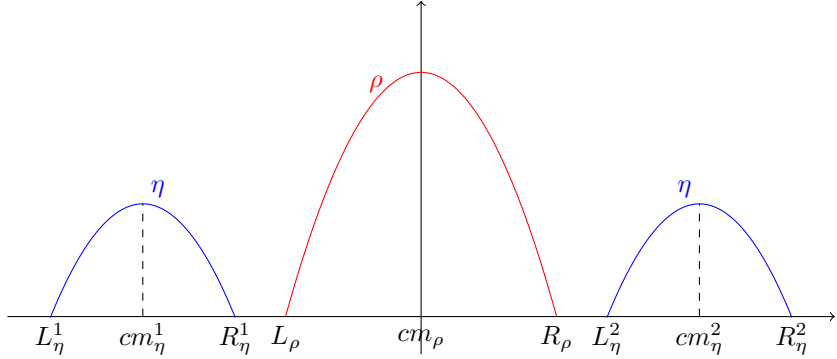
\begin{figure}[htbp]
\begin{tikzpicture}
\draw[->] (-4,0) -- (7,0);
\draw[->] (1.5,-0.5) -- (1.5,4.2) ;
\draw[scale=1,domain=-0.3:3.3,smooth,variable=\x,red] plot ({\x},{(3.25-(\x-1.5)*(\x-1.5)))});
\draw[scale=1,domain=-3.43:-0.97,smooth,variable=\x,blue] plot ({\x},{(1.5-(\x+2.2)*(\x+2.2)))});
\draw[scale=1,domain=3.97:6.43,smooth,variable=\x,blue] plot ({\x},{(1.5-(\x-5.2)*(\x-5.2)))});

\node[scale=1] at (-0.3,-0.3) {$L_\rho$};
\node[scale=1] at (1.5,-0.3) {$cm_\rho$};
\node[scale=1] at (3.3,-0.3) {$R_\rho$};
\node[scale=1] at (-3.43,-0.3) {$L_\eta^1$};
\node[scale=1] at (-2.2,-0.3) {$cm_\eta^1$};
\node[scale=1] at (-0.97,-0.3) {$R_\eta^1$};
\node[scale=1] at (3.97,-0.3) {$L_\eta^2$};
\node[scale=1] at (5.2,-0.3) {$cm_\eta^2$};
\node[scale=1] at (6.43,-0.3) {$R_\eta^2$};

\draw[dashed] (-2.2,0) -- (-2.2,1.5);
\draw[dashed] (5.2,0) -- (5.2,1.5);

\node[red,scale=1] at (1.8/2,3.1){$\rho$};
\node[blue,scale=1] at (-2,1.7){$\eta$};
\node[blue,scale=1] at (5,1.7){$\eta$};

\end{tikzpicture}
\caption{Example of separated stationary state. Assuming symmetry in both the densities we easily recover that $L_\rho=-R_\rho$ and $cm_\rho=0$, $L_\eta^1=-R_\eta^2$ and $cm_\eta^1=-cm_\eta^2$.}
\label{fig:separated_state}
\end{figure}

The proof of the existence of such steady states follows by using the strong version of the Krein-Rutman theorem, see \cite{Yu}.
\begin{thm}[Krein-Rutman]
\label{K-R}
Let $X$ be a Banach space, $K\subset X$ be a solid cone, such that
$\lambda K\subset K$ for all $\lambda\geq0$ and $K$ has a nonempty interior
$K^{o}$. Let $T$ be a compact linear operator on X, which is strongly positive with respect to $K$, i.e. $T[u]\in K^{o}$ if $u\in K\setminus\{0\}$. Then,
\begin{itemize}
\item [(i)] the spectral radius $r(T)$ is strictly positive and $r(T)$ is a simple eigenvalue with an eigenvector $v\in K^{o}$. There is no other eigenvalue with a corresponding eigenvector $v\in K$.
\item [(ii)] $|\lambda|<r(T)$ for all other eigenvalues $\lambda\neq r(T)$.
\end{itemize}
\end{thm}

As pointed out in \cite{BDFFS}, using this strategy we only obtain existence of stationary states for a diffusion coefficient that depends on the support, without providing an \emph{explicit} formula. Even if non completely satisfactory, the following results give a useful insight on the possible conditions we can expect in order to get existence of steady states, see Remark \ref{rem:conditions}.


\subsubsection{Mixed steady state} Let us first introduce a proper definition for mixed steady states as in \figurename~\ref{fig:mixed_state}.

\begin{defn}\label{def:mixed}
Let $0<R_\rho<R_\eta$ be fixed. We call a pair $(\rho,\eta)$ a \emph{mixed steady state} solution to system \eqref{main-equation} if $\rho$ and $\eta$ are $ L^1\cap L^{\infty}(\R)$, non-negative, symmetric and radially decreasing functions with supports
$$ 
I_\rho:= \mbox{supp}(\rho)=[-R_\rho,R_\rho],\qquad \mbox{and}\quad  I_\eta:= \mbox{supp}(\eta)=[-R_\eta,R_\eta],
$$
and $\rho'(0)=\eta'(0)=0.$
\end{defn}
Let us now assume that $(\rho,\eta)$ is a steady state to system \eqref{main-equation} as in Definition \ref{def:mixed}, then \eqref{stationary system1} can be rephrased as
\begin{equation}
\label{stationary system3}
\begin{dcases}
d \rho(x) - S_\rho*\rho(x) - K*\eta(x) = C_{\rho}  &\qquad x\in I_\rho
\\
d\eta(x) - S_\eta*\eta(x)+\alpha K*\rho(x) = C_{\eta}  &\qquad x\in I_\eta
\end{dcases}.
\end{equation}
where $C_{\rho},C_{\eta}>0$ are two constants. Differentiating the two equations in \eqref{stationary system3} w.r.t. $x\in\mbox{supp}(\rho)$ and $x\in\mbox{supp}(\eta)$ respectively, we obtain
\begin{equation}
\label{stationary system4}
\begin{dcases}
d\rho_{x} = \int_{-R_\rho}^{R_\rho}S_\rho(x-y)\rho_{y}(y)dy + \int_{-R_\eta}^{R_\eta}K(x-y)\eta_{y}(y)dy  &\quad\quad x\in I_\rho
\\[.5cm]
d\eta_{x} = \int_{-R_\eta}^{R_\eta}S_\eta(x-y)\eta_{y}(y)dy - \alpha\int_{-R_\rho}^{R_\rho}K(x-y)\rho_{y}(y)dy &\quad\quad x\in I_\eta
\end{dcases}.
\end{equation}
By symmetry properties of the kernels $S_\rho$, $S_\eta$ and $K$ and the steady states $\rho$ and $\eta$, for $x>0$, we get
\begin{equation}
\label{stationary system5}
\begin{aligned}
&d\rho_{x} = \int_{0}^{R_\rho}\Big(S_\rho(x-y) - S_\rho(x+y)\Big)\rho_{y}(y)dy +\int_{0}^{R_\eta}\Big(K(x-y) - K(x+y)\Big)\eta_{y}(y)dy,
\\
&d\eta_{x} = \int_{0}^{R_\eta}\Big(S_\eta(x-y) - S_\eta(x+y)\Big)\eta_{y}(y)dy - \alpha\int_{0}^{R_\rho}\Big(K(x-y) - K(x+y)\Big)\rho_{y}(y)dy.
\end{aligned}
\end{equation}
For simplifying notations, we set
\begin{equation*}
\widetilde{G}(x,y):=G(x-y)-G(x+y),\quad\mbox{for}\quad G=S_\rho,S_\eta,K.
\end{equation*}
Notice that $G$ is a nonnegative function for $x,y>0$. We also set $p(x)=-\rho_{x}(x)$ for $x\in (-R_\rho,R_\rho)$ and $q(x)=-\eta_{x}(x)$ for $x\in (-R_\eta,R_\eta)$. Hence, \eqref{stationary system5} is rewritten simply as

\begin{equation}
\label{stationary system6}
\begin{dcases}
d p(x) = \int_{0}^{R_\rho}\widetilde{S}_{\rho}(x,y)p(y)dy +  \int_{0}^{R_\eta}\widetilde{K}(x,y)q(y)dy
\\[0.2cm]
d q(x) = \int_{0}^{R_\eta}\widetilde{S}_{\eta}(x,y)q(y)dy - \alpha\int_{0}^{R_\rho}\widetilde{K}(x,y)p(y)dy
\end{dcases}.
\end{equation}

\begin{prop}
Assume that $S_\rho,S_\eta,K$ satisfy (A1), (A2) and (A3) and fix $0<R_\rho<R_\eta$. Assume that there exists a constant $C$ such that 
\begin{equation}\label{eq:cond_a_1}
C < \frac{\int_{0}^{R_\eta}\widetilde{S}_{\eta}(x,y)q(y)dy}{\int_{0}^{R_\rho}\widetilde{K}(x,y)p(y)dy}. 
\end{equation}
then, there exists a unique mixed steady state $(\rho,\eta)$ in the sense of Definition \ref{def:mixed} to system \eqref{main-equation} with $d = d(R_\rho,R_\eta)>0$, provided that
$$\alpha < \min\bigg\{ C\,,\frac{- S'_2(R_\eta)z_\eta}{- R_\eta^2 K''(0) z_\rho} \bigg\},$$
where $z_\rho$ and $z_\eta$ are masses of $\rho$ and $\eta$ respectively.
\end{prop}
\begin{proof}
Let us first introduce the following Banach space
\begin{equation*}
\Xre=\big\{(p,q)\in C^{1}[0,R_\rho]\times C^{1}[0,R_\eta]:p(0) = q(0)=0\big\},
\end{equation*}
endowed with the $W^{1,\infty}$-norm for the two components $p$ and $q$. Define the operator $T_{R_\rho,R_\eta}[P]$ on the Banach space $\Xre$ as
\begin{equation*}
T_{R_\rho,R_\eta}[P] := (f,g)\in C^{1}[0,R_\rho]\times C^{1}[0,R_\eta],
\end{equation*}
where $P$ denotes the elements $P=(p,q)\in \Xre$, and $(f,g)$ are given by
\begin{equation*}
\begin{aligned}
&f(x)= \int_{0}^{R_\rho}\widetilde{S}_{\rho}(x,y)p(y)dy +  \int_{0}^{R_\eta}\widetilde{K}(x,y)q(y)dy &\quad\mbox{for}\quad x\in[0,R_\rho],
\\
&g(x)=\int_{0}^{R_\eta}\widetilde{S}_{\eta}(x,y)q(y)dy - \alpha\int_{0}^{R_\rho}\widetilde{K}(x,y)p(y)dy &\quad\mbox{for}\quad x\in[0,R_\eta].
\end{aligned}
\end{equation*}
By assumptions on the kernels, it follows that the operator $T_{R_\rho,R_\eta}$ is compact on the Banach space $\Xre$. Now, consider the subset $\Kre\subseteq \Xre$ defined as
\begin{equation*}
\Kre=\big\{P\in \Xre:p\geq0,q\geq0\big\}.
\end{equation*}
It can be shown that this set is indeed a solid cone in $\Kre$. Moreover, we have that
\begin{equation*}
\begin{aligned}
\Hre= \big\{P\in \Kre:p^{\prime}(0)>0,~&p
(x)>0\hspace{0.2cm}\mbox{for all}\hspace{0.2cm}x\in (0,R_\rho),\hspace{0.2cm}\mbox{and} 
\\
&q^{\prime}(0)>0,~q(x)>0\hspace{0.2cm}\mbox{for all}\hspace{0.2cm}x\in (0,R_\eta) \big\}\subset \overset{\circ}{\Kre},
\end{aligned}
\end{equation*}
where $\overset{\circ}{\Kre}$ denotes the interior of $\Kre$. Next, we show that the operator $T_{R_\rho,R_\eta}$ defined above is strongly positive on the solid cone $\Kre$ in the sense of Theorem \ref{K-R}. Let $(p,q)\in \Kre$ with $p,q\neq 0$, then by the definition of the operator $T_{R_\rho,R_\eta}$, the first component is non-negative. Concerning the second component, we have
\begin{equation}
\int_{0}^{R_\eta}\widetilde{S}_{\eta}(x,y)q(y)dy - \alpha\int_{0}^{R_\rho}\widetilde{K}(x,y)p(y)dy > 0
\end{equation}
if and only if $\alpha <C$ with $C$ as in \eqref{eq:cond_a_1}. Next, it is easy to show that the derivative at $x=0$ of the first component is strictly positive. The derivative of the second component is given by 
\begin{equation*}
\begin{aligned}
&\frac{d}{dx}\Big|_{x=0}\Bigg( \int_{0}^{R_\eta}\widetilde{S}_{\eta}(x,y)q(y)dy - \alpha\int_{0}^{R_\rho}\widetilde{K}(x,y)p(y)dy \Bigg) 
\\
& =\int_{0}^{R_\eta}\widetilde{S}_{\eta,x}(0,y)q(y)dy - \alpha\int_{0}^{R_\rho}\widetilde{K}_{x}(0,y)p(y)dy
\\
& =\int_{0}^{R_\eta}\big(S_\eta^{\prime}(-y) - S_\eta^{\prime}(y)\big)q(y)dy - \alpha\int_{0}^{R_\rho}\big(K^{\prime}(-y) - K^{\prime}(y)\big)p(y)dy
\\
& =-2\int_{0}^{R_\eta}S_\eta^{\prime}(y)q(y)dy + 2\alpha\int_{0}^{R_\rho}K^{\prime}(y)p(y)dy:=A.
\end{aligned}
\end{equation*} 
Now, we need to find the condition on $\alpha$ such that $A>0$. Chebyshev's inequality in the first integral of $A$ yields the bound
\begin{equation*}
\begin{aligned}
-\frac{2}{R_\eta}\int_{0}^{R_\eta}S_\eta^{\prime}(y)q(y)dy &= -\frac{2}{R_\eta}\int_{0}^{R_\eta}S''_\eta(y)\eta(y)dy 
\\
&\geq \left( \frac{1}{R_\eta}\int_{0}^{R_\eta}-S''_2(y) dy \right) \left( \frac{2}{R_\eta}\int_{0}^{R_\eta} \eta(y)dy \right) = \frac{-S'_2(R_\eta) z_\eta}{R_\eta^2}.
\end{aligned}
\end{equation*}
The other integral can be easily bounded by
\begin{equation*}
-2\int_{0}^{R_\rho}K'(y)p(y)dy = -2\int_{0}^{R_\rho}K''(y)\rho(y)dy < -K''(0) z_\rho.
\end{equation*}
Thus, $A>0$ holds under the condition
\begin{equation}\label{1st alpha condition}
\alpha < \frac{- S'_2(R_\eta)z_\eta}{- R_\eta^2 K''(0) z_\rho}.
\end{equation}
As a consequence, the operator $T_{R_\rho,R_\eta}[P]$ belongs to $H_{R_\rho,R_\eta}$, which implies that the operator $T_{R_\rho,R_\eta}$ is strongly positive on the solid cone $\Kre$. Then, the Krein-Rutman theorem applies and guarantees the existence of an eigenvalue $d=d(R_\rho,R_\eta)$ such that
\begin{equation*}
T_{R_\rho,R_\eta}[P] = dP,
\end{equation*}
with an eigenspace generated by a nontrivial element $(p,q)$ in the interior of the solid cone $\Kre$.
\end{proof}

\subsubsection{Separated stationary states}
We now introduce the following definition
\begin{defn}\label{def3}
Fix $0 < R_\rho < L_\eta < R_\eta$. Set $\cmei = \frac{R_\eta+L_\eta}{2}$. A pair $(\rho,\eta)$ is called \emph{separated steady state} to system \eqref{main-equation} if $\rho$ and $\eta$ have supports
$$ 
I_\rho:= \mbox{supp}(\rho)=[-R_\rho,R_\rho],\qquad \mbox{and}\quad  I_\eta:= \mbox{supp}(\eta)=[-R_\eta,-L_\eta]\cup[L_\eta,R_\eta],
$$
respectively, $\rho,\eta \in L^1\cap L^{\infty}(\R)$, both are non-negative, $\rho$ is symmetric around zero and monotone decreasing on $[0,R_\rho]$, $\eta$ symmetric on both parts of its support, monotone decreasing on $[-\cmei,-L_\eta]\cup[\cmei,R_\eta]$ , and $\rho'(0)=\eta'(\cmei)= \eta'(-\cmei) =0$.
\end{defn}  

Assume that $\rho$ and $\eta$ are solutions to \eqref{stationary system1} with a structure as described in Definition \ref{def3}, then \eqref{stationary system1} can be rephrased as
\begin{equation}
\label{3rd stationary system222}
\begin{dcases}
d\rho - S_\rho*\rho-K*\eta = C_{1} &\quad \mbox{for}\quad x\in I_\rho
\\
d\eta - S_\eta*\eta+\alpha K*\rho = C_{2} &\quad \mbox{for}\quad x\in I_\eta
\end{dcases}
\end{equation}
where $C_1,C_2>0$ are two constants. A similar procedure to the one before leads to the following
\begin{equation}
\label{3rd stationary system6}
\begin{dcases}
d p(x) = \int_{0}^{R_\rho}\hat{S}_{\rho}(x,y)p(y)dy +  \int_{\cmei}^{R_\eta}\tilde{K}(x,y)q(y)dy  \quad x\in[0,R_\rho]
\\[0.2cm]
d q(x) = \int_{\cmei}^{R_\eta}\tilde{S}_{\eta}(x,y)q(y)dy - \alpha\int_{0}^{R_\eta}\hat{K}(x,y)p(y)dy  \quad x\in[\cmei,R_\eta]
\end{dcases},
\end{equation}
where
\begin{equation}
\begin{aligned}
&\hat{S}_{\rho}(x,y) = S_\rho(x-y)-S_\rho(x+y),
\\
&\hat{K}(x,y) = K(x-y)-K(x+y),
\\
&\widetilde{K}(x,y) = K(x-y)-K(x+y) - K(x+y-2\cmei) + K(x-y+2\cmei),
\\
&\widetilde{S}_{\eta}(x,y) = S_\eta(x-y)-S_\eta(x+y) - S_\eta(x+y-2\cmei) + S_\eta(x-y+2\cmei),
\end{aligned}
\end{equation}
and $p(x)=-\rho^{\prime}(x)$ restricted to the interior of $I_\rho$ and $q(x)=-\eta^{\prime}(x)$ restricted to the interior of $I_\eta$.
Let us now introduce the Banach space
\begin{equation*}
\Xrre = \big\{(p,q)\in C^{1}[0,R_\rho]\times C^{1}[\cmei,R_\eta] : p(0) = q(\cmei) = 0\big\},
\end{equation*}
endowed with the $W^{1,\infty}$-norm for the two components $p$ and $q$. Define the operator $T_{R_\rho,L_\eta,R_\eta}[P]$  as
\begin{equation}\label{operator2}
\Xrre\ni P \longmapsto  T_{R_\rho,L_\eta,R_\eta}[P] := (f,g)\in C^{1}[0,R_\rho]\times C^{1}[\cmei,R_\eta],
\end{equation}
with 
\begin{align}
&f(x) = \int_{0}^{R_\eta}\hat{S}_{\rho}(x,y)p(y)dy +  \int_{\cmei}^{R_\eta}\tilde{K}(x,y)q(y)dy:=A_1(x)+A_2(x)  \quad x\in[0,R_\rho],
 \label{component_1} \\
&g(x) = \int_{\cmei}^{R_\eta}\tilde{S}_{\eta}(x,y)q(y)dy - \alpha\int_{0}^{R_\rho}\hat{K}(x,y)p(y)dy:=A_3(x)-\alpha A_4(x)  \quad x\in[\cmei,R_\eta]. \label{component_2}
\end{align}
Define $\Krre\subseteq \Xrre$ as
\begin{equation*}
\Krre=\big\{P\in \Xrre:p\geq0,q\geq0\big\}.
\end{equation*}
It can be shown that this set is a solid cone in $\Xrre$. Moreover,
\begin{equation*}
\begin{aligned}
\Hrre= \big\{P\in \Krre:p'(0)>0,~&p(x)>0\hspace{0.2cm}\mbox{for all}\hspace{0.2cm}x\in (0,R_\rho],\hspace{0.2cm}\mbox{and} 
\\
&q'(\cmei)>0,~q(x)>0\hspace{0.2cm}\mbox{for all}\hspace{0.2cm}x\in (\cmei,R_\eta] \big\}\subset \overset{\circ}{\Krre},
\end{aligned}
\end{equation*}
where $\overset{\circ}{\Krre}$ denotes the interior of $\Krre$.
\begin{prop}\label{prop:sss}
Let $0 < R_\rho < L_\eta < R_\eta$ be fixed and set $\cmei = \frac{R_\eta+L_\eta}{2}$. Assume that $S_\rho,S_\eta,K$ satisfy (A1), (A2) and (A3) and moreover that $K,~S_\eta$ are strictly concave on $\big[-(R_\rho+R_\eta),R_\rho+R_\eta\big],~\big[-(\cmei+R_\eta),\cmei+R_\eta\big]$ respectively. Assume that there exists a constant $c>0$ such that 
\begin{equation*}
\begin{aligned}
c<\frac{A_3(x)}{A_4(x)} \qquad\mbox{for all}~x\in[\cmei,R_\eta],
\end{aligned}
\end{equation*}
with $A_3$ and $A_4$ defined in \eqref{component_2} and consider $\alpha$ such that
\begin{equation}\label{alfa_sep}
\alpha < \min\left(c,\frac{ - \big( 2S''_\eta(\cmei-R_\eta) + S''_\eta(\cmei+R_\eta) + S''_\eta(2\cmei) \big) z_\eta}{-4K''(0)z_\rho }\right).
\end{equation}
Then, there exists a unique separated steady state $(\rho,\eta)$ to \eqref{main-equation} in the sense of Definition \ref{def3} with $d=d(R_\rho,L_\eta,R_\eta)>0$.
\end{prop}
\begin{proof}
We study each integral in \eqref{component_1} and \eqref{component_2} separately. $A_1(x)>0$ by the assumptions that $S_\rho$ is decreasing on $[0,\infty)$.
\begin{figure}
\begin{tikzpicture}
\draw[->] (-3,0) -- (6,0) node[right] {$x$};
\draw[->] (0,-0.5) -- (0,4.2) node[above] {$K(x)$};
\draw[scale=0.5,domain=-6:12,smooth,variable=\x,blue] plot ({\x},{10/sqrt(pi)*exp(-0.01*\x*\x)});
\draw[] (7.07/2,-0.3) -- (7.07/2,2); 

\draw[dashed] (6/2,0) -- (6/2,3.93622/2);
\draw[dashed] (5.2/2,0) -- (5.2/2,4.30519/2);
\draw[dashed] (1.8/2,0) -- (1.8/2,5.46203/2);
\draw[dashed] (2.6/2,0) -- (2.6/2,5.27311/2);

\draw[dashed] (0,3.93622/2) -- (6/2,3.93622/2);
\draw[dashed] (0,4.30519/2) -- (5.2/2,4.30519/2);
\draw[dashed] (0,5.46203/2) -- (1.8/2,5.46203/2);
\draw[dashed] (0,5.27311/2) -- (2.6/2,5.27311/2);

\node[scale=0.8] at (5,2) {concavity point};
\node[scale=3] at (7.07/2,3.422/2){.};

\node at (0.8,-0.3) {$|a_1|$};
\node[green,scale=3] at (1.8/2,0){.};

\node at (1.4,-0.3) {$|a_2|$};
\node[green,scale=3] at (2.6/2,0){.};

\node at (2.6,-0.3) {$a_3$};
\node[green,scale=3] at (5.2/2,0){.};

\node at (3,-0.3) {$a_4$};
\node[green,scale=3] at (6/2,0){.};

\node[scale=0.8] at (-0.4,2.65){$\Delta_1 y$};
\node[scale=0.8] at (-0.4,2.12){$\Delta_2 y$};

\draw[domain=-0.5:3,smooth,variable=\x,red] plot ({\x},{((5.27311/2-5.46203/2)/(2.6/2-1.8/2))*(\x-1.8/2) + 5.46203/2});
\draw[domain=1:5,smooth,variable=\x,red] plot ({\x},{((3.93622/2-4.30519/2)/(6/2-5.2/2))*(\x-6/2) + 3.93622/2});

\node[scale=2.5] at (1.8/2,5.46203/2){.};
\node[scale=2.5] at (2.6/2,5.27311/2){.};
\node[scale=2.5] at (5.2/2,4.30519/2){.};
\node[scale=2.5] at (6/2,3.93622/2){.};

\end{tikzpicture}
\caption{For a generic kernel $K$ and $a_1,$ $a_2$,  $a_3$ and $a_4$ defined in \eqref{eq:a}, the picture shows that $\Delta_2 y > \Delta_1 y$ and so the slope between the two point $(a_3,K(a_3)),(a_4,K(a_4))$ is less than the slope between the two points $(|a_1|,K(|a_1|)),(|a_2|,K(|a_2|))$.}
\label{fig1}
\end{figure}
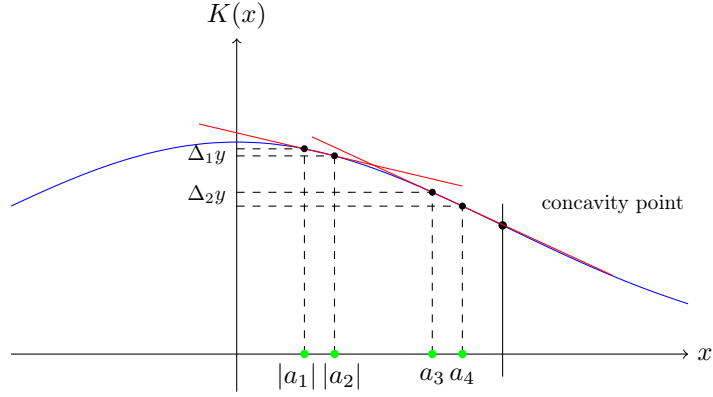
Concerning the sign of $A_2$(x) we set 
\begin{equation}\label{eq:a}
    a_1 = x+y-2\cmei,\quad a_2 = x-y,\quad a_3 = x-y+2\cmei,\quad a_4 = x+y,
\end{equation}
see \figurename~\ref{fig1}.  The following order holds for all $(x,y)\in[0,R_\rho]\times[\cmei,R_\eta]$, 
\begin{equation}\label{order}
|a_1| < |a_2|,\qquad  |a_1| < a_3,\qquad |a_1| < a_4,\qquad  |a_2| < a_4,\qquad a_3 < a_4. 
\end{equation}
Using these notations, $A_2$ becomes 
\begin{equation*}
A_2 (x)= \int_{\cmei}^{R_\eta}\big[ K(a_2)-K(a_4) - K(a_1) + K(a_3) \big]q(y)dy.
\end{equation*}
It is easy to check that $a_4 - a_3 = a_1 - a_2$, then we have
\begin{equation}\label{3rd A2}
\begin{aligned}
K(a_3)-K(a_4) - K(a_1) + K(a_2) &= -(a_4-a_3)\left( \frac{K(a_4)-K(a_3) }{a_4-a_3} + \frac{K(a_1) - K(a_2)}{a_4-a_3} \right)
\\
&= -(a_4-a_3)\left( \frac{K(a_4)-K(a_3) }{a_4-a_3} + \frac{K(a_1) - K(a_2)}{a_1-a_2}\right)>0,
\end{aligned}
\end{equation}
where the last inequality holds thanks to the concavity assumption on $K$ and \eqref{order}. This implies that $A_2(x)>0$ and the first component of $T_{R_\rho,R_\eta,R_\eta}$, $f(x)>0$ for all $x\in(0,R_\rho)$. The same can be done to prove that $A_3(x)>0$. On the other hand, since the function $K(x-y)-K(x+y)$ is non-negative for $x,y>0$ we get that $-\alpha A_4(x)<0$. Therefore, the second component $g(x)$ can be positive under the condition  $\alpha < c.$
The derivatives of the first component computed at $x=0$ and the second computed at $x=\cmei$, are given by
\begin{equation*}
\begin{aligned}
\frac{d}{dx}f(x)\big|_{x=0} = & \int_{0}^{R_\rho}-2S'_\rho(y)p(y)dy  +  \int_{\cmei}^{R_\eta}-2\big[ K'(y) + K'(y-2\cmei) \big]q(y)dy,\\
\frac{d}{dx}g(x)\big|_{x=\cmei}  =& \int_{\cmei}^{R_\eta}\big[ 2S'_\eta(\cmei-y)-S'_\eta(\cmei+y) + S'_\eta(3\cmei-y) \big]q(y)dy  
\\
& - \alpha\int_{0}^{R_\rho}\big[ K'(\cmei-y) - K'(\cmei+y) \big] p(y)dy.
\end{aligned}
\end{equation*}
By the assumptions on the kernels $K,\,S_\rho$, we have
\begin{equation*}
 \int_{0}^{R_\rho}-2S'_\rho(y)p(y)dy > 0, \qquad   \int_{0}^{R_\rho}\big[ K'(\cmei-y) - K'(\cmei+y) \big] p(y)dy > 0.
\end{equation*}
Under the assumptions on the concavity of the kernels $K,S_\eta$, it is easy to show
\begin{equation*}
\begin{aligned}
& \int_{\cmei}^{R_\eta}-2\big[ K'(y) + K'(y-2\cmei) \big]q(y)dy > 0,
\\
&\int_{\cmei}^{R_\eta}\big[ 2S'_\eta(\cmei-y)-S'_\eta(\cmei+y) + S'_\eta(3\cmei-y) \big]q(y)dy  > 0.
\end{aligned}
\end{equation*}
By concavity assumption on $S_\eta$ i.e. $S''_\eta(x)<0$ on $[-2\cmei,2\cmei]$, we have the bound
\begin{equation*}
\begin{aligned}
 &\int_{\cmei}^{R_\eta}\big[ 2S'_\eta(\cmei-y)-S'_\eta(\cmei+y) + S'_\eta(3\cmei-y) \big]q(y)dy  
\\ 
 &=  -\int_{\cmei}^{R_\eta}\big[ 2S''_\eta(\cmei-y) + S''_2(\cmei+y) + S''_\eta(3\cmei-y) \big]\eta(y)dy 
\\ 
&\geq -\frac{z_\eta}{4} (2S''_\eta(\cmei-R_\eta) + S''_\eta(\cmei+R_\eta) + S''_\eta(2\cmei)) > 0.
\end{aligned}
\end{equation*}
On the other hand, 
\begin{equation*}
\begin{aligned}
 \int_{0}^{R_\rho}\left( K'(\cmei-y) - K'(\cmei+y) \right) p(y)dy =  -\int_{0}^{R_\rho}\left(K''(\cmei-y) + K''(\cmei+y) \right) \rho(y)dy \leq -K''(0)z_\rho .
\end{aligned}
\end{equation*}
Hence, we get $\frac{d}{dx}g(x)\big|_{x=0}>0$ under the condition in \eqref{alfa_sep}. Then the operator $T$ defined above is strongly positive on the solid cone $\Krre$ and Krein-Rutman theorem applies.
\end{proof}

\section{Existence for Multiple Bumps Steady States}\label{sec:implicit}
In this Section we prove existence and uniqueness of a multiple bumps steady state in the sense of Definition \ref{def:multi_bump} fixing masses and a small diffusion coefficient. Following the approach in \cite{budif,BDFFS}, we first formulate the problem in terms of the pseudo-inverse
functions and then we use the Implicit Function Theorem (cf. \cite[Theorem 15.1]{Del}). 

We start rewriting our stationary system in terms of pseudo-inverse functions. Let $(\rho,\eta)$ be a solution to the stationary system  
\begin{equation}\label{stationary system sec2}
\begin{dcases}
0 = \big( \rho \big( d\rho - S_\rho*\rho - \alpha_{\rho} K*\eta \big)_{x}\big)_{x}
\\
0 = \big( \eta \big( d\eta - S_\eta*\eta - \alpha_{\eta} K*\rho \big)_{x}\big)_{x}.
\end{dcases}
\end{equation}
where $\alpha_{\rho} = 1$ and $\alpha_{\eta} = -\alpha$. Assume that $(\rho,\eta)$ have masses $z_\rho$ and $z_\eta$ respectively and denote by $cm_{l},~l = \{\rho,\eta \},$ the centres of masses
\[\int_\R x\rho(x)dx = cm_{\rho}, \qquad \int_\R x\eta(x)dx = cm_{\eta}. \]
Remember that the only conserved quantity in the evolution, together with the masses, is the \emph{joint centre of mass}
\begin{equation}\label{joint}
    CM_\alpha = \alpha cm_\rho - cm_\eta,
\end{equation}
that we can consider fixed. Define the cumulative distribution of $\rho$ and $\eta$ as
\[ F_{\rho}(x) = \int_{-\infty}^{x}\rho(x)dx, \qquad F_{\eta}(x) = \int_{-\infty}^{x}\eta(x)dx.  \] 
Let $u_{l}:[0,z_l]\rightarrow\R,~l \in \{\rho,\eta\},$ be the pseudo-inverse of $F_l,$ namely
\[u_{l}(z) = \inf \{ x\in \R : F_{l}\geq z \}, \quad l \in \{\rho,\eta\},\]
supported on
\[ \mbox{supp}(u_{l}) = [0,z_{l}] := J_{l},\qquad l \in \{\rho,\eta\}.\]
For $\rho$ and $\eta$ multiple bumps in the sense of Definition \ref{def:multi_bump} we can denote the mass of each bump as 
$$
\int\rho_{i}(x)dx = z_{\rho}^i, \quad \int\eta_{j}(x)dx = z_{\eta}^j, \quad i = 1,2,\ldots, N_\rho. \quad j = 1,2,\ldots, N_\eta,
$$
and the centres of masses accordingly,
$$
\int x\rho_{i}(x)dx = cm_{\rho}^i, \quad \int x\eta_{j}(x)dx = cm_{\eta}^j, \quad \quad i = 1,2,\ldots, N_\rho. \quad j = 1,2,\ldots, N_\eta.$$
and we can always assume that the centres of masses are ordered species by species, i.e. $cm_l^i\geq cm_l^j$ if $i\geq j$. Let us consider the case of  centres of masses that are stationary solutions of the purely non-local particle system \eqref{eq:pred_particle_steady}, that we recall for the reader convenience,
\begin{equation}\label{eq:cmi}
    \begin{cases}
    \displaystyle
     \sum_{j=1}^{N_\rho} S'_\rho( cm_\rho^i - cm_\rho^j )z_\rho^j +  \sum_{j=1}^{N_\eta} K'( cm_\rho^i - cm_\eta^j ) z_\eta^j = 0, \quad i=1,\ldots,N_\rho,\\
     \displaystyle
      \sum_{j=1}^{N_\eta} S'_\eta( cm_\eta^i - cm_\eta^j ) z_\eta^j - \alpha\sum_{j=1}^{N_\rho} K'( cm_\eta^i - cm_\rho^j ) z_\rho^j = 0, \quad i=1,\ldots,N_\eta,
    \end{cases}
\end{equation}
coupled with the conservation of the joint centre of mass $CM_\alpha$ in \eqref{joint}, see the discussion in Section \ref{sec:preliminaries}.
Then the pseudo-inverse $u_l$ reads as
$$
u_l(z) = \sum_{i=1}^{N_l} u_l^i(z) \mathds{1}_{J_l^i}(z), \quad l \in \{\rho,\eta\},
$$
where
$$
\mbox{supp}(u_{l}) = [0,z_{l}] = J_{l} = \bigcup_{i=1}^{N_l}\left[\sum_{k=1}^{i} z_l^{k-1}, \sum_{k=1}^{i} z_l^k \right]:=\bigcup_{i=1}^{N_l}\left[\hat{z}_l^{i},  \tilde{z}_l^i \right]:=\bigcup_{i=1}^{N_l}J_l^i,\qquad l \in \{\rho,\eta\},
$$
with $z_l^0 = 0$ and $ z_l = \sum_{k=1}^{N_l} z_l^k.$ We are now in the position of reformulating \eqref{stationary system sec2} in terms of the pseudo-inverse functions as follows:
\begin{equation}
\label{Pseudo-inverse1}
\begin{dcases}
\frac{d}{2}\partial_{z}\Big(\big(\partial_{z}u_{\rho}(z)\big)^{-2}\Big) = \int_{J_{\rho}}S_{\rho}^{\prime}\big(u_{\rho}(z)-u_{\rho}(\xi)\big)d\xi + \alpha_{\rho}\int_{J_{\eta}}K^{\prime}\big(u_{\rho}(z)-u_{\eta}(\xi)\big)d\xi,~~z\in J_{\rho},
\\
\frac{d}{2}\partial_{z}\Big(\big(\partial_{z}u_{\eta}(z)\big)^{-2}\Big) = \int_{J_{\eta}}S_{\eta}^{\prime}\big(u_{\eta}(z)-u_{\eta}(\xi)\big)d\xi + \alpha_{\eta}\int_{J_{\rho}}K^{\prime}\big(u_{\eta}(z)-u_{\rho}(\xi)\big)d\xi,~~z\in J_{\eta}.
\end{dcases}
\end{equation}
The restriction to $z\in J_{l}^i$, $i=1,2,\cdots,N_l,$ and $l \in \{\rho,\eta\},$ , allow us to rephrase \eqref{Pseudo-inverse1} in the compact form
\begin{equation}\label{compact}
\begin{aligned}
\frac{d}{2}\partial_{z}\Big(\big(\partial_{z}u_{l}^i(z)\big)^{-2}\Big) & = \sum_{j=1}^{N_l} \int_{J_{l}^j}S_{l}^{\prime}\big(u_{l}^i(z) -  u_{l}^j(\xi)\big)d\xi + \alpha_{l}\sum_{j=1}^{N_h}\int_{J_{h}^j}K^{\prime}\big(u_{l}^i(z) -  u_{h}^j(\xi)\big)d\xi, \quad z\in J_l^i.
\end{aligned}
\end{equation}
Similar to \cite{budif, BDFFS}, we suggest the linearization  
$$
u_l^i = cm_l^i  + \delta v_l^i \quad i=1,2,\cdots,N_l, \mbox{ and } l \in \{\rho,\eta\},
$$ 
with $v_l^i$, being odd functions defined on $J_l^i$. Using this ansatz in  \eqref{compact}, with the scaling $d =\delta^{3}$ we have 
\begin{equation}
\label{Pseudo-inverse2}
\begin{aligned}
\frac{\delta}{2}\partial_{z}\Big(\big(\partial_{z}v_{l}^i(z)\big)^{-2}\Big)  = & \sum_{j=1}^{N_l}\int_{J_{l}^j}S_{l}^{\prime}\Big(cm_l^i - cm_l^j + \delta\big( v_{l}^i(z)-v_{l}^j(\xi)\big) \Big)d\xi 
\\
& + \alpha_{l}\sum_{j=1}^{N_h}\int_{J_{h}^j}K^{\prime}\Big(cm_l^i - cm_h^j + \delta\big( v_{l}^i(z)-v_{h}^j(\xi) \big)\Big)d\xi.
\end{aligned}
\end{equation}
Multiplying \eqref{Pseudo-inverse2} by $\delta\partial_{z}v_l^i$, and taking the primitives w.r.t. $z$, we obtain
\begin{equation}
\label{Pseudo-inverse3}
\begin{aligned}
\frac{\delta^2}{\partial_{z}v_{l}^i(z)}  = & \sum_{j=1}^{N_l}\int_{J_{l}^j}S_{l}\Big(cm_l^i - cm_l^j + \delta\big( v_{l}^i(z)-v_{l}^j(\xi)\big) \Big)d\xi 
\\
& + \alpha_{l}\sum_{j=1}^{N_h}\int_{J_{h}^j}K\Big(cm_l^i - cm_h^j + \delta\big( v_{l}^i(z)-v_{h}^j(\xi) \big)\Big)d\xi  + A_l^i,\quad z\in J_l^i,
\end{aligned}
\end{equation}
where $A_l^i$ are the integration constants. In order to recover the constants $A_l^i$, we substitute $\tilde{z}_l^i$ into equation \eqref{Pseudo-inverse3}. Denoting by $v_l^i(\tilde{z}_l^i) = \lambda_l^i,$ we obtain
\begin{equation}
\label{A constants}
\begin{aligned}
A_l^i = & -\sum_{j=1}^{N_l}\int_{J_{l}^j}S_{l}\Big(cm_l^i - cm_l^j + \delta\big( \lambda_l^i - v_{l}^j(\xi)\big) \Big)d\xi 
 - \alpha_{l}\sum_{j=1}^{N_h}\int_{J_{h}^j}K\Big(cm_l^i - cm_h^j + \delta\big( \lambda_l^i - v_{h}^j(\xi) \big)\Big)d\xi.
\end{aligned}
\end{equation}
 Next, we set $G_{l}$ and $H$ such that $G_l^{\prime} = S_l$ and $H^{\prime} = K$, with $G_l,~H$ to be odd and satisfy $G_l(0) = H(0) = 0$.  Then, multiplying \eqref{Pseudo-inverse3} again by $\delta\partial_{z}v_l^i$ and taking the primitives w.r.t. $z\in J_l^i$, we obtain
\begin{equation}
\label{Pseudo-inverse4}
\begin{aligned}
\delta^3 z  = & \sum_{j=1}^{N_l}\int_{J_{l}^j}G_{l}\Big(cm_l^i - cm_l^j + \delta\big( v_{l}^i(z)-v_{l}^j(\xi)\big) \Big)d\xi 
\\
& + \alpha_{l}\sum_{j=1}^{N_h}\int_{J_{h}^j}H\Big(cm_l^i - cm_h^j + \delta\big( v_{l}^i(z)-v_{h}^j(\xi) \big)\Big)d\xi  + A_l^i\delta v_l^i(z) + \beta_l^i,\quad z\in J_l^i.
\end{aligned}
\end{equation}
Let us denote with $\bar{z}_l^i$ the middle point of each interval $J_l^i.$ Then, in order to recover the integration constants $\beta_l^i$, we substitute $\bar{z}_l^i$ in \eqref{Pseudo-inverse4} which yields
\begin{equation}
\label{Pseudo-inverse5}
\begin{aligned}
\beta_l^i =\, \delta^3 \bar{z}_l^i & - \sum_{j=1}^{N_l}\int_{J_{l}^j}G_{l}\big(cm_l^i - cm_l^j - \delta v_{l}^j(\xi) \big)d\xi 
\\
& \hspace{2cm}- \alpha_{l}\sum_{j=1}^{N_h}\int_{J_{h}^j}H\big(cm_l^i - cm_h^j - \delta v_{h}^j(\xi)\big)d\xi.
\end{aligned}
\end{equation}
As a consequence of all above computations, we get the following relation for $z\in J_l^i.$
\begin{equation}
\label{Pseudo-inverse6}
\begin{aligned}
\delta^3 (z - \bar{z}_l^i)  = & \,\sum_{j=1}^{N_l}\int_{J_{l}^j}G_{l}\Big(cm_l^i - cm_l^j + \delta\big( v_{l}^i(z)-v_{l}^j(\xi)\big) \Big) - G_l\big(cm_l^i - cm_l^j - \delta v_{l}^j(\xi) \big) \,d\xi 
\\
& -\delta v_l^i(z) \sum_{j=1}^{N_l}\int_{J_{l}^j}S_{l}\Big(cm_l^i - cm_l^j + \delta\big( \lambda_l^i - v_{l}^j(\xi)\big) \Big) \,d\xi 
\\
& + \alpha_{l}\sum_{j=1}^{N_h}\int_{J_{h}^j}H\Big(cm_l^i - cm_h^j + \delta\big( v_{l}^i(z)-v_{h}^j(\xi) \big)\Big) - H\big(cm_l^i - cm_h^j - \delta v_{h}^j(\xi)\big) \,d\xi
\\
& - \delta v_l^i(z) \alpha_{l}\sum_{j=1}^{N_h}\int_{J_{h}^j}K\Big(cm_l^i - cm_h^j + \delta\big( \lambda_l^i - v_{h}^j(\xi) \big)\Big) \,d\xi .
\end{aligned}
\end{equation}
If we define, for  $p = (v_{\rho}^1,\ldots,v_{\rho}^{N_\rho},v_{\eta}^1,\ldots,v_{\eta}^{N_\eta})$
\begin{equation}
\label{functional equation}
\begin{aligned}
& \mathcal{F}_l^i[p;\delta](z) 
\\
&= \,  \bar{z}_l^i - z  +  \, \delta^{-3}\Bigg[ \sum_{j=1}^{N_l}\int_{J_{l}^j}G_{l}\Big(cm_l^i - cm_l^j + \delta\big( v_{l}^i(z)-v_{l}^j(\xi)\big) \Big) - G_l\big(cm_l^i - cm_l^j - \delta v_{l}^j(\xi) \big) \,d\xi 
\\
& -\delta v_l^i(z) \sum_{j=1}^{N_l}\int_{J_{l}^j}S_{l}\Big(cm_l^i - cm_l^j + \delta\big( \lambda_l^i - v_{l}^j(\xi)\big) \Big) \,d\xi 
\\
& + \alpha_{l}\sum_{j=1}^{N_h}\int_{J_{h}^j}H\Big(cm_l^i - cm_h^j + \delta\big( v_{l}^i(z)-v_{h}^j(\xi) \big)\Big) - H\big(cm_l^i - cm_h^j - \delta v_{h}^j(\xi)\big) \,d\xi
\\
& - \delta v_l^i(z) \alpha_{l}\sum_{j=1}^{N_h}\int_{J_{h}^j}K\Big(cm_l^i - cm_h^j + \delta\big( \lambda_l^i - v_{h}^j(\xi) \big)\Big) \,d\xi \Bigg], \quad z\in J_l^i.
\end{aligned}
\end{equation}
we have that \eqref{compact} reduces to the equation $\mathcal{F}_l^i[p;\delta](z) = 0$. In the following we compute the Taylor expansion of $\mathcal{F}_l^i[p;\delta](z)$ around $\delta = 0$. Let us  begin with the first integral on the r.h.s. of \eqref{functional equation}, i.e.,
\begin{equation}
\label{Taylor1}
\begin{aligned}
& \int_{J_{l}^j}G_{l}\Big(cm_l^i - cm_l^j + \delta\big( v_{l}^i(z)-v_{l}^j(\xi)\big) \Big) -  G_l\big(cm_l^i - cm_l^j - \delta v_{l}^j(\xi) \big) \,d\xi
\\
& = \Big[ S_l(cm_l^i - cm_l^j) \delta v_{l}^i(z) +  \frac{\delta^2}{2} S'_l(cm_l^i - cm_l^j)(v_{l}^i(z))^2  +   \frac{\delta^3}{6} S''_l(cm_l^i - cm_l^j)(v_{l}^i(z))^3 \Big] |J_l^j|
\\
& \hspace{5.5cm} + \int_{J_{l}^j}  \frac{\delta^3}{2} S''_l(cm_l^i - cm_l^j)\big( v_{l}^j(\xi) \big)^2 v_l^i(z) \,d\xi +R(S'''_l,\delta^4),
\end{aligned}
\end{equation}
where we used the fact that $\int_{J_l^i}v_l^i(\xi)\,d\xi = 0$ and $R(S'''_l,\delta^4)$ is a remainder term. For the second integral we have
\begin{equation}
\label{Taylor2}
\begin{aligned}
& - \delta v_l^i(z) \int_{J_{l}^j}S_{l}\Big(cm_l^i - cm_l^j + \delta\big( \lambda_{l}^i - v_{l}^j(\xi)\big) \Big) \,d\xi
\\
& = \Big[ - S_l(cm_l^i - cm_l^j) \delta v_{l}^i(z) - \delta^2 S'_l(cm_l^i - cm_l^j) \lambda_l^i v_{l}^i(z) - \frac{\delta^3}{2} S''_l(cm_l^i - cm_l^j) (\lambda_l^i)^2 v_{l}^i(z)  \Big] |J_l^j|
\\
& \hspace{5.5cm} - \int_{J_{l}^j}  \frac{\delta^3}{2} S''_l(cm_l^i - cm_l^j)\big( v_{l}^j(\xi) \big)^2 v_l^i(z) \,d\xi+R(S'''_l,\delta^4).
\end{aligned}
\end{equation}
Summing up the contributions in \eqref{Taylor1} to \eqref{Taylor2}, we get that the \emph{self-interaction} part in \eqref{functional equation} reduces to
\begin{equation}\label{Taylor3}
 \delta^3 \Big[ \frac{\delta^{-1}}{2} S'_l( cm_l^i - cm_l^j ) v_l^i(z)(v_l^i(z)-2\lambda_l^i)  +  \frac{1}{6} S''_l(cm_l^i - cm_l^j)  \big( (v_{l}^i(z))^3 - 3 v_{l}^i(z) (\lambda_l^i)^2 \big) \Big] |J_l^j|+R(S'''_l,\delta^4).
\end{equation}
Similarly, for the \emph{cross-interaction} part we obtain
\begin{equation}
\label{Taylor4}
 \delta^3 \Big[ \frac{\delta^{-1}}{2} K'( cm_l^i - cm_h^j ) v_l^i(z)(v_l^i(z)-2\lambda_l^i)  +  \frac{1}{6} K''(cm_l^i - cm_h^j)  \big( (v_{l}^i(z))^3 - 3 v_{l}^i(z) (\lambda_l^i)^2 \big) \Big] |J_h^j|+R(K''',\delta^4).
\end{equation}
Putting together the contributions of \eqref{Taylor3} and \eqref{Taylor4} in the functional equation \eqref{functional equation}, we get
\begin{equation}\label{functinal1}
\begin{aligned}
\mathcal{F}_l^i[p;\delta](z) &=  \,  (\bar{z}_l^i - z) +  \frac{D_{l}^i}{6} \Big( 3 v_{l}^i(z) (\lambda_l^i)^2 - (v_{l}^i(z))^3 \Big)   +  \delta^{-1} \frac{B_{l}^i}{2} v_l^i(z)(v_l^i(z)-2\lambda_l^i) +R(S'''_l,K''',\delta^4),
\end{aligned}
\end{equation}
where we used the notations introduced in \eqref{dli} and \eqref{bli}, namely
$$
D_{l}^i = - \sum_{j=1}^{N_l} S''_l(cm_l^i - cm_l^j) |J_l^j| -\alpha_l \sum_{j=1}^{N_h}  K''(cm_l^i - cm_h^j ) |J_h^j| ,
$$
and
$$
B_{l}^i = \sum_{j=1}^{N_l} S'_l( cm_l^i - cm_l^j ) |J_l^j| +  \alpha_l\sum_{j=1}^{N_h} K'( cm_l^i - cm_h^j ) |J_h^j|.
$$
Note that since the values $cm_l^i$ satisfy \eqref{eq:cmi} we have that $B_{l}^i=0$.  After the manipulations above, equation $\mathcal{F}_l^i[p;0](z) = 0$ reads as
\begin{equation}\label{eq:Bare_PI}
       (\bar{z}_l^i - z)  + \frac{D_{l}^i}{6} \Big( 3 v_{l}^i(z) (\lambda_l^i)^2 - (v_{l}^i(z))^3 \Big)  = 0, \quad z\in J_l^i,
\end{equation}
that gives a unique solution once the value of $\lambda_l^i$ is determined. In order to do that, let us introduce the following quantity 
\begin{equation}\label{Lambda}
    \Lambda_l^i[p;\delta] = \mathcal{F}_l^i[p;\delta](\tilde{z}_l^i)
\end{equation}
Performing Taylor expansions similar to the ones in \eqref{Taylor1} and \eqref{Taylor2} we get that
\begin{equation}\label{functinal2}
\begin{aligned}
\Lambda_l^i[p;0] & =  \,  (\bar{z}_l^i - \tilde{z}_l^i) +  \frac{D_{l}^i}{3} (\lambda_l^i)^3,
\end{aligned}
\end{equation}
and we are now in the position to solve
\begin{equation}\label{systemdelta0}
\begin{dcases}
   (\bar{z}_l^i - z)  + \frac{D_{l}^i}{6} \Big( 3 v_{l}^i(z) (\lambda_l^i)^2 - (v_{l}^i(z))^3 \Big)  = 0,
\\
(\bar{z}_l^i - \tilde{z}_l^i) +  \frac{D_{l}^i}{3} (\lambda_l^i)^3  = 0.
\end{dcases}
\end{equation}
The second equation in \eqref{systemdelta0} admits a solution once we have that $D_l^i > 0$, and $\lambda_l^i$ is uniquely determined by
\begin{equation}
\lambda_l^i = \left(\frac{3(\tilde{z}_l^i - \bar{z}_l^i)}{D_l^i} \right)^{1/3}.
\end{equation}
By construction $\lambda_l^i\geq \lambda_l^j$ if $i\geq j$, and this implies that equation \eqref{systemdelta0} admits the unique solution $\bar{v}_l^i$ which can be recovered as the pseudo inverse of the following Barenblatt type profiles
\begin{equation}
\begin{aligned}
\bar{\rho}^i(x)& = \frac{D_\rho^i}{2}\left((\lambda_\rho^i)^2-(x-cm_\rho^i)^2\right)\mathds{1}_{I_{\rho}^i}(x), \quad i=1,\ldots,N_\rho,
\\
\bar{\eta}^h(x)& = \frac{D_\eta^h}{2}\left((\lambda_\eta^h)^2-(x-cm_\eta^h)^2\right)\mathds{1}_{I_{\eta}^h}(x),\quad h=1,\ldots,N_\eta,
\end{aligned}
\end{equation}
where the intervals $I_{\rho}^i=\left[l_\rho^i,r_\rho^i\right]$ and $I_{\eta}^h=\left[l_\eta^h,r_\eta^h\right]$ are determined imposing
\[
 l_k^{i_k} = cm_k^{i_k}-\lambda_k^{i_k}, \quad r_k^{i_k} = cm_k^{i_k}+\lambda_k^{i_k}, \quad i_k=1,\ldots,N_k,\,k=\rho,\eta.
\]

We are now ready to reformulate \eqref{Pseudo-inverse6} as a functional equation on a proper Banach space. Consider the spaces
\begin{equation}\label{eq:Omegai}
    \Omega_l^i = \left\{v \in L^\infty\left(\left[\bar{z}_l^i,\tilde{z}_l^i\right)\right)\,|\, v \mbox{ increasing},\, v(\bar{z}_l^i)=0\right\},\, i = 1,\ldots,N_l,\, l\in\{\rho,\eta\},
\end{equation}
endowed with the  $L^\infty$ norm and take the product spaces
\[
 \Omega_l = \bigtimes_{i=1}^{N_l}\Omega_l^i,\,\mbox{ for } l\in\{\rho,\eta\}.
\]
We now introduce the space $\Omega$ defined by
\begin{equation}\label{eq:Omega}
    \Omega = \Omega_\rho \times \R^{N_\rho}\times\Omega_\eta \times \R^{N_\eta},
\end{equation}
with elements $\omega=(v_\rho^1,\ldots,v_\rho^{N_\rho},\lambda_\rho^1,\ldots,\lambda_\rho^{N_\rho},v_\eta^1,\ldots,v_\eta^{N_\eta},\lambda_\eta^1,\ldots,\lambda_\eta^{N_\eta}) $ endowed with the norm
\begin{equation}
|||\omega ||| = \sum_{i=1}^{N_\rho}\Big(\|v_{\rho}^i\|_{L^{\infty}}+ |\lambda_{\rho}^i| \Big)+ \sum_{i=1}^{N_\eta}\Big(\|v_{\eta}^i\|_{L^{\infty}}  + |\lambda_{\eta}^i|\Big).
\end{equation}
For $\gamma>0$, calling $\tilde{J}_l^i= \left[\bar{z}_l^i,\tilde{z}_l^i\right)$,  we consider the norm
\begin{equation}\label{alpha-norm}
 ||| \omega|||_{\gamma} = ||| \omega||| + \sum_{l\in \{ \rho,\eta \} }\sum_{i=1}^{N_l} \sup_{z\in \tilde{J}_l^i}\frac{|\lambda_l^i-v_l^i(z)|}{(\tilde{z}_l^i - z)^{\gamma}},
 \end{equation}
and set $\Omega_\gamma := \{ \omega \in\Omega ~:~ |||\omega|||_\gamma < +\infty\}$. For a given $\omega\in\Omega$, we define the operator $\mathcal{T}:\Omega_{\frac{1}{2}}\rightarrow \Omega_1$
\begin{equation}\label{eq:oper}
\mathcal{T}[\omega;\delta](z) := 
\begin{pmatrix}
\mathcal{F}_{\rho}[\omega;\delta](z) 
\\[0.2cm]
\Lambda_{\rho}[\omega;\delta]
\\[0.2cm]
\mathcal{F}_{\eta}[\omega;\delta](z) 
\\[0.2cm]
\Lambda_{\eta}[\omega;\delta]
\end{pmatrix}
\end{equation}
where for $l\in \{ \rho,\eta \}$ we have shorted the notation introducing
\begin{equation}\label{functionalsum}
\mF_{l}[\omega;\delta](z)  = \left( \mF_l^1[\omega;\delta](z),\ldots,\mF_l^{N_l}[\omega;\delta](z)  \right),\quad 
\Lambda_{l}[\omega;\delta](z)  :=\left( \Lambda_l^1[\omega;\delta],\ldots,\Lambda_l^{N_l}[\omega;\delta] \right).
\end{equation}
The operator $\mathcal{T}$ is a bounded operator for any fixed $\delta\geq 0$ and can be continuously extended at $\delta = 0$ to \eqref{eq:Bare_PI} and \eqref{functinal2}. In order to prove existence of stationary solutions for small $\delta > 0$ using the Implicit
Function Theorem, we need to prove that the Jacobian matrix of $\mathcal{T}$ is a bounded linear operator form $\Omega_{1/2}$ to $\Omega_1$ with bounded inverse. The Jacobian of $\mathcal{T}$ has the following structure
\begin{equation}\label{Jac_Mat}
D\mathcal{T}[\omega;\delta] = \begin{pmatrix} 
D_{v_\rho}\mF_\rho(\delta) 
& 
D_{\lambda_\rho}\mF_\rho(\delta) 
&
D_{v_\eta}\mF_\rho(\delta) 
& 
D_{\lambda_\eta}\mF_\rho(\delta) 
\\[0.2 cm]
D_{v_\rho}\Lambda_\rho(\delta) 
& 
D_{\lambda_\rho}\Lambda_\rho(\delta) 
&
D_{v_\eta}\Lambda_\rho(\delta) 
& 
D_{\lambda_\eta}\Lambda_\rho(\delta) 
\\[0.2 cm]
D_{v_\rho}\mF_\eta(\delta) 
& 
D_{\lambda_\rho}\mF_\eta(\delta) 
&
D_{v_\eta}\mF_\eta(\delta) 
& 
D_{\lambda_\eta}\mF_\eta(\delta) 
\\[0.2 cm]
D_{v_\rho}\Lambda_\eta(\delta) 
& 
D_{\lambda_\rho}\Lambda_\eta(\delta) 
&
D_{v_\eta}\Lambda_\eta(\delta) 
& 
D_{\lambda_\eta}\Lambda_\eta(\delta) 
\end{pmatrix},
\end{equation}
where the components are actually matrices defined by
\begin{equation*}
    \begin{aligned}
     &D_{v_h}\mF_l(\delta)   = \left(\frac{ \partial\mF_l^i[\omega;\delta]}{\partial v_h^j }(\nu_h^j) \right)_{i,j=1}^{N_l,N_h}, \quad  D_{\lambda_h}\mF_l(\delta)   = \left(\frac{ \partial\mF_l^i[\omega;\delta]}{\partial \lambda_h ^j}(a_h^j) \right)_{i,j=1}^{N_l,N_h} \\
     &D_{v_h}\Lambda_l(\delta)   = \left(\frac{ \partial\Lambda_l^i[\omega;\delta]}{\partial v_h ^j}(\nu_h^j) \right)_{i,j=1}^{N_l,N_h}, \quad  D_{\lambda_h}\Lambda_l(\delta)   = \left(\frac{ \partial\Lambda_l^i[\omega;\delta]}{\partial \lambda_h^j }(a_h^j) \right)_{i,j=1}^{N_l,N_h},
    \end{aligned}
\end{equation*}
where $\nu_h^j$ and $a_h^j$ are generic directions.  We first compute the diagonal terms in the matrix $D_{v_l}\mF_l(\delta)$. We have
\begin{equation*}
\begin{aligned}
& \frac{\partial\mF_l^i[\omega;\delta]}{\partial v_l^i}(\nu_l^i)  =
-\delta ^{-2}\int_{J_{l}^i} \nu_{l}^i(\xi) \Bigg [S_{l}\Big( \delta\big( v_{l}^i(z)-v_{l}^i(\xi)\big) \Big)
- \delta  v_l^i(z) S'_{l}\Big(\delta\big( \lambda_l^i - v_{l}^i(\xi)\big) \Big) - S_l\Big( -\delta v_{l}^i(\xi) \Big) \Bigg] \,d\xi 
\\
& + \delta^{-2}\nu_{l}^i(z) \sum_{j=1 }^{N_l}\int_{J_{l}^j}\Bigg[S_{l}\Big(cm_l^i - cm_l^j + \delta\big( v_{l}^i(z)-v_{l}^j(\xi)\big) \Big) -S_{l}\Big(cm_l^i - cm_l^j + \delta\big( \lambda_l^i - v_{l}^j(\xi)\big) \Big) \Bigg] \,d\xi 
\\
& + \delta^{-2}\nu_{l}^i(z)\alpha_{l}\sum_{j=1}^{N_h}\int_{J_{h}^j}\Bigg[K\Big(cm_l^i - cm_h^j + \delta\big( v_{l}^i(z)-v_{h}^j(\xi) \big)\Big)-K\Big(cm_l^i - cm_h^j + \delta\big( \lambda_l^i - v_{h}^j(\xi) \big)\Big)\Bigg]d\xi.
\end{aligned}
\end{equation*}
A Taylor expansion around $\delta = 0$ similar to the ones in \eqref{Taylor1} - \eqref{Taylor4} easily  gives that in the limit $\delta\to 0$ we obtain
\begin{equation*}
\frac{\partial\mF_l^i[\omega;0]}{\partial v_l^i}(\nu_l^i) =  \,     \frac{D_l^i}{2} \Big(  (\lambda_l^i)^2  -  (v_{l}^i(z))^2  \Big) \nu_{l}^i(z).
\end{equation*}
Concerning the other terms in $D_{v_l}\mF_l(\delta)$ we get
\begin{equation*}
\begin{aligned}
 \frac{\partial\mF_l^i[\omega;\delta]}{\partial v_l^j}(\nu_l^j)  =-\delta ^{-2}\sum_{j=1}^{N_l}\int_{J_{l}^j}\nu_{l}^j(\xi)\Bigg[& S_{l}\Big( cm_l^i - cm_l^j+\delta\big( v_{l}^i(z)-v_{l}^j(\xi)\big) \Big)-S_l\Big( (cm_l^i - cm_l^j-\delta v_{l}^j(\xi) \Big)\\
 & -\delta v_l^i(z)S'_{l}\Big(cm_l^i - cm_l^j+\delta\big( \lambda_l^i - v_{l}^j(\xi)\big) \Big) \Bigg] \,d\xi,
\end{aligned}
\end{equation*}
that all vanish in the limit $\delta\to 0$. Let us now focus on the matrix $D_{\lambda_l}\mF_l(\delta)$. By \eqref{functional equation} it is easy to see that the only non-zero terms in $D_{\lambda_l}\mF_l(\delta)$ are the diagonal ones that are given by
\begin{equation*}
\begin{aligned}
 \frac{\partial\mF_l^i[\omega;\delta]}{\partial \lambda_l^i}(a_l^i) = 
\, -\delta^{-1}v_l^i(z)a_l^i\Bigg[ 
 & \sum_{j=1}^{N_l}\int_{J_{l}^j}S'_{l}\Big(cm_l^i - cm_l^j + \delta\big( \lambda_l^i - v_{l}^j(\xi)\big) \Big)  \,d\xi 
\\
& + \alpha_{l}\sum_{j=1}^{N_h}\int_{J_{h}^j}K'\Big(cm_l^i - cm_h^j + \delta\big( \lambda_l^i - v_{h}^j(\xi) \big)\Big)  \,d\xi \Bigg].
\end{aligned}
\end{equation*}
Then, Taylor expansion w.r.t. $\delta$ yields
\begin{equation*}
\begin{aligned}
\frac{\partial\mF_l^i[\omega;0]}{\partial \lambda_l^i}(a_l^i)  =  \,    D_l^i \lambda_l^i v_{l}^i(z) a_l^i.
\end{aligned}
\end{equation*}
Since all the entrances in the matrix $D_{\lambda_h}\mF_l(\delta)$ are zero, the last matrix that concerns $\mF_l^i$ is $D_{v_h}\mF_l(\delta)$. The elements of this matrix are given by
\begin{equation*}
\begin{aligned}
 \frac{\partial\mF_l^i[\omega;\delta]}{\partial v_h^j}(\nu_h^j)  =-\delta ^{-2}\alpha_l\sum_{j=1}^{N_h}\int_{J_{h}^j}\nu_{h}^j(\xi)\Bigg[& K\Big( cm_l^i - cm_h^j+\delta\big( v_{l}^i(z)-v_{h}^j(\xi)\big) \Big)-K\Big( cm_l^i - cm_h^j-\delta v_{h}^j(\xi) \Big)\\
 & -\delta v_l^i(z)K'\Big(cm_l^i - cm_h^j+\delta\big( \lambda_l^i - v_{h}^j(\xi)\big) \Big) \Bigg] \,d\xi,
\end{aligned}
\end{equation*}
that vanish in the limit $\delta\to 0$. We now start in computing the functional derivatives for $\Lambda_l^i$ in \eqref{Lambda}. Again we should consider the four matrix in \eqref{Jac_Mat}, and we start from $D_{v_l}\Lambda_l(\delta)$. Note that the terms in the diagonal are zero in this case and the others are given by
\begin{equation*}
\begin{aligned}
 \frac{\partial\Lambda_l^i[\omega;\delta]}{\partial v_l^j}(\nu_l^j)  =-\delta ^{-2}\sum_{j=1}^{N_l}\int_{J_{l}^j}\nu_{l}^j(\xi)\Bigg[& S_{l}\Big( cm_l^i - cm_l^j+\delta\big( \lambda_{l}^i(z)-v_{l}^j(\xi)\big) \Big)-S_l\Big( (cm_l^i - cm_l^j-\delta v_{l}^j(\xi) \Big)\\
 & -\delta\lambda_l^i S'_{l}\Big(cm_l^i - cm_l^j+\delta\big( \lambda_l^i - v_{l}^j(\xi)\big) \Big) \Bigg] \,d\xi.
\end{aligned}
\end{equation*}
The terms in $D_{v_h}\Lambda_l(\delta)$ are
\begin{equation*}
\begin{aligned}
 \frac{\partial\Lambda_l^i[\omega;\delta]}{\partial v_h^j}(\nu_h^j)  =-\delta ^{-2}\alpha_l\sum_{j=1}^{N_h}\int_{J_{h}^j}\nu_{h}^j(\xi)\Bigg[& K\Big( cm_l^i - cm_h^j+\delta\big( \lambda_{l}^i-v_{h}^j(\xi)\big) \Big)-K\Big( cm_l^i - cm_h^j-\delta v_{h}^j(\xi) \Big)\\
 & -\delta\lambda_l^i K'\Big(cm_l^i - cm_h^j+\delta\big( \lambda_l^i - v_{h}^j(\xi)\big) \Big) \Bigg] \,d\xi,
\end{aligned}
\end{equation*}
and the usual Taylor expansions around $\delta = 0$, shows that both the matrices  $D_{v_l}\Lambda_l(0)= D_{v_h}\Lambda_l(0)= 0$. Since $D_{\lambda_h}\Lambda_l(\delta)$ is trivially a zero matrix, only remains to compute the diagonal terms in $D_{\lambda_l}\Lambda_l(\delta)$. We have
\begin{equation*}
\begin{aligned}
 \frac{\partial\Lambda_l^i[\omega;\delta]}{\partial \lambda_l^i}(a_l^i)&=  - \delta^{-1}a_l^i \sum_{j=1}^{N_l}\int_{J_{l}^j} \lambda_l^i S'_{l}\Big(cm_l^i - cm_l^j + \delta\big( \lambda_l^i - v_{l}^j(\xi)\big) \Big) \,d\xi 
\\
&  - \delta^{-1}a_l^i\alpha_l \sum_{j=1}^{N_h}\int_{J_{h}^j} \lambda_l^i K'_{l}\Big(cm_l^i - cm_h^j + \delta\big( \lambda_l^i - v_{h}^j(\xi)\big) \Big) \,d\xi 
\\
\end{aligned}
\end{equation*}
The last Taylor expansion gives
\begin{equation*}
 \frac{\partial\Lambda_l^i[\omega;0]}{\partial \lambda_l^i}(a_l^i) = D_l^i (\lambda_l^i)^2 a_l^i.
\end{equation*}
We have proved that
\begin{equation}\label{matrix operator}
D\mathcal{T}[\omega;0] =
\begin{pmatrix} 
\diag\left(\frac{D_{\rho}^i}{2} \Big(  (\lambda_{\rho}^i)^2  -  (v_{\rho}^i)^2  \Big) \nu_{\rho}^i \right)
& \diag \left(D_{\rho}^i \lambda_{\rho}^i v_{\rho}^i a_{\rho}^i \right)
&
0
& 
0
\\[0.3cm] 
0 
&\diag \left(D_{\rho}^i (\lambda_{\rho}^i)^2 a_{\rho}^i \right) 
& 
0 
& 
0 
\\[0.3cm] 
0 
& 
0
& \hspace{-1cm}
\diag \left(\frac{D_{\eta}^i}{2} \Big(  (\lambda_{\eta}^i)^2  -  (v_{\eta}^i)^2  \Big) \nu_{\eta}^i \right)
& \displaystyle
\diag \left(D_{\eta}^i \lambda_{\eta}^i v_{\eta}^i a_{\eta}^i \right)
\\[0.3cm] 
0 
& 
0
& 
0 
& \diag \left(D_{\eta}^i (\lambda_{\eta}^i)^2 a_{\eta}^i \right)
\end{pmatrix},
\end{equation}
with $\diag(A_i)$ diagonal matrix with elements $A_i$.
Let us denote by $\omega_0$ the unique solution to \eqref{systemdelta0}, we have the following lemma:
\begin{lem}\label{lemm:1}
For $\delta>0$ small enough, the operator $D\mathcal{T}[\omega_0;\delta]$ is a bounded linear operator from $\Omega_{1/2}$ to $\Omega_1$.
\end{lem}
\begin{proof}
Thanks to the previous computations is easy to see that $D\mathcal{T}$ is a bounded linear operator from $\Omega$ into itself and it is continuous at $\delta = 0$. The definition of the norm in \eqref{alpha-norm} implies that   for $z\in \tilde{J}_l^i$ we need to control only that
\begin{equation}\label{sup1}
\begin{aligned}
\sup_{|||\omega|||_{1/2}\leq 1}&\frac{1}{(\tilde{z}_l^i - z)}\bigg| 
\frac{\partial \mathcal{F}_l^i}{\partial v_l^i}[\cdot,\delta](\nu_l^i)  
-\frac{\partial \Lambda_l^i}{\partial v_l^i}[\cdot;\delta](\nu_l^i)
-\bigg( \frac{\partial \mathcal{F}_l^i}{\partial v_l^i}[\cdot;0](\nu_l^i)  
-\frac{\partial \Lambda_l^i}{\partial v_l^i}[\cdot;0](\nu_l^i) \bigg)
\\
&+\frac{\partial \mathcal{F}_l^i}{\partial \lambda_l^i}[\cdot,\delta](a_l^i)  
-\frac{\partial \Lambda_l^i}{\partial \lambda_l^i}[\cdot;\delta](a_l^i)
-\bigg( \frac{\partial \mathcal{F}_l^i}{\partial \lambda_l^i}[\cdot;0](a_l^i)  
-\frac{\partial \Lambda_l^i}{\partial \lambda_l^i}[\cdot;0](a_l^i) \bigg)
\\
&+\frac{\partial \mathcal{F}_l^i}{\partial v_h^j}[\cdot,\delta](\nu_h^j) 
-\frac{\partial \Lambda_l^i}{\partial v_h^j}[\cdot;\delta](\nu_h^j)
-\bigg( \frac{\partial \mathcal{F}_l^i}{\partial v_h^j}[\cdot;0](\nu_h^j)  
-\frac{\partial \Lambda_l^i}{\partial v_h^j}[\cdot;0](\nu_h^j) \bigg)
\\
&+\frac{\partial \mathcal{F}_l^i}{\partial \lambda_h^j}[\cdot,\delta](a_h^j)
-\frac{\partial \Lambda_l^i}{\partial \lambda_h^j}[\cdot;\delta](a_h^j)
-\bigg( \frac{\partial \mathcal{F}_l^i}{\partial \lambda_h^j}[\cdot;0](a_h^j)  
-\frac{\partial \Lambda_l^i}{\partial \lambda_h^j}[\cdot;0](a_h^j) \bigg)
\bigg| \searrow 0 
\end{aligned}
\end{equation}
as $\delta\searrow 0$. We start estimating the third row in \eqref{sup1},
\begin{equation*}
\begin{aligned}
& \frac{1}{(\tilde{z}_l^i - z)}\Bigg[ 
\frac{\partial \mathcal{F}_l^i}{\partial v_h^j}[\cdot,\delta](\nu_h^j)  
-\frac{\partial \Lambda_l^i}{\partial v_h^j}[\cdot;\delta](\nu_h^j)
-\bigg( \frac{\partial \mathcal{F}_l^i}{\partial v_h^j}[\cdot;0](\nu_h^j)  
-\frac{\partial \Lambda_l^i}{\partial v_h^j}[\cdot;0](\nu_h^j) \bigg)\Bigg]\\
=& -\alpha_l \frac{(v_l^i(z)-\lambda_l^i)^2}{\tilde{z}_l^i-z}\sum_{j=1}^{N_h}\int_{J_h^j}\frac{K''\big(cm_l^i - cm_h^j + \delta(\lambda_l^i - v_h^j(\xi))\big)\nu_h^j(\xi)}{2} \, d\xi 
\\
&- \delta\alpha_l \frac{(v_l^i(z)-\lambda_l^i)^3}{\tilde{z}_l^i-z}\sum_{j=1}^{N_h}\int_{J_h^j}\frac{K'''(\tilde{x}(\xi))\nu_h^j(\xi)}{6} \, d\xi
\\
=& -\delta\alpha_l \frac{(v_l^i(z)-\lambda_l^i)^2}{\tilde{z}_l^i-z}\sum_{j=1}^{N_h}\int_{J_h^j}\frac{K'''(\bar{x}(\xi))(\lambda_l^i-v_h^j(\xi))\nu_h^j(\xi)}{2} \, d\xi 
\\
&-\delta\alpha_l \frac{(v_l^i(z)-\lambda_l^i)^3}{\tilde{z}_l^i-z}\sum_{j=1}^{N_h}\int_{J_h^j}\frac{K'''(\tilde{x}(\xi))\nu_h^j(\xi)}{6} \, d\xi.
\\
 = &-\alpha_l \left( \frac{(v_l^i(z)-\lambda_l^i)^2}{\tilde{z}_l^i-z} +\frac{(v_l^i(z)-\lambda_l^i)^3}{\tilde{z}_l^i-z} \right) O(\delta),
\end{aligned}
\end{equation*}
where in the first equality we did a Taylor expansion around the point $x_0 = cm_l^i - cm_h^j + \delta(\lambda_l^i - v_h^j(\xi))$ for the kernel $K\big(cm_l^i - cm_h^j + \delta( v_{l}^i(z)-v_{h}^j(\xi)) \big)$, while in the second equality we did a Taylor expansion around the point $x_0 = cm_l^i - cm_h^j$ for the kernel $K''\big(cm_l^i - cm_h^j + \delta(\lambda_l^i - v_h^j(\xi))\big)$. Similarly, we can show that
\begin{equation*}
\begin{aligned}
& \frac{1}{(\tilde{z}_l^i - z)}\Bigg[ 
\frac{\partial \mathcal{F}_l^i}{\partial \lambda_h^j}[\cdot,\delta](a_h^j)
-\frac{\partial \Lambda_l^i}{\partial \lambda_h^j}[\cdot;\delta](a_h^j)
-\bigg( \frac{\partial \mathcal{F}_l^i}{\partial \lambda_h^j}[\cdot;0](a_h^j)  
-\frac{\partial \Lambda_l^i}{\partial \lambda_h^j}[\cdot;0](a_h^j) \bigg)\Bigg] = 0.
\end{aligned}
\end{equation*}
The first two rows in \eqref{sup1} can be treated as follows,
\begin{equation*}
\begin{aligned}
& \frac{1}{(\tilde{z}_l^i - z)} \Bigg[ \frac{\partial \mathcal{F}_l^i}{\partial v_l^i}[\cdot,\delta](\nu_l^i)  
-\frac{\partial \Lambda_l^i}{\partial v_l^i}[\cdot;\delta](\nu_l^i)
-\bigg( \frac{\partial \mathcal{F}_l^i}{\partial v_l^i}[\cdot;0](\nu_l^i)  
-\frac{\partial \Lambda_l^i}{\partial v_l^i}[\cdot;0](\nu_l^i) \bigg)
\\
&+\frac{\partial \mathcal{F}_l^i}{\partial \lambda_l^i}[\cdot,\delta](a_l^i)  
-\frac{\partial \Lambda_l^i}{\partial \lambda_l^i}[\cdot;\delta](a_l^i)
-\bigg( \frac{\partial \mathcal{F}_l^i}{\partial \lambda_l^i}[\cdot;0](a_l^i)  
-\frac{\partial \Lambda_l^i}{\partial \lambda_l^i}[\cdot;0](a_l^i) \bigg) \Bigg]
\\
&= \,  \frac{\delta}{(\tilde{z}_l^i - z)}  \Bigg[ \sum_{j=1}^{N_l} \int_{J_l^j} \delta^{-2} (v_l^i(z) - \lambda_l^i)(\nu_l^i(z) - a_l^i) S'_l\big(cm_l^i - cm_l^j + \delta(\lambda_l^i - v_l^j(\xi))\big)
\\
&\qquad + \delta^{-1}\frac{1}{2} (v_l^i(z) - \lambda_l^i)^2(\nu_l^i(z) - \nu_l^j(\xi)) S''_l\big(cm_l^i - cm_l^j + \delta(\lambda_l^i - v_l^j(\xi))\big)
\\
& \qquad + \frac{1}{6} (v_l^i(z) - \lambda_l^i)^3(\nu_l^i(z) - \nu_l^j(\xi)) S'''_l\big(\tilde{x}_1(\xi)\big) \, d\xi
\\
&+\alpha_l\sum_{j=1}^{N_h} \int_{J_h^j} \delta^{-2} (v_l^i(z) - \lambda_l^i)(\nu_l^i(z) - a_l^i) K'\big(cm_l^i - cm_h^j + \delta(\lambda_l^i - v_h^j(\xi))\big)
\\
&\qquad + \delta^{-1}\frac{1}{2} (v_l^i(z) - \lambda_l^i)^2(\nu_l^i(z) - \nu_l^j(\xi)) K''\big(cm_l^i - cm_h^j + \delta(\lambda_l^i - v_h^j(\xi))\big)
\\
& \qquad + \frac{1}{6} (v_l^i(z) - \lambda_l^i)^3(\nu_l^i(z) - \nu_l^j(\xi)) K'''\big(\tilde{x}_2(\xi)\big) \, d\xi \Bigg]
\\
&= \, \Bigg( 2 \frac{(\nu_l^i(z) - a_l^i) (v_l^i(z) - \lambda_l^i)}{(\tilde{z}_l^i - z)} + (2\nu_l^i(z) - 1) \bigg( \frac{(v_l^i(z) - \lambda_l^i)^3}{(\tilde{z}_l^i - z)} + \frac{(v_l^i(z) - \lambda_l^i)^2}{(\tilde{z}_l^i - z)} \bigg) \Bigg) O(\delta).
\end{aligned}
\end{equation*}
Since the functions $v_l^i$ are components of a vector $\omega$ belonging to $\Omega_{1/2}$ the quantities
\[
\frac{\lambda_l^i-v_l^i(z)}{(\tilde{z}_l^i - z)^{1/2}}
\]
are uniformly bounded in $\tilde{J}_l^i$, that gives \eqref{sup1}.
\end{proof}
\begin{lem}\label{lemm:2}
For any $\delta>0$ small enough, $D\mathcal{T}[\omega_0; 0]:\Omega_{1/2}\rightarrow\Omega_{1}$ is a linear isomorphism.
\end{lem}
\begin{proof}
Given $w\in\Omega_1$,  we have to prove that 
\begin{equation} \label{linearsystem1}
D\mathcal{T}[\omega_0;0] \omega = w,
\end{equation}  
admits a unique solution $\omega\in\Omega_{1/2}$ with the property
\begin{equation*}
||\omega||_{1/2}\leq C ||w||_{1}.
\end{equation*}
The determinant of the matrix in \eqref{matrix operator} is given by
\begin{equation}
\det D\mathcal{T} = \left(\prod_{i=1}^{N_{\rho}}\frac{(D_{\rho}^i)^2}{2} \Big(  (\lambda_{\rho}^i)^2  -  (v_{\rho}^i)^2  \Big) (\lambda_{\rho}^i)^2 \right) \cdot \left(\prod_{i=1}^{N_{\eta}}\frac{(D_{\eta}^i)^2}{2} \Big(  (\lambda_{\eta}^i)^2  -  (v_{\eta}^i)^2  \Big) (\lambda_{\eta}^i)^2  \right) ,
\end{equation}
that is always different from zero under the condition $D_l^i > 0$ and since $\big( v_l^i(z)-\lambda_l^i\big) < 0$ on $z\in[\bar{z}_l^i, \tilde{z}_l^i)$. Thanks to the structure in \eqref{matrix operator} and denoting with $\nu_l^i,a_l^i$ and $\sigma_l^i,k_l^i$ the generic entrances in $\omega$ and $w$ respectively, we easily get that
\begin{equation*}
\begin{aligned}
\nu_l^i(z) & =\frac{- 2 \sigma_l^i(z)}{D_l^i \big( (v_l^i(z))^2 - (\lambda_l^i)^2 \big)} + \frac{2 \lambda_l^i v_l^i(z) a_l^i}{ (v_l^i(z))^2 - (\lambda_l^i)^2 },
\\
a_l^i  & = \frac{k_l^i}{D_l^i(\lambda_l^i)^2},
\end{aligned}
\end{equation*}
that implies $||\nu_l^i||_{1/2}\leq C ||\sigma_l^i||_{\infty}$ for $i=1,2,\cdots,N_l,$ and $l \in \{\rho,\eta\}$. In order to close the argument, it is enough to note that the ratio
\begin{equation*}
\frac{a_l^i-\nu_l^i(z)}{(\tilde{z}_l^i - z)^{1/2}} ,
\end{equation*}
is uniformly bounded since $(\lambda_l^i - v_l^i(z))/(\tilde{z}_l^i-z)$ is uniformly bounded, see \cite[Lemma 4.4]{budif}.
\end{proof}
 We are now in the position of proving the main result of the paper, namely Theorem \ref{main_thm}, that we recall below for convenience.

\begin{thm}\label{main_thm_r}
Assume that the interaction kernels are under the assumptions (A1), (A2) and (A3). 
Consider $N_\rho,N_\eta\in\mathbb{N}$ and let $z_l^i$ be fixed positive numbers  for $i=1,2,\cdots,N_l,$ and $l \in \{\rho,\eta\}$. Consider two families of real numbers $\{cm_\rho^i\}_{i=1}^{N_\rho}$ and $\{cm_\eta^i\}_{i=1}^{N_\eta}$ such that
    \begin{itemize}
        \item[(i)]  $\{cm_\rho^i\}_{i=1}^{N_\rho}$ and $\{cm_\eta^i\}_{i=1}^{N_\eta}$ are stationary solutions of the purely non-local particle system, that is, for $i=1,2,\cdots,N_l,$ for  $l,h \in \{\rho,\eta\}$ and $l\neq h$,
\begin{equation*}
    B_{l}^i = \sum_{j=1}^{N_l} S'_l( cm_l^i - cm_l^j ) z_l^j+  \alpha_l\sum_{j=1}^{N_h} K'( cm_l^i - cm_h^j ) z_h^j=0,
\end{equation*}
\item[(ii)] the following quantities 
\begin{equation*}
 D_{l}^i = - \sum_{j=1}^{N_l} S''_l(cm_l^i - cm_l^j) z_l^j -\alpha_l \sum_{j=1}^{N_h}  K''(cm_l^i - cm_h^j ) z_h^j, 
\end{equation*}
are strictly positive, for all $i=1,2,\cdots,N_l$,  $l,h \in \{\rho,\eta\}$ and $l\neq h$.
\end{itemize}

Then, there exists
a constant $d_0$ such that for all $d\in (0, d_0)$ the stationary equation \eqref{stationary system1} admits a unique solution in the sense of Definition \ref{def:multi_bump} of the form
\[
    \rho(x) = \sum_{i=1}^{N_\rho}\rho^i(x)\mathds{1}_{I_\rho^i}(x)\quad\mbox{ and }\quad \eta (x) = \sum_{h=1}^{N_\eta}\eta^h(x)\mathds{1}_{I_\eta^h}(x)
\]
where 
\begin{itemize}
\item each interval $I_l^i$ is symmetric around $cm_l^i$ for all  $i=1,2,\cdots,N_l$,  $l \in \{\rho,\eta\}$,
\item $\rho^i$ and $\eta^j$ are $C^1$ and even w.r.t the center of $I_\rho^i$ and $I_\eta^j$ respectively, with masses $z_\rho^i$ and  $z_\eta^j$, for $i=1,...,N_\rho$ and $j=1,...,N_\eta$,
\item the solutions $\rho$ and $\eta$ have fixed masses 
\[ z_\rho=\sum_{i=1}^{N_\rho} z_\rho^i \mbox { and } z_\eta = \sum_{i=1}^{N_\eta} z_\eta^i,\]
respectively.
\end{itemize}
\end{thm}
\begin{proof}
Consider $z_\rho$ and $z_\eta$ fixed masses and a set of points $cm_l^i$ for $i=1,2,\cdots,N_l,$ and $l \in \{\rho,\eta\}$ that satisfy (i) and (ii). The results in Lemma \ref{lemm:1} and Lemma \ref{lemm:2} imply that given $\mathcal{T}$ defined in \eqref{eq:oper}, the functional equation
\[
\mathcal{T}[\omega;\delta](z) = 0,
\]
admits a unique solution $\omega = ( v_\rho^1(z),\ldots,v_\rho^{N_\rho}(z),\lambda_\rho^1,\ldots,\lambda_\rho^{N_\rho},v_\eta^1(z),\ldots,v_\eta^{N_\eta}(z),\lambda_\eta^1,\ldots,\lambda_\eta^{N_\eta} )$ for $\delta >0$ small enough. The entrances $v_l^i(z)$ are solutions to \eqref{Pseudo-inverse6} for $z\in J_l^i$. Consider now $u_l^i$ defined for $z\in J_l^i$ as $u_l^i (z) = cm_l^i + \delta v_l^i(z)$. Differentiating  \eqref{Pseudo-inverse6} twice w.r.t $z$ we get that $v_l^i$ is differentiable  and strictly increasing for $z\in J_l^i$ and that $u_l^i$ is a solution to \eqref{compact}. Moreover $u_l^i$ is also strictly increasing and we can define the inverse $F_l^i$ that and its spatial derivative $\rho_l^i = \partial_x F_l^i$ is a solution to \eqref{stationary system1}.
\end{proof}

\begin{rem}\label{rem:conditions}
Note that conditions (ii) in Theorem \ref{main_thm_r} turn to be  conditions on the positions of the centres of masses and on the value of $\alpha$. Indeed, as a sufficient condition for $D_\rho^i> 0$ we can assume that all the differences between the centres of masses are in the range of concavity of the kernels. Moreover,  $D_\eta^i> 0$ is satisfied if 
$$\alpha <\min_{i=1,\ldots,N_\eta}\frac{ \sum_{j=1}^{N_\eta} S''_\eta(cm_\eta^i - cm_\eta^j) |J_\eta^j|}{  \sum_{j=1}^{N_\rho}  K''(cm_\eta^i - cm_\rho^j ) |J_\rho^j| }.
$$
Note that the above conditions are comparable to the ones we got in the proofs of Section \ref{sec:preliminaries} using the Krein-Rutmann approach.
\end{rem}

\section{Numerics and Perspectives}\label{sec:numerics}
In this section, we study numerically solutions to system \eqref{main-equation} using two different methods, the finite volume method introduced in \cite{CCH,CHS} and particles method studied in \cite{DFFRR,FRa}. We validate the results about the existence of the mixed steady state, separated steady state, and the multiple bumps steady states. Moreover, we perform some  examples to show the variation in the behavior of the solution to system \eqref{main-equation} under different choices of initial data and the parameter $\alpha$, which, in turn, suggests a future work on the stability of the solutions to system \eqref{main-equation}. Finally, traveling waves are detected under a special choice of initial data and value for the parameter $\alpha$.
We begin by sketching the particles method. This method essentially consists in a finite difference discretization in space to the pseudo inverse version of system \eqref{main-equation}
\begin{equation}
\label{main pseudo inverse}
\begin{dcases}
\partial_t u_{\rho}(z) = -\frac{d}{2}\partial_{z}\Big(\big(\partial_{z}u_{\rho}(z)\big)^{-2}\Big) + \int_{J_{\rho}}S_{\rho}^{\prime}\big(u_{\rho}(z)-u_{\rho}(\xi)\big)d\xi + \alpha_{\rho}\int_{J_{\eta}}K^{\prime}\big(u_{\rho}(z)-u_{\eta}(\xi)\big)d\xi,~~z\in J_{\rho}
\\
\partial_t u_{\eta}(z) = -\frac{d}{2}\partial_{z}\Big(\big(\partial_{z}u_{\eta}(z)\big)^{-2}\Big) + \int_{J_{\eta}}S_{\eta}^{\prime}\big(u_{\eta}(z)-u_{\eta}(\xi)\big)d\xi + \alpha_{\eta}\int_{J_{\rho}}K^{\prime}\big(u_{\eta}(z)-u_{\rho}(\xi)\big)d\xi,~~z\in J_{\eta},
\end{dcases}
\end{equation}
as the following: Let $N\in\mathbb{N}$, let $\{z^i\}_{i=1}^N$ be a sequence of points that partition the interval $[0,1]$ uniformly. Denote by $X_l^i(t):= u_l(t,z^i)$ the approximating particles of the pseudo inverse at each point $z^i$ of the partition. Assuming that the densities $\rho$ and $\eta$ are of unit masses, then we have the following approximating system of ODEs
\begin{equation}
\label{particles system1}
\begin{dcases}
\partial_t X_{\rho}^i(t) =& \frac{d}{2N} \Big(\big(\rho^{i-1}(t)\big)^{-2} - \big(\rho^{i}(t)\big)^{-2} \Big) + \frac{1}{N}\sum_{j=1}^N S_{\rho}^{\prime}\big(X_{\rho}^i(t)-X_{\rho}^j(t)\big)d\xi 
\\
& + \frac{\alpha_{\rho}}{N}\sum_{j=1}^N K^{\prime}\big(X_{\rho}^i(t)-X_{\eta}^j(t)\big)d\xi, \qquad i=1,\cdots N,
\\
\partial_t X_{\eta}^i(t) =& \frac{d}{2N} \Big(\big(\eta^{i-1}(t)\big)^{-2} - \big(\eta^{i}(t)\big)^{-2} \Big) + \frac{1}{N}\sum_{j=1}^N S_{\eta}^{\prime}\big(X_{\eta}^i(t)-X_{\eta}^j(t)\big)d\xi 
\\
& + \frac{\alpha_{\eta}}{N}\sum_{j=1}^N K^{\prime}\big(X_{\eta}^i(t)-X_{\rho}^j(t)\big)d\xi, \qquad i=1,\cdots N,
\end{dcases}
\end{equation}
where the densities are reconstructed as
\begin{equation}
\begin{dcases}
\rho^i(t) = \frac{1}{N ( X_l^{i+1}(t) - X_l^i(t) )}, \qquad i = 1,\cdots, N-1,
\\
\rho^0(t) = 0,
\\
\rho^N(t) = 0,
\end{dcases}
\end{equation}
and the same for the  $\eta^i(t)$. To this end, we are ready to solve this particle system by applying the Runge-Kutta MATLAB solver ODE23, with initial positions $\displaystyle X_l(0)= X_{l,0} = \{ X_{l,0}^i \}_{i=1}^N,~l=\rho,\eta$ determined by solving
$$
\int_{X_{\rho,0}^i}^{X_{\rho,0}^{i+1}} \rho(t=0) dX = \frac{1}{N-1}, \qquad i = 1,\cdots,N-1.
$$

The second method we use is the finite volume method which  introduced in \cite{CCH} and extended to systems in \cite{CHS}, that consists in a $1$D positive preserving finite-volume method for system \eqref{main-equation}. To do so, we first partition the computational domain into finite-volume cells $U_{i}=[x_{i-\frac{1}{2}},x_{i+\frac{1}{2}}]$ of a uniform size $\Delta x$ with $x_{i} = i\Delta x$, $i \in\{-s,~\dots,s \}$. Define
\begin{equation*}
\widetilde{\rho}_{i}(t) := \frac{1}{\Delta x}\int_{U_{i}}\rho(x,t)dx,\qquad \widetilde{\eta}_{i}(t) := \frac{1}{\Delta x}\int_{U_{i}}\eta(x,t)dx,
\end{equation*}
the averages of the solutions $\rho,\eta$ computed at each cell $U_{i}$. Then we integrate each equation in system \eqref{main-equation} over each cell $U_{i}$, and so we obtain a semi-discrete finite-volume scheme described by the following system of ODEs for $\overline{\rho}^{i}$ and $\overline{\eta}^{i}$ 
\begin{equation}
\label{volume ode1}
\begin{dcases}
\frac{d\widetilde{\rho}_{i}(t)}{dt} = -\frac{F_{i+\frac{1}{2}}^{\rho}(t) - F_{i-\frac{1}{2}}^{\rho}(t) }{\Delta x},
\\
\frac{d\widetilde{\eta}_{i}(t)}{dt} = -\frac{F_{i+\frac{1}{2}}^{\eta}(t) - F_{i-\frac{1}{2}}^{\eta}(t) }{\Delta x},
\end{dcases}
\end{equation}
where the numerical flux $F_{i+\frac{1}{2}}^l$, $l=\rho,\eta$, is considered as an approximation for the continuous fluxes $-\rho(\varepsilon\rho + S_{\rho}*\rho + \alpha_{\rho} K*\eta )_{x}$ and $-\eta(\varepsilon\eta + S_{\eta}*\eta + \alpha_{\eta} K*\rho )_{x} $ respectively. More precisely, the expression for $F_{i+\frac{1}{2}}^{\rho}$ is given by
\begin{equation}
F_{i+\frac{1}{2}}^{\rho} = \max (\vartheta_{\rho}^{i+1},0)\Big[ \widetilde{\rho}_{i} + \frac{\Delta x}{2}(\rho_{x})_{i} \Big] + \min (\vartheta_{\rho}^{i+1},0)\Big[ \widetilde{\rho}_{i} - \frac{\Delta x}{2}(\rho_{x})^{i} \Big],
\end{equation}
where
\begin{equation}
\vartheta_{\rho}^{i+1} =  - \frac{\varepsilon}{\Delta x} \big( \widetilde{\rho}_{i+1} - \widetilde{\rho}_{i} \big) - \sum_{j}\widetilde{\rho}_{j}\big( S_{\rho}(x_{i+1} - x_{j}) - S_{\rho}(x_{i} - x_{j}) \big) - \alpha_{\rho} \sum_{j}\widetilde{\eta}_{j}\big( K(x_{i+1} - x_{j}) - K(x_{i} - x_{j}) \big),
\end{equation}
and
\begin{equation}\label{minmod}
(\rho_{x})^{i} = \mbox{minmod} \Bigg( 2\frac{\widetilde{\rho}_{i+1} - \widetilde{\rho}_{i}}{\Delta x},~ \frac{\widetilde{\rho}_{i+1} - \widetilde{\rho}_{i-1}}{2\Delta x},~ 2\frac{\widetilde{\rho}_{i} - \widetilde{\rho}_{i-1}}{\Delta x} \Bigg).
\end{equation}
The minmod limiter in \eqref{minmod} has the following definition
\begin{equation}
\mbox{minmod}(a_{1}, a_{2},~\dots) :=
\begin{cases}
\min (a_{1}, a_{2},~\dots), \quad \mbox{if} ~ a_{i} > 0\quad \forall i
\\
\max (a_{1}, a_{2},~\dots), \quad \mbox{if} ~ a_{i} < 0\quad \forall i
\\
0, \qquad\qquad\qquad\mbox{otherwise.}
\end{cases}
\end{equation}
We have the same as the above expresions for $\eta$. Finally, we integrate the semi-discrete scheme \eqref{volume ode1} numerically using the third-order strong preserving Runge-Kutta (SSP-RK) ODE solver used in \cite{gott}.

 \begin{figure}[htbp]
\begin{center}
\begin{multicols}{2}
    \includegraphics[width=8cm,height=5cm]{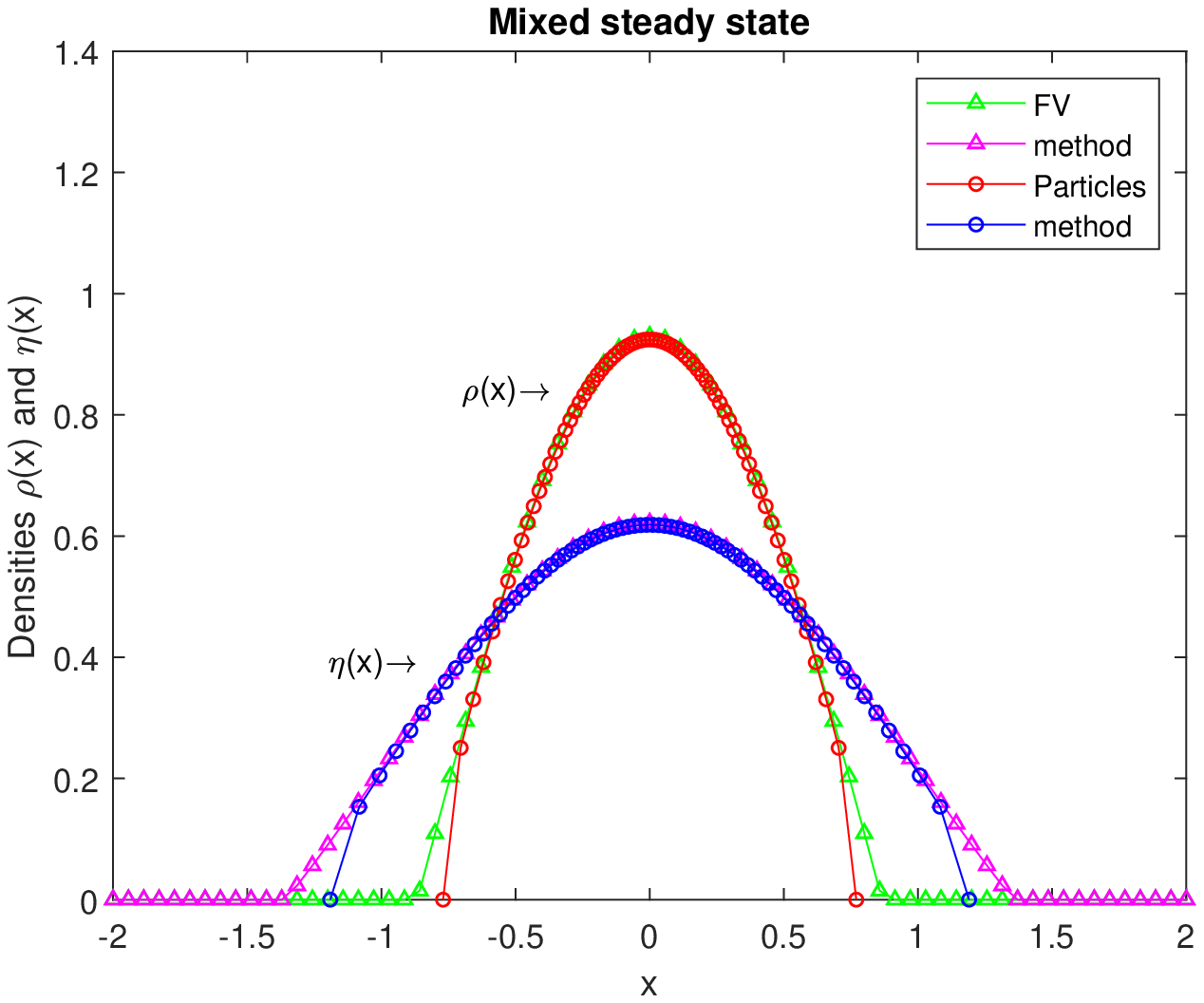}\par 
    \includegraphics[width=8cm,height=5cm]{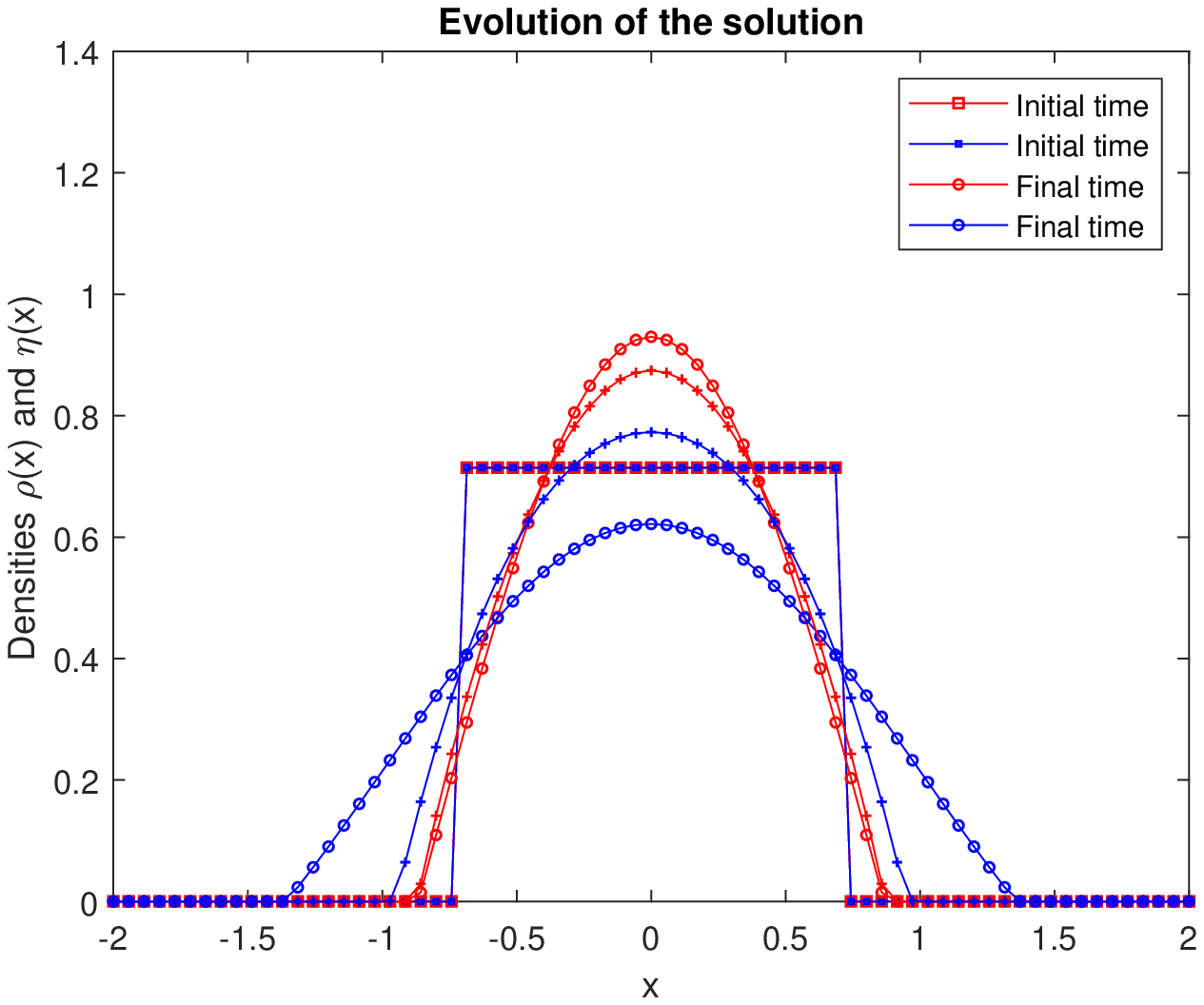}\par 
\end{multicols}
\begin{multicols}{2}
    \includegraphics[width=8cm,height=5cm]{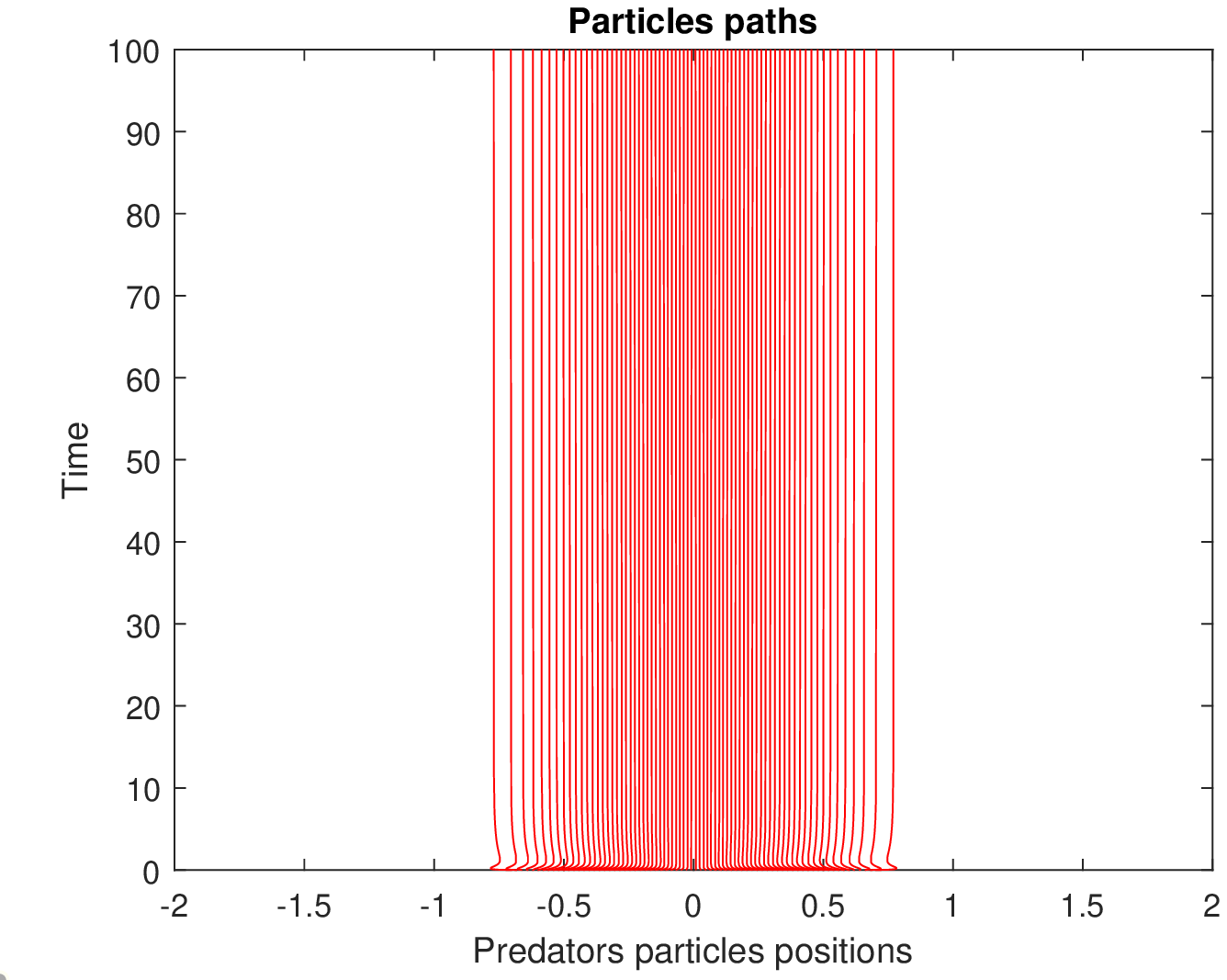}\par
    \includegraphics[width=8cm,height=5cm]{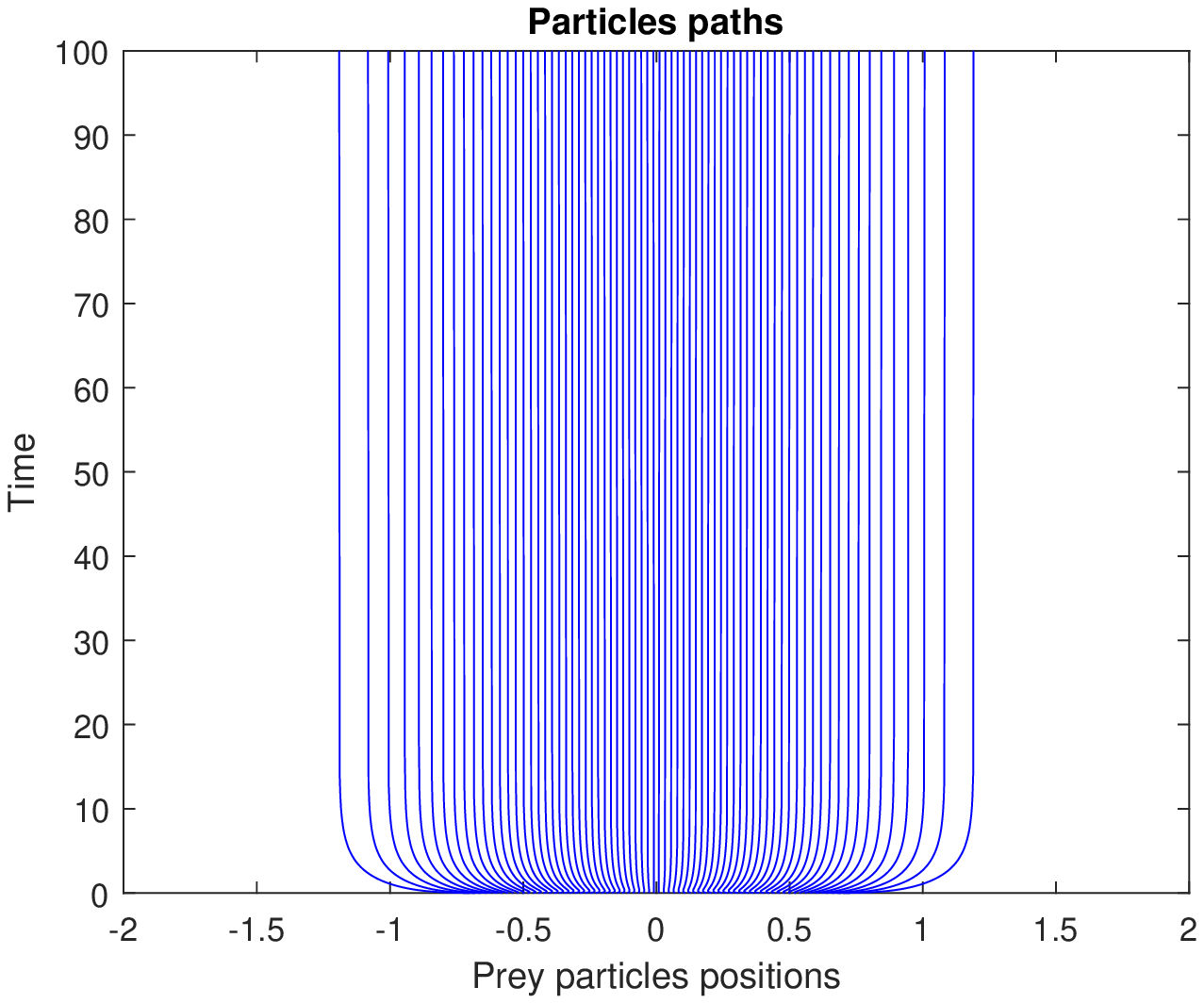}\par
\end{multicols}
\begin{multicols}{2}
\includegraphics[width=8cm,height=5cm]{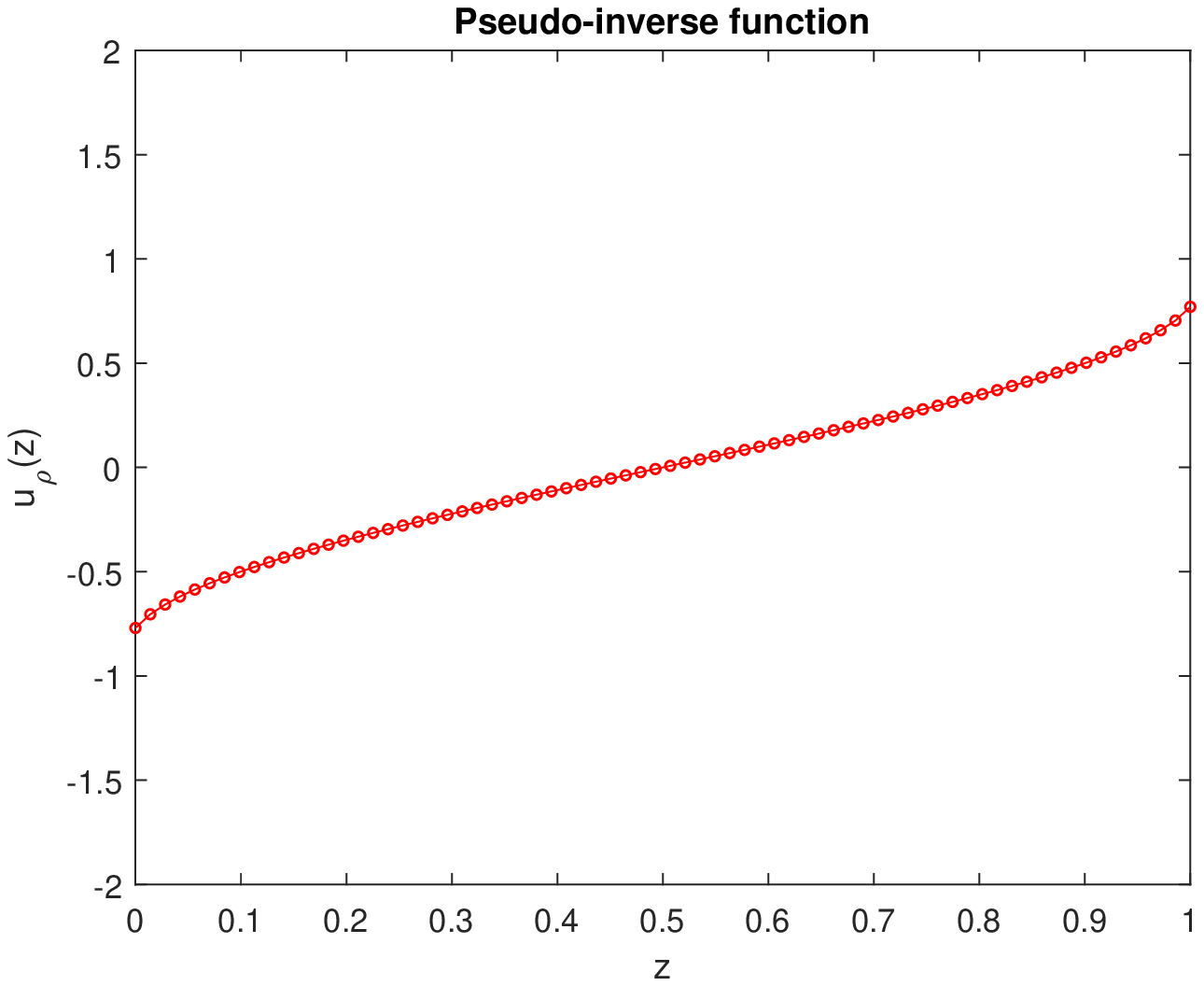}\par
   \includegraphics[width=8cm,height=5cm]{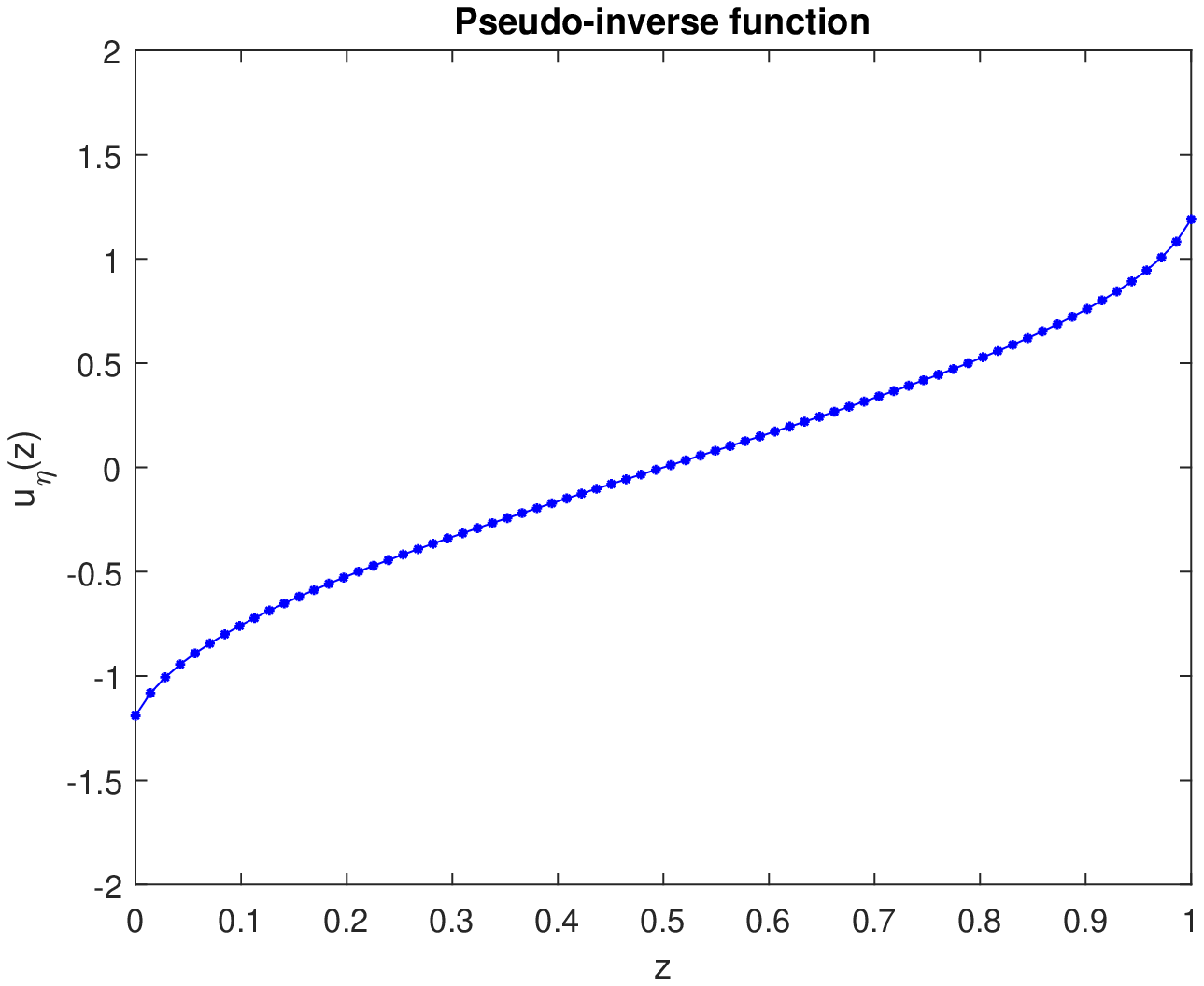}\par
\end{multicols}

\caption{
In this figure, a mixed steady state is plotted by using initial data given by \eqref{initial1}, $\alpha=0.1$, $d=0.4$. Number of particles are chosen equal to number of cells in the finite volume method, which is $N=71$. 
}
\label{Mixed SS}
\end{center}
\end{figure}

\begin{figure}[htbp]
\begin{center}
\begin{multicols}{2}
    \includegraphics[width=8cm,height=5cm]{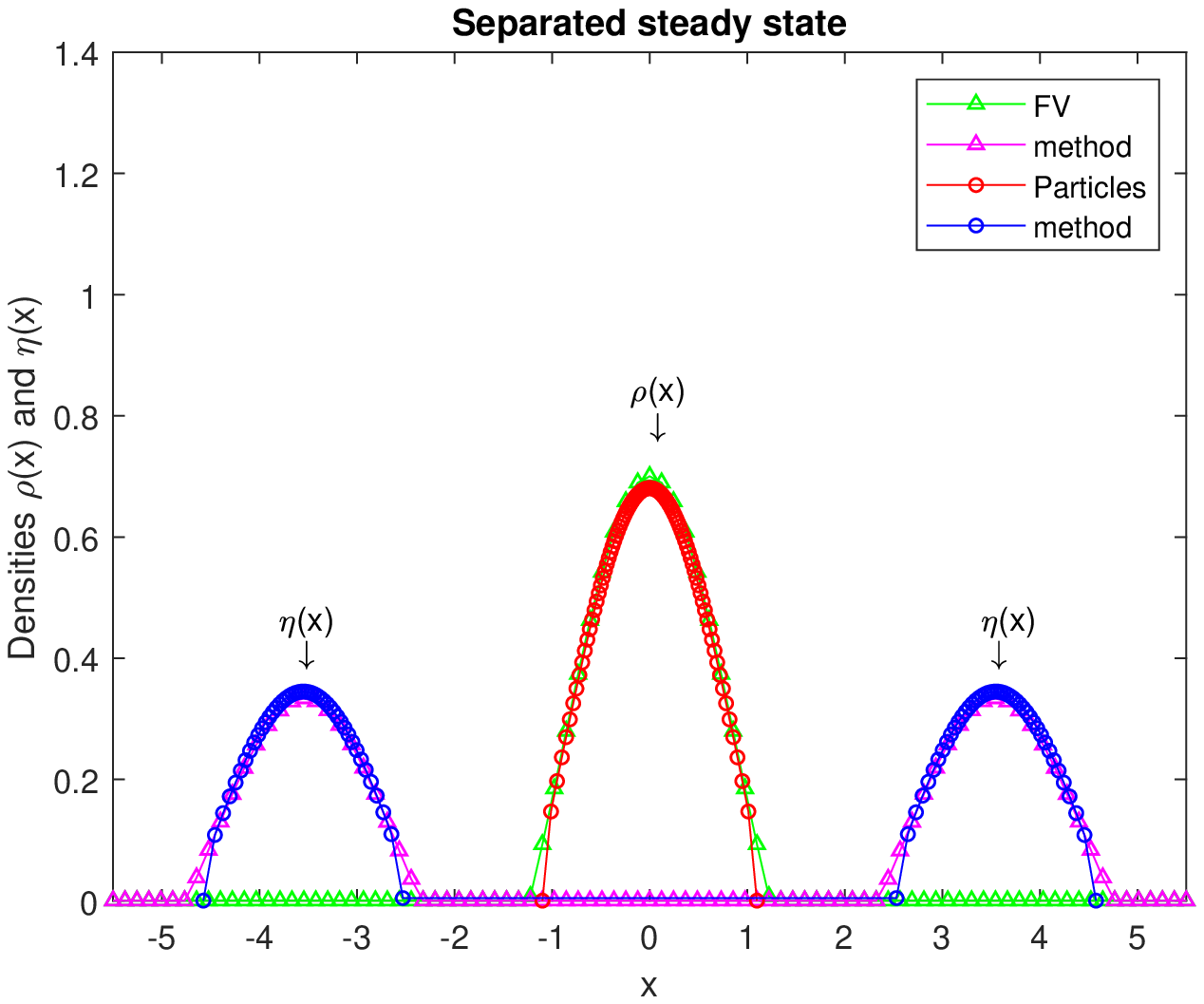}\par 
    \includegraphics[width=8cm,height=5cm]{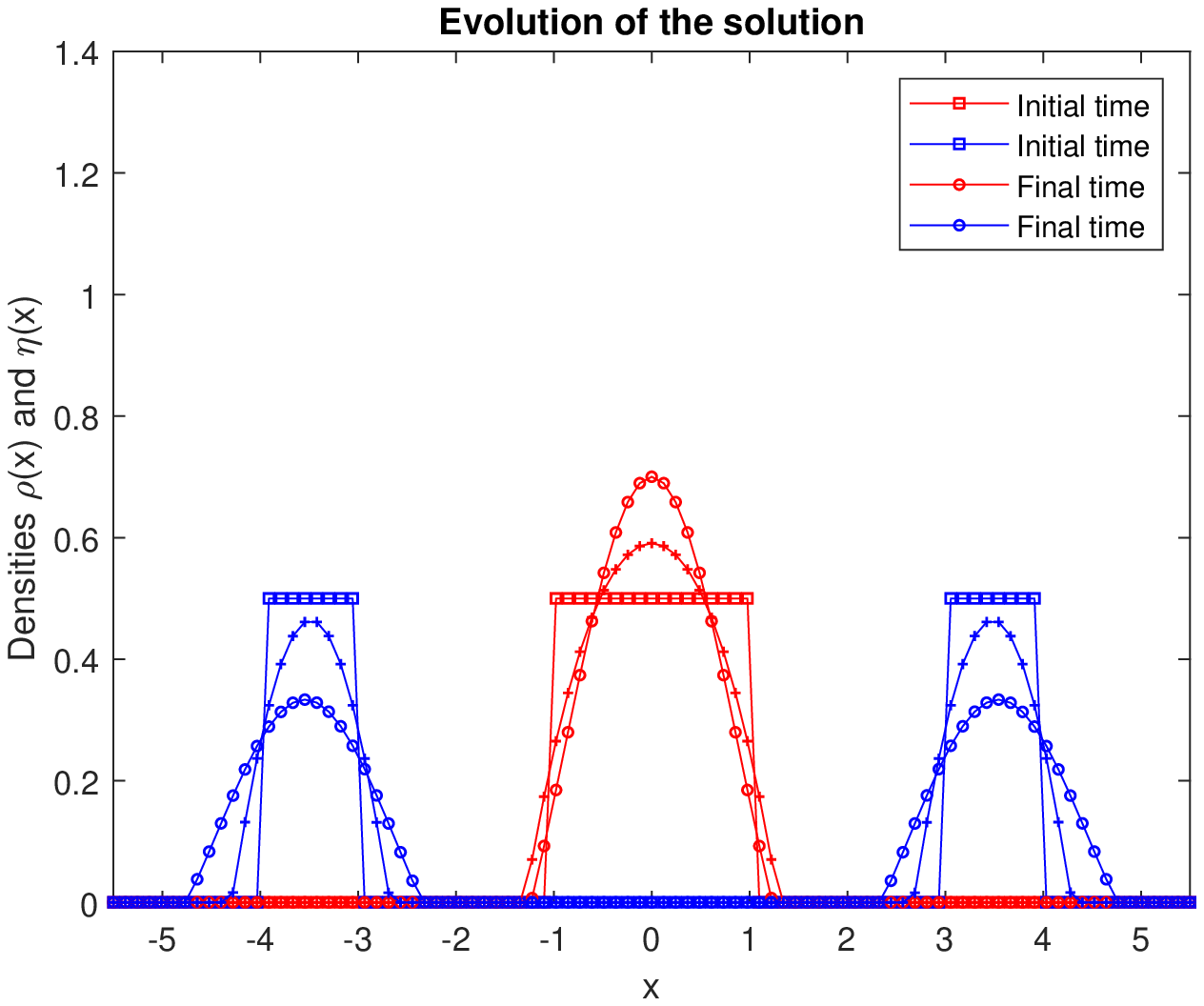}\par 
\end{multicols}
\begin{multicols}{2}
    \includegraphics[width=8cm,height=5cm]{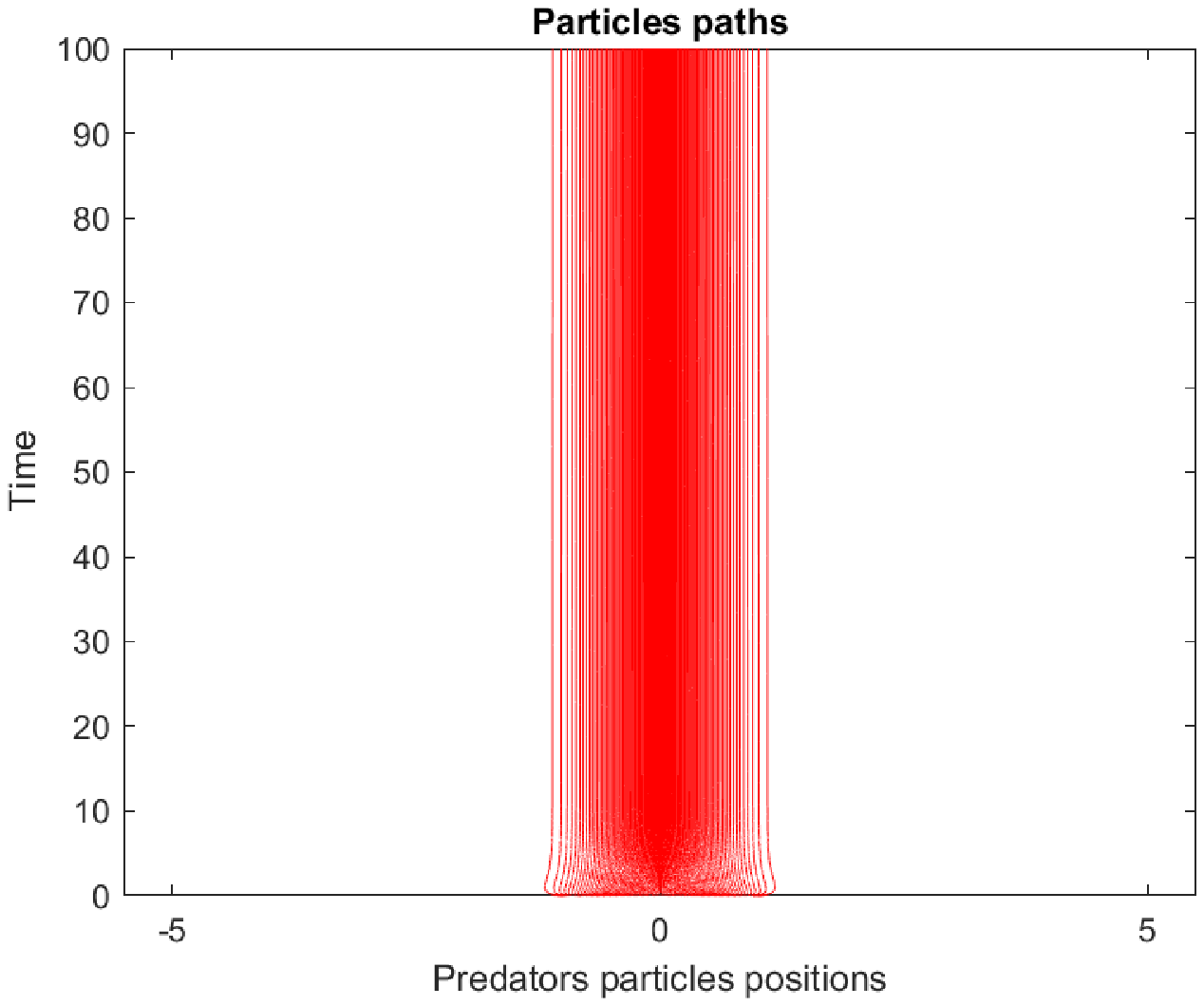}\par
    \includegraphics[width=8cm,height=5cm]{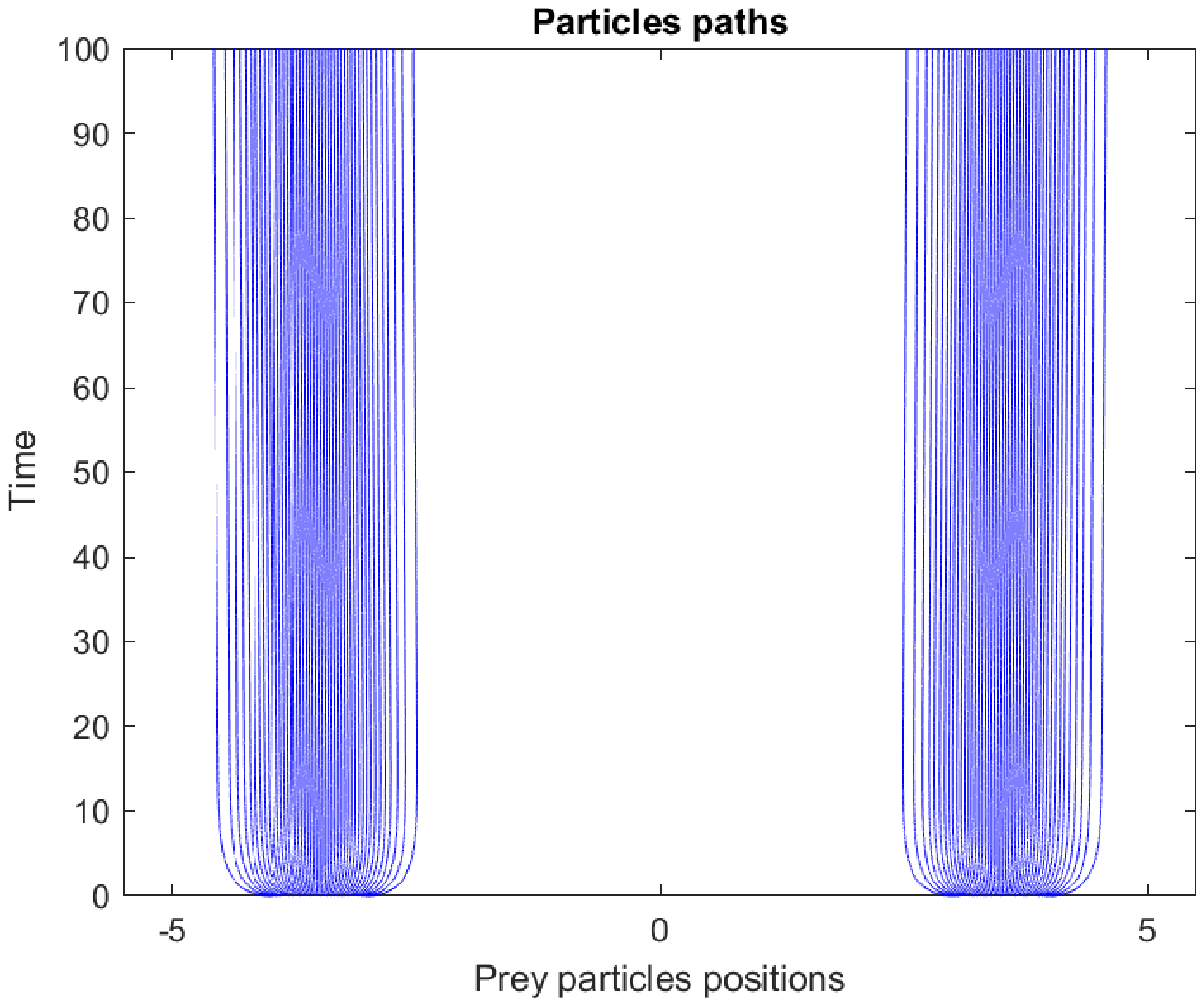}\par
\end{multicols}
\begin{multicols}{2}
\includegraphics[width=8cm,height=5cm]{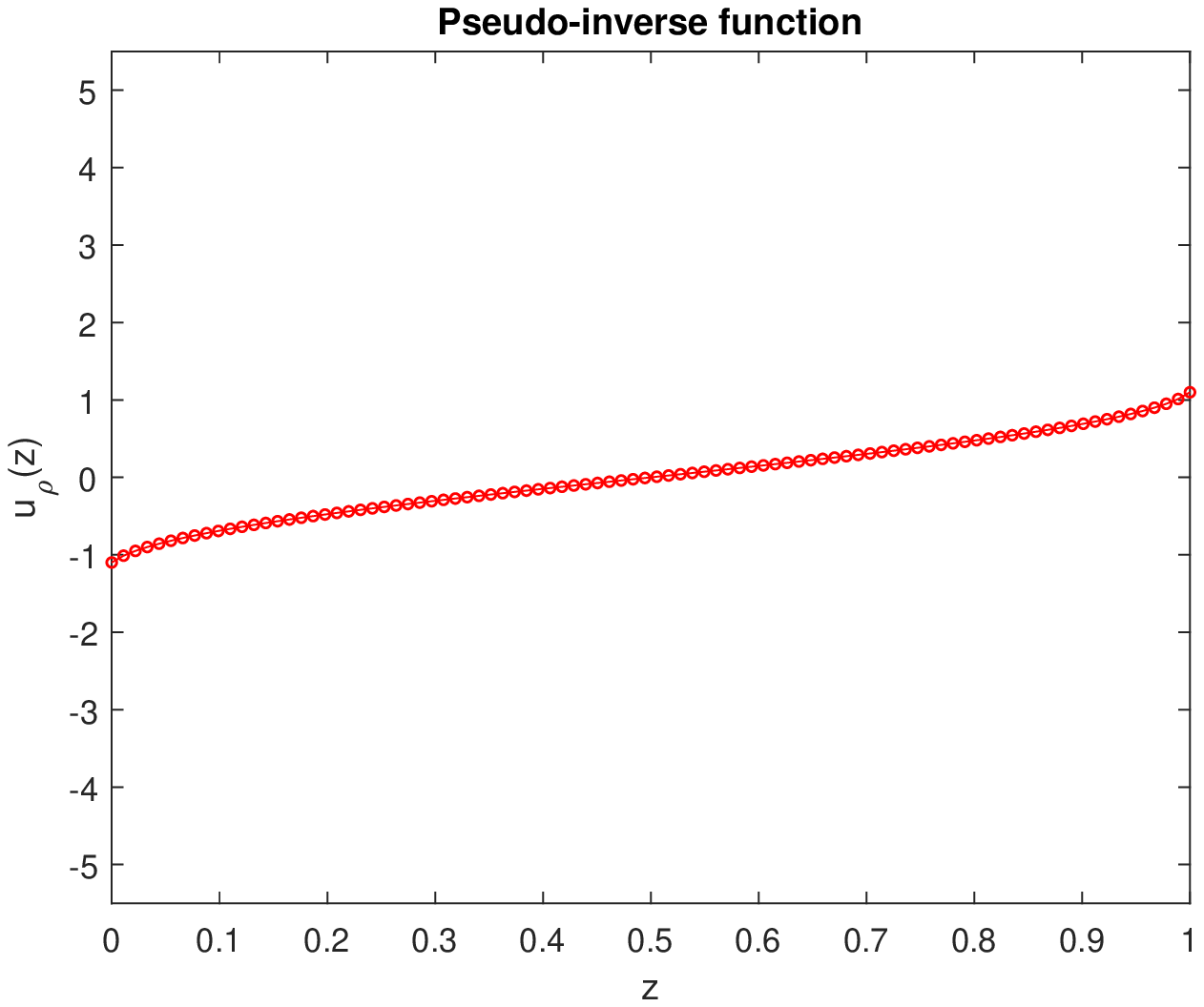}\par
   \includegraphics[width=8cm,height=5cm]{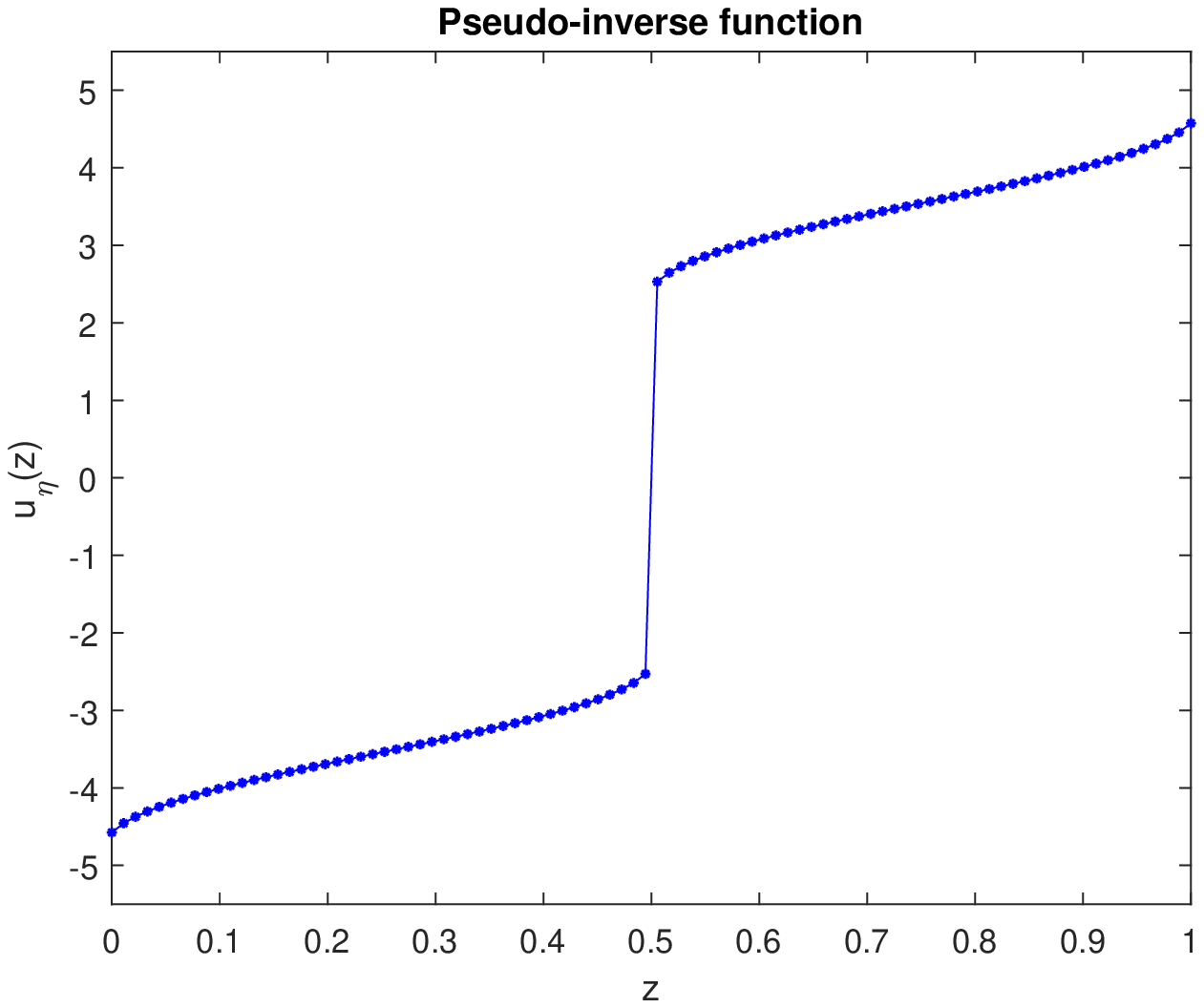}\par
\end{multicols}

\caption{A separated steady state is presented in this figure. Initial data are given by \eqref{initial2}. The parameters are $\alpha=0.2$ and $d=0.4$ with $N=91$}
\label{Separated SS}
\end{center}
\end{figure}

\begin{figure}[htbp]
\begin{center}
\begin{multicols}{2}
    \includegraphics[width=8cm,height=5cm]{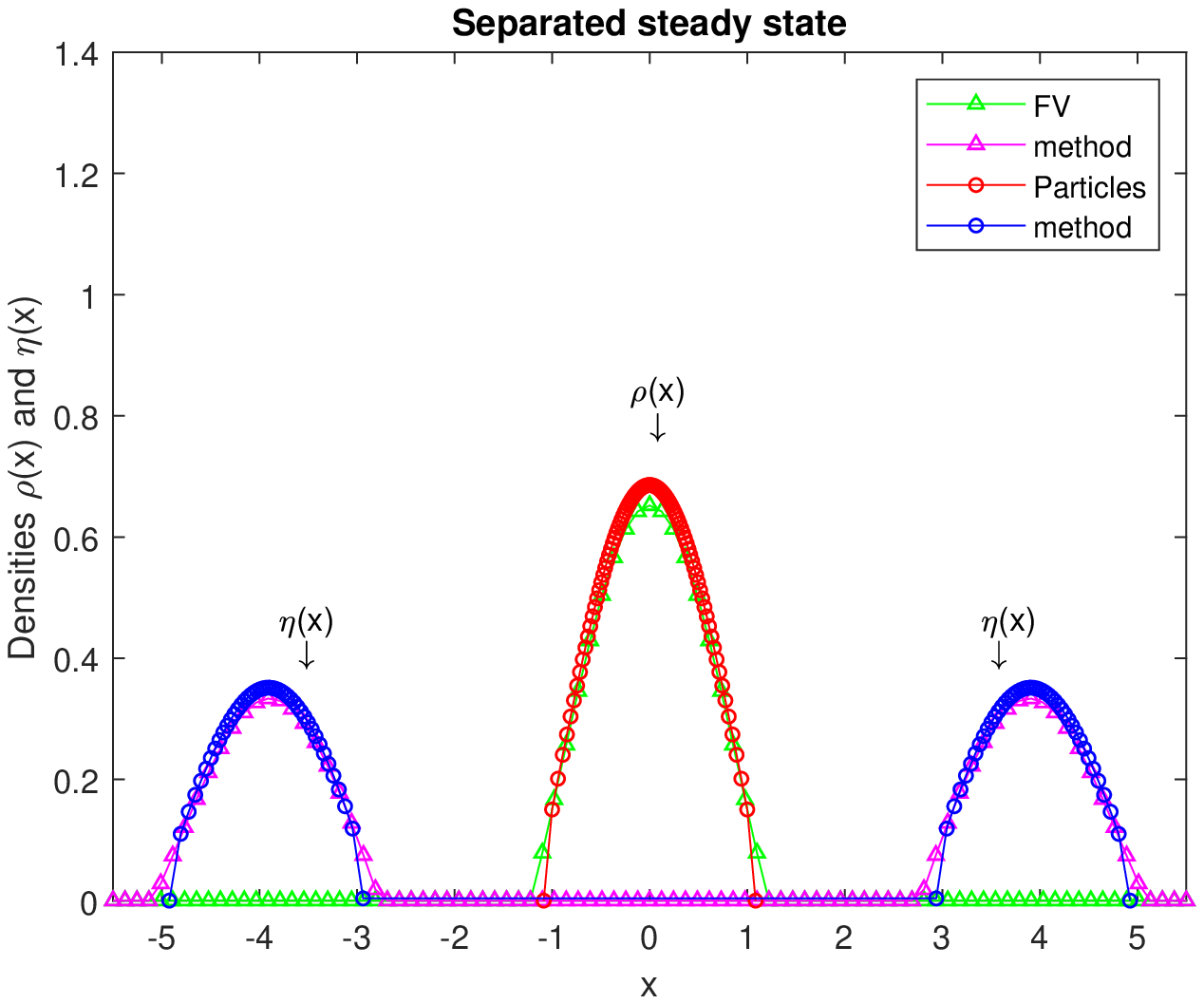}\par 
    \includegraphics[width=8cm,height=5cm]{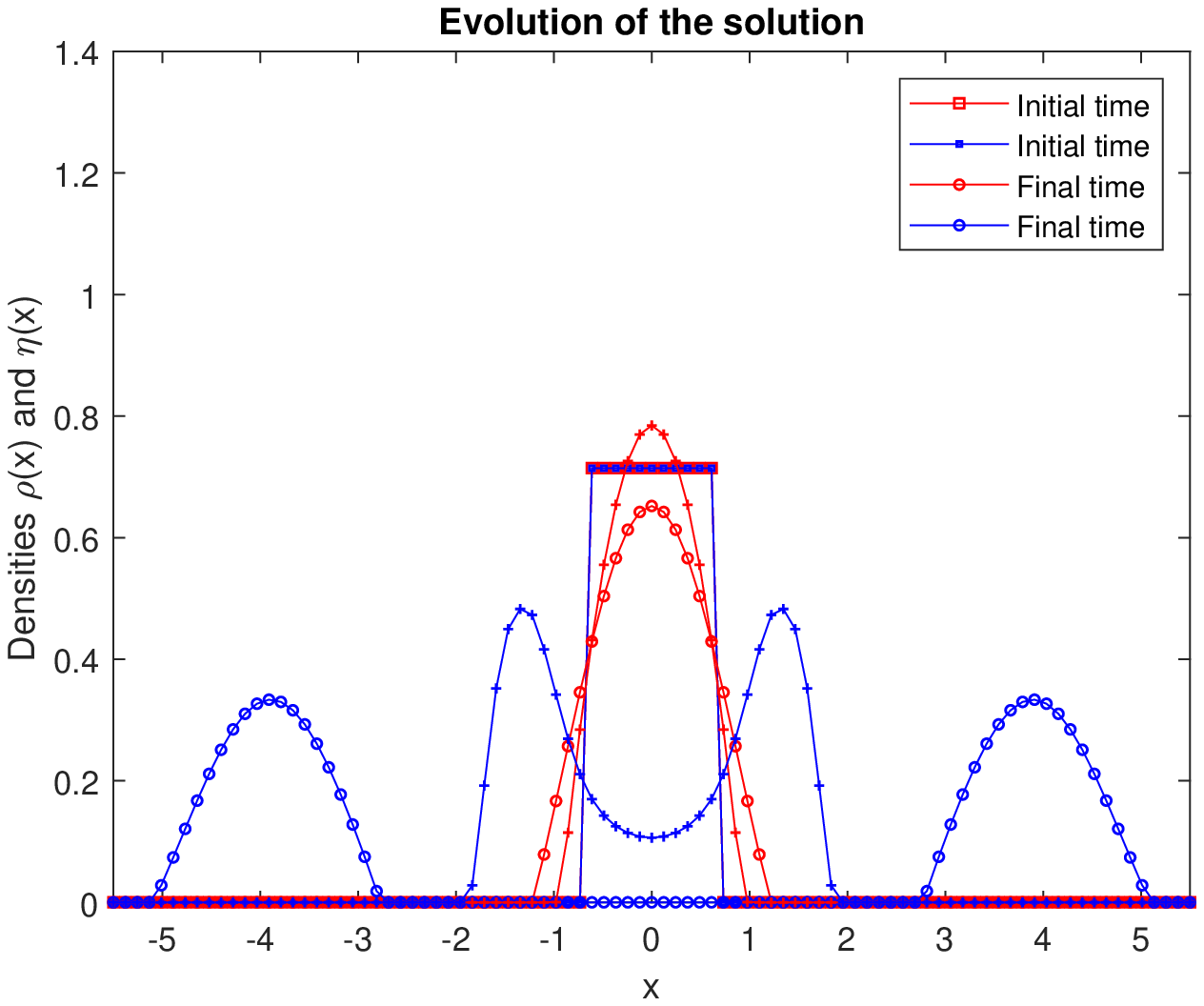}\par 
\end{multicols}
\begin{multicols}{2}
    \includegraphics[width=8cm,height=5cm]{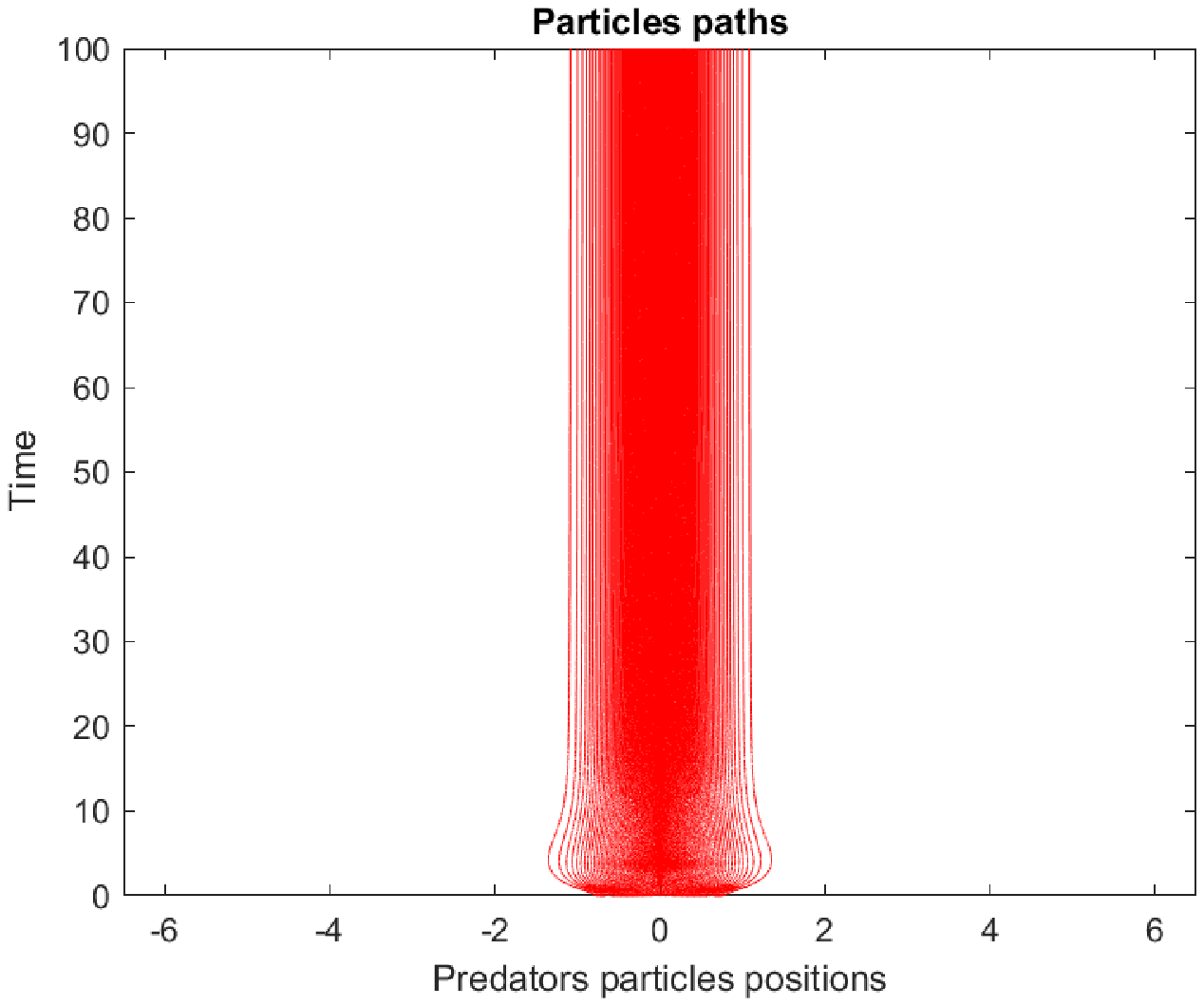}\par
    \includegraphics[width=8cm,height=5cm]{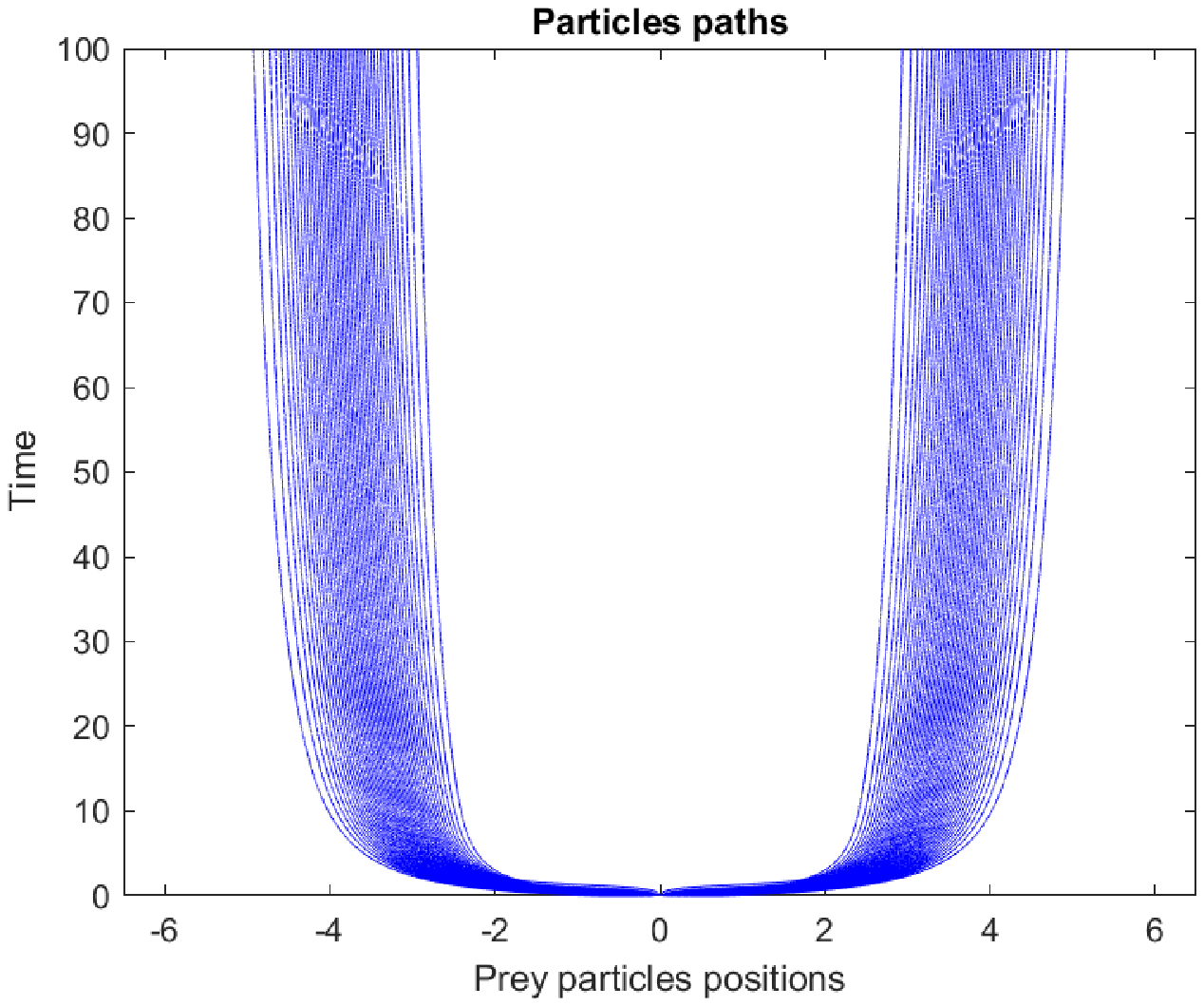}\par
\end{multicols}
\begin{multicols}{2}
    \includegraphics[width=8cm,height=5cm]{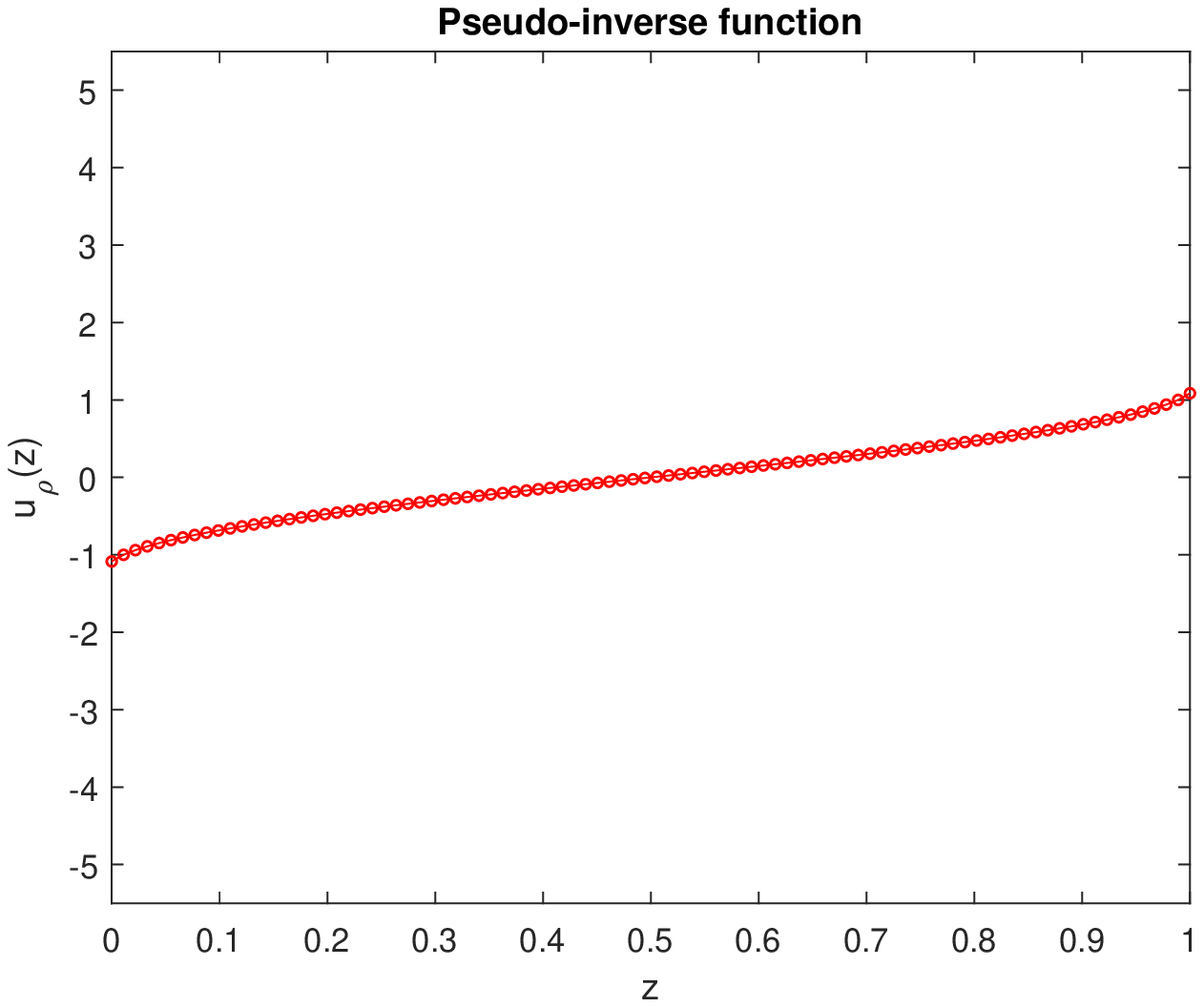}\par
    \includegraphics[width=8cm,height=5cm]{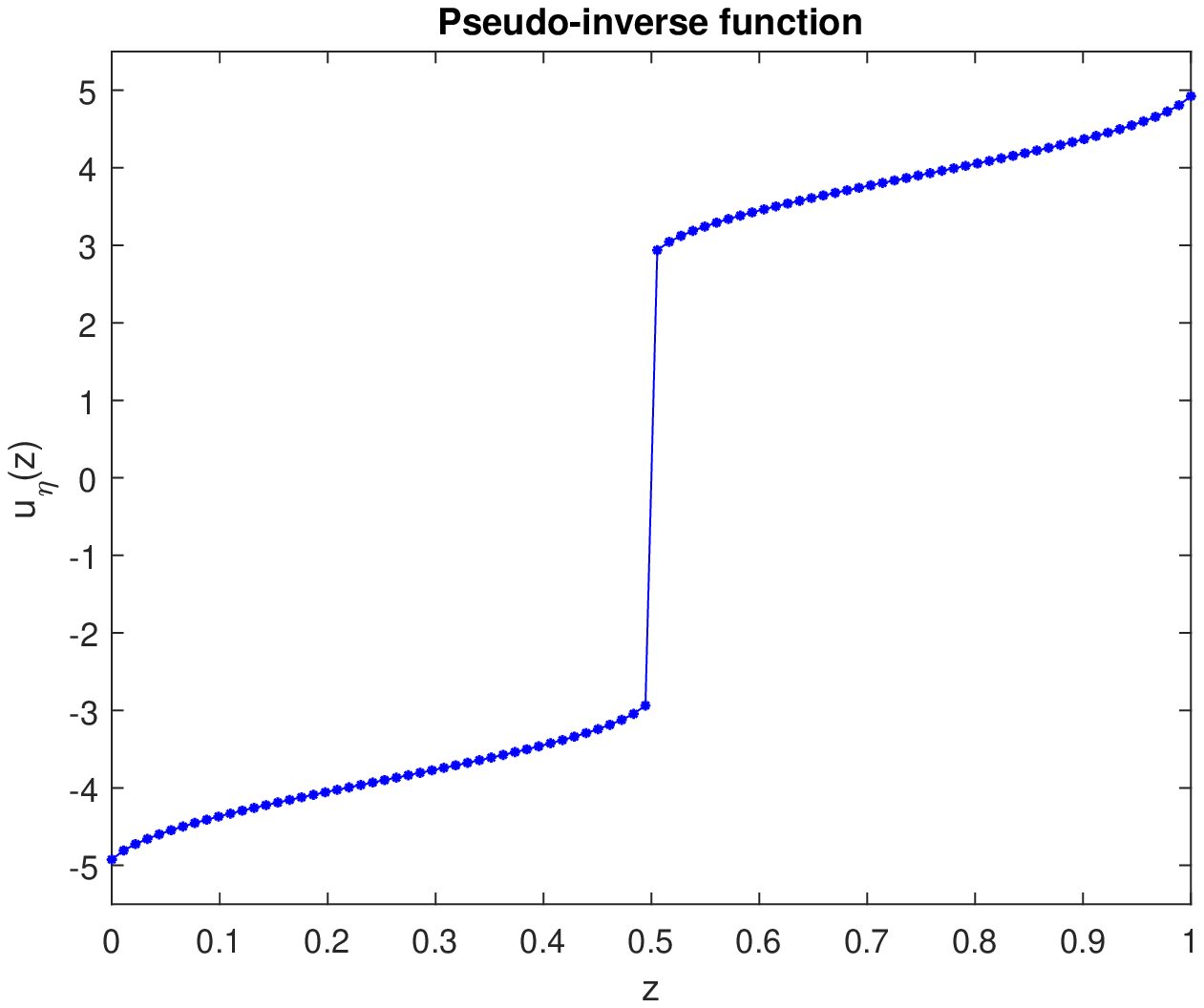}\par
\end{multicols}

\caption{This figure shows how from the initial densities $\rho_0,\eta_0$ and  $d$, as in Figure \ref{Mixed SS}, a transition between mixed and a sort of separated  steady state appears by choosing the value of $\alpha=6$. This large value of $\alpha$ suggests an \emph{unstable} behaviour in the profile, see Remark \ref{rem:conditions}.}
\label{From 1st to 2nd SS}
\end{center}
\end{figure}

\begin{figure}[htbp]
\begin{center}
\begin{multicols}{2}
    \includegraphics[width=8cm,height=5cm]{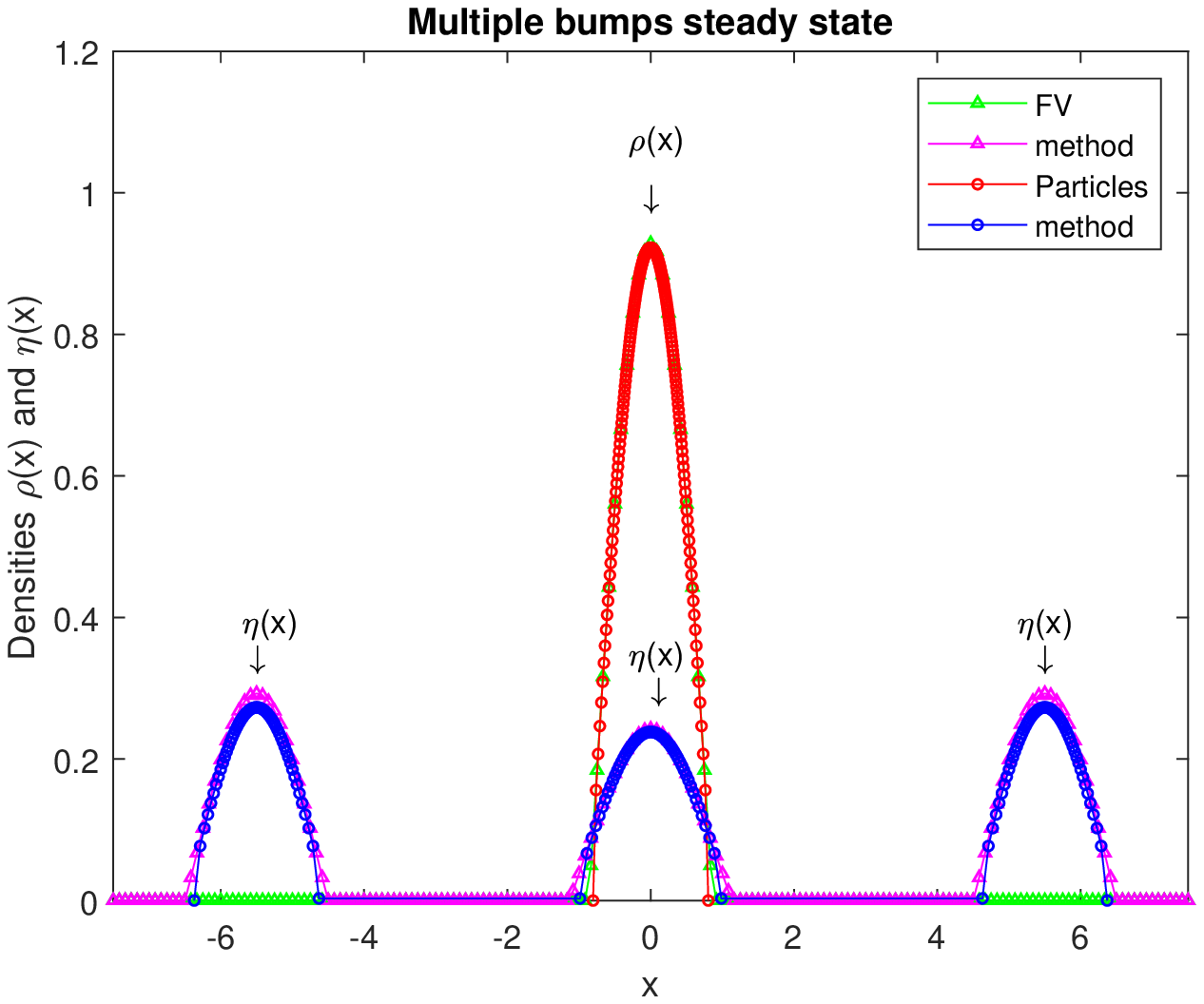}\par 
    \includegraphics[width=8cm,height=5cm]{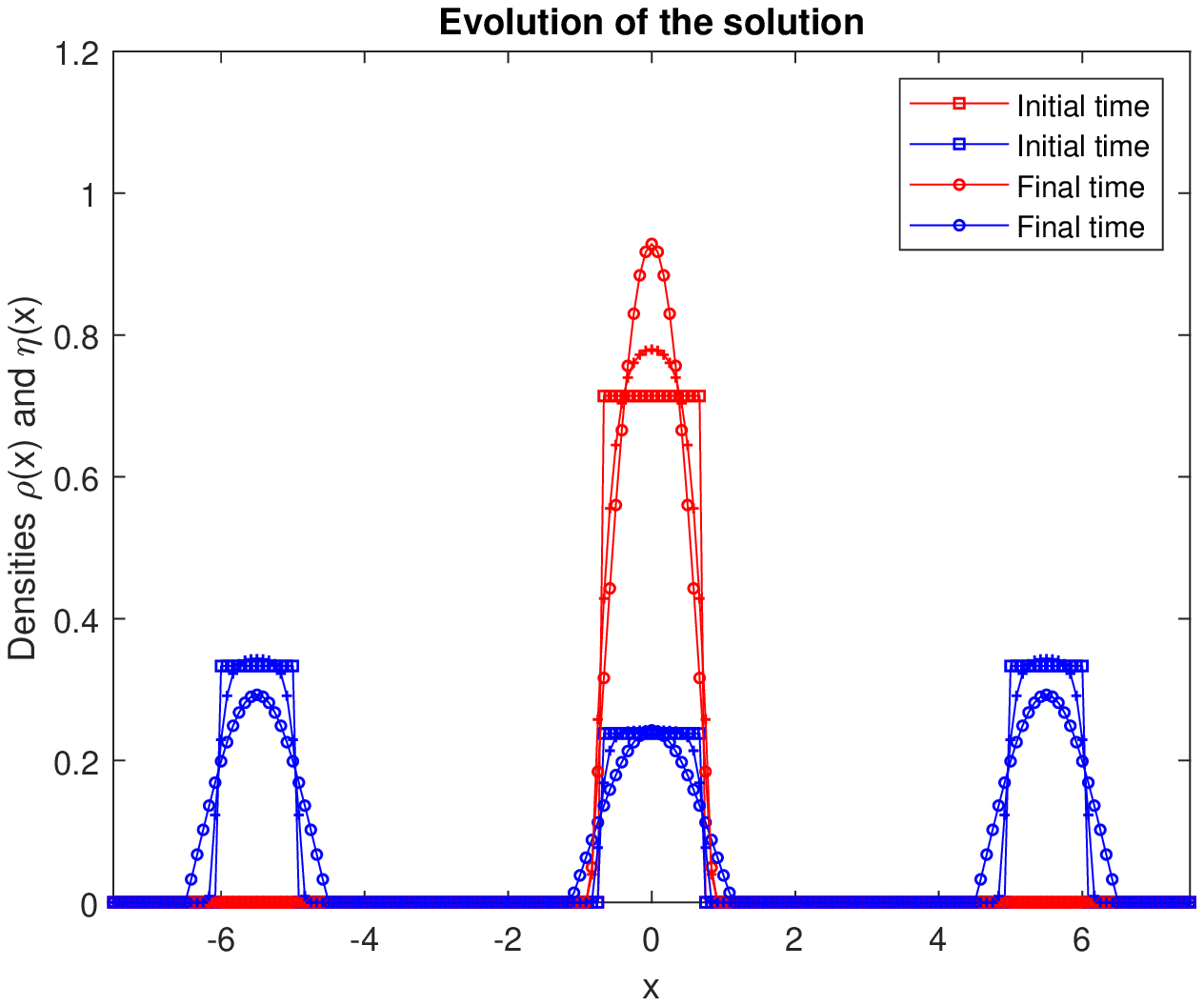}\par 
\end{multicols}
\begin{multicols}{2}
    \includegraphics[width=8cm,height=5cm]{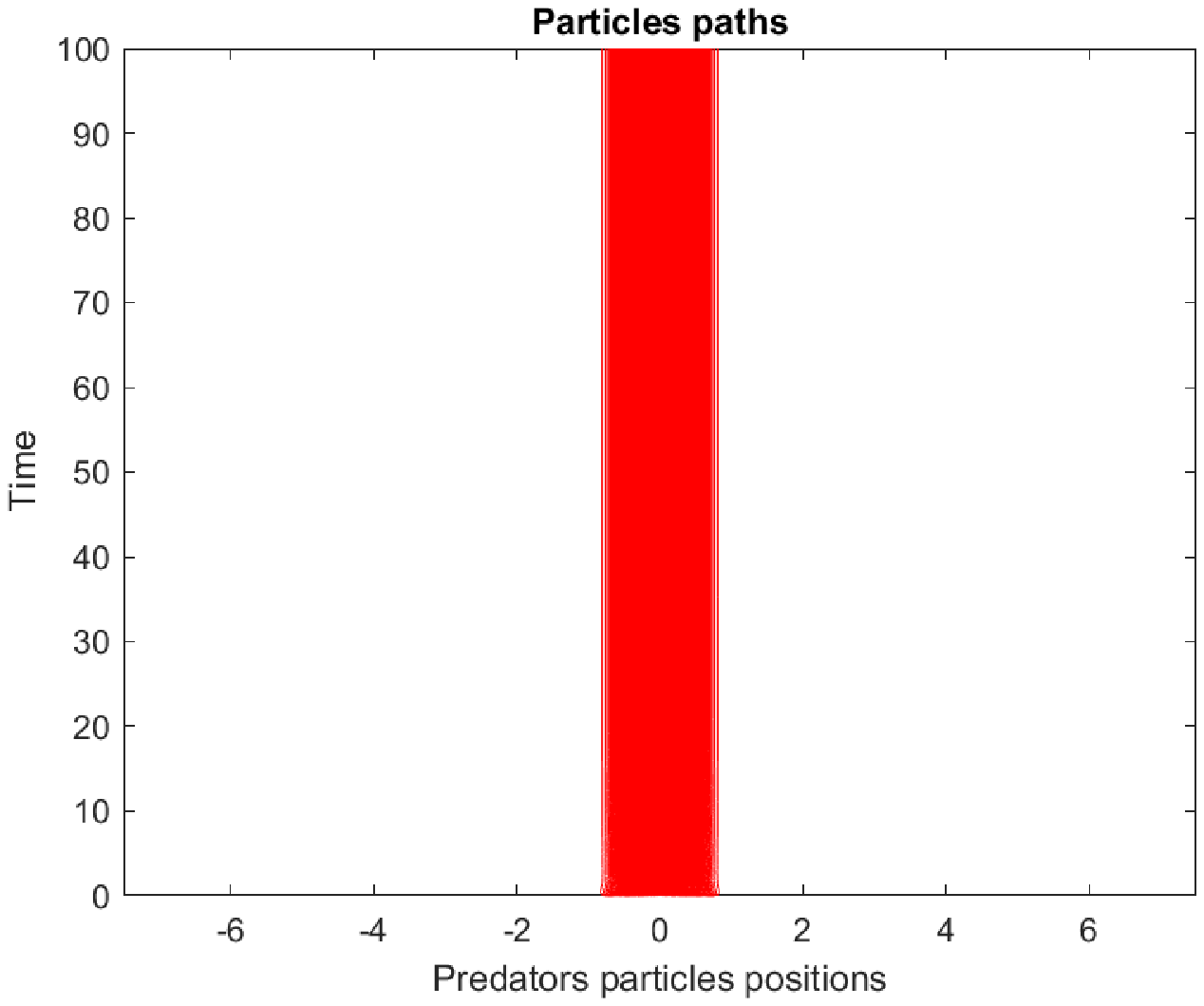}\par
    \includegraphics[width=8cm,height=5cm]{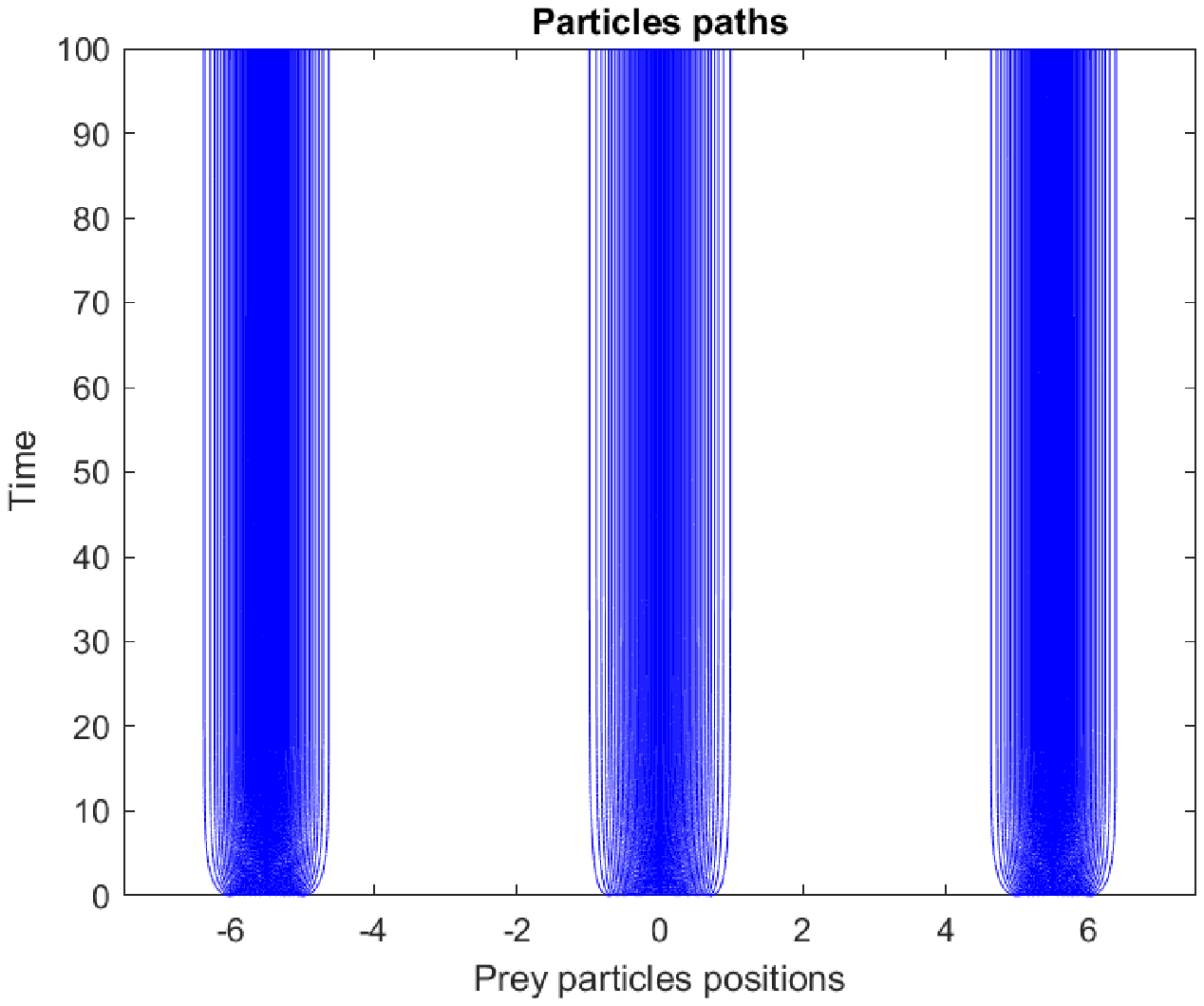}\par
\end{multicols}
\begin{multicols}{2}
\includegraphics[width=8cm,height=5cm]{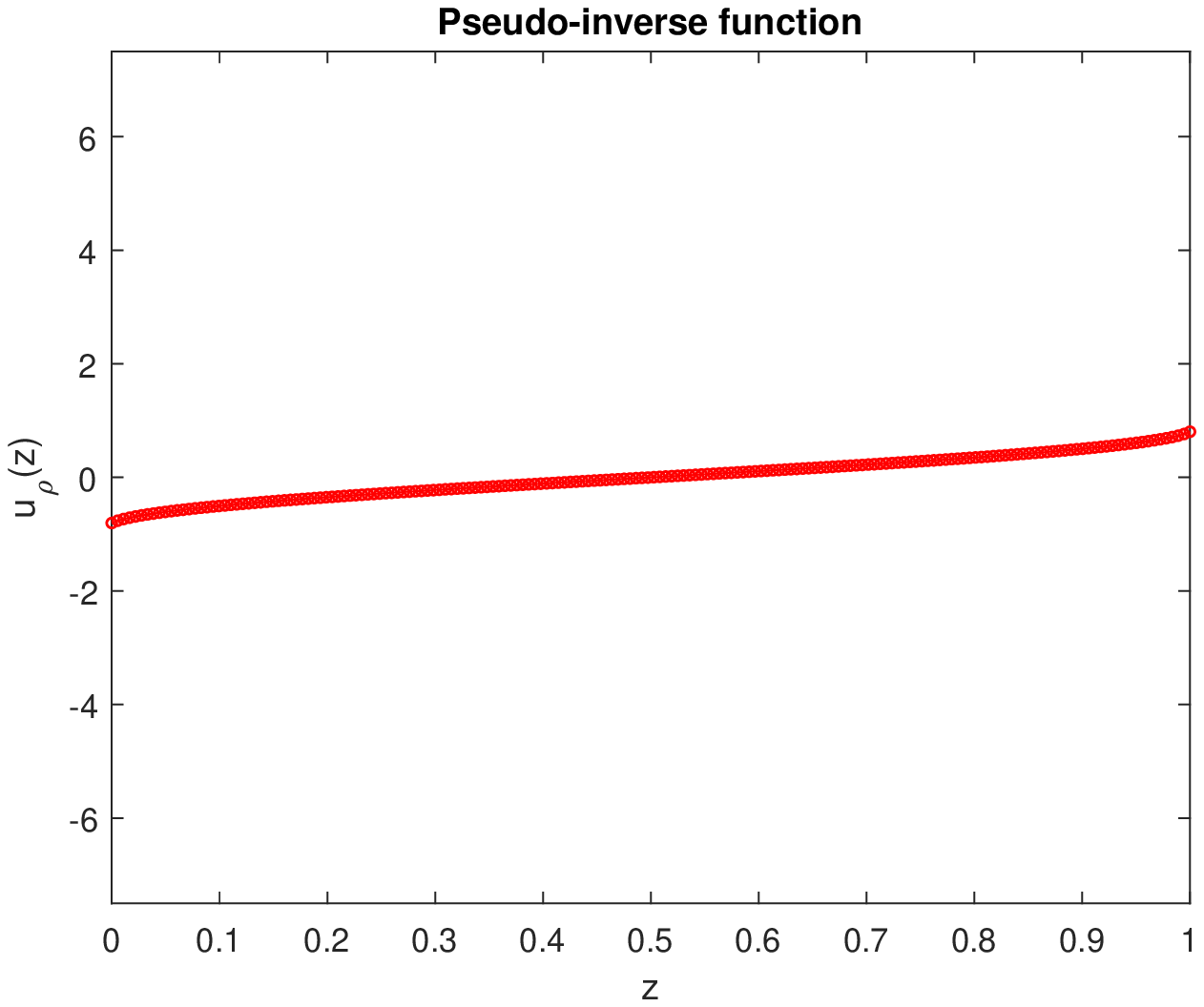}\par
   \includegraphics[width=8cm,height=5cm]{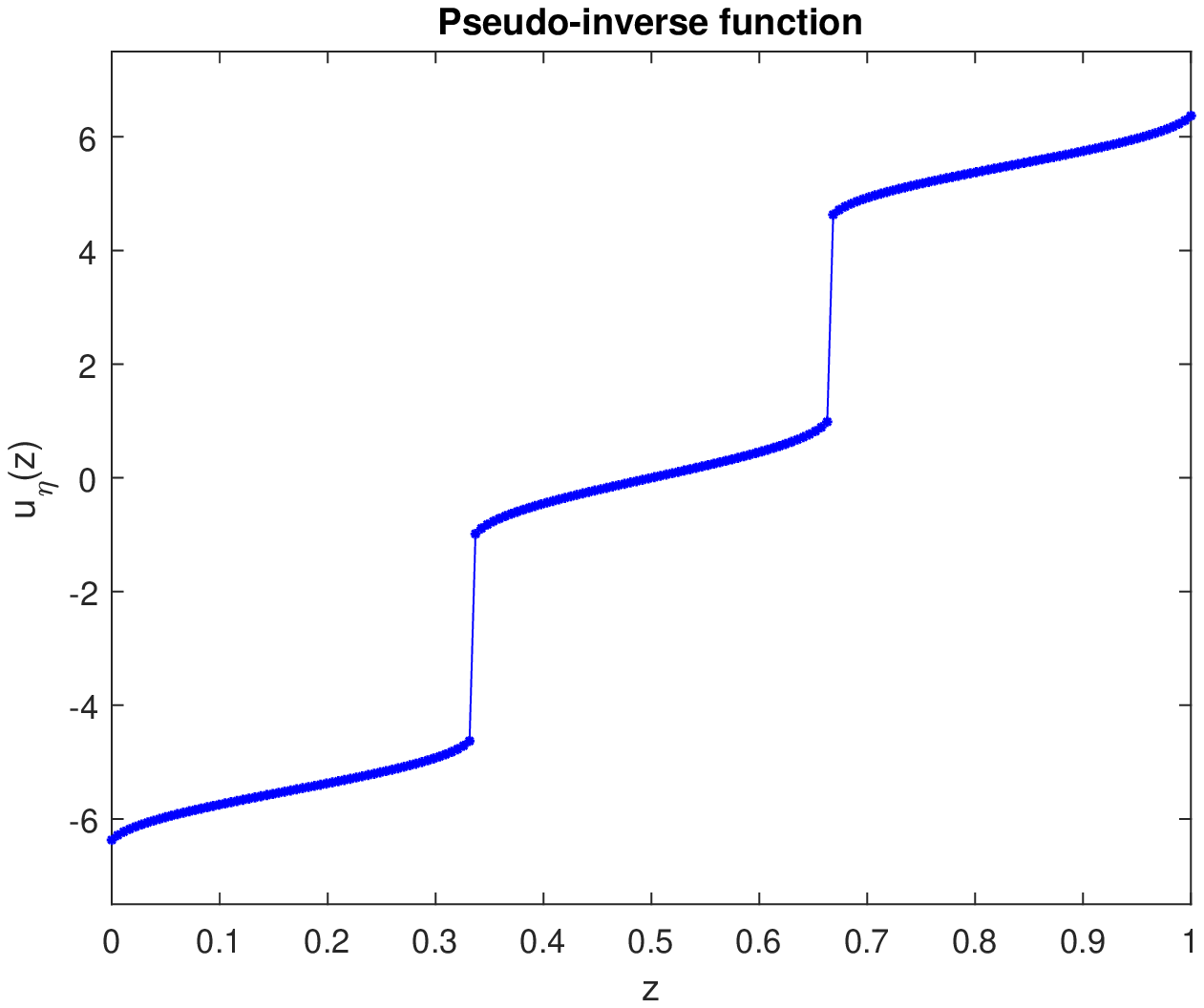}\par
\end{multicols}

\caption{A steady state of four bumps is showed in this figure starting from initial data as in \eqref{initial3} with $\alpha=0.05$ and $d=0.3$. The number of particles $N=181$, which is the same as number of cells. }
\label{Multi SS2}
\end{center}
\end{figure}

\begin{figure}[htbp]
\begin{center}
\begin{multicols}{2}
    \includegraphics[width=8cm,height=5cm]{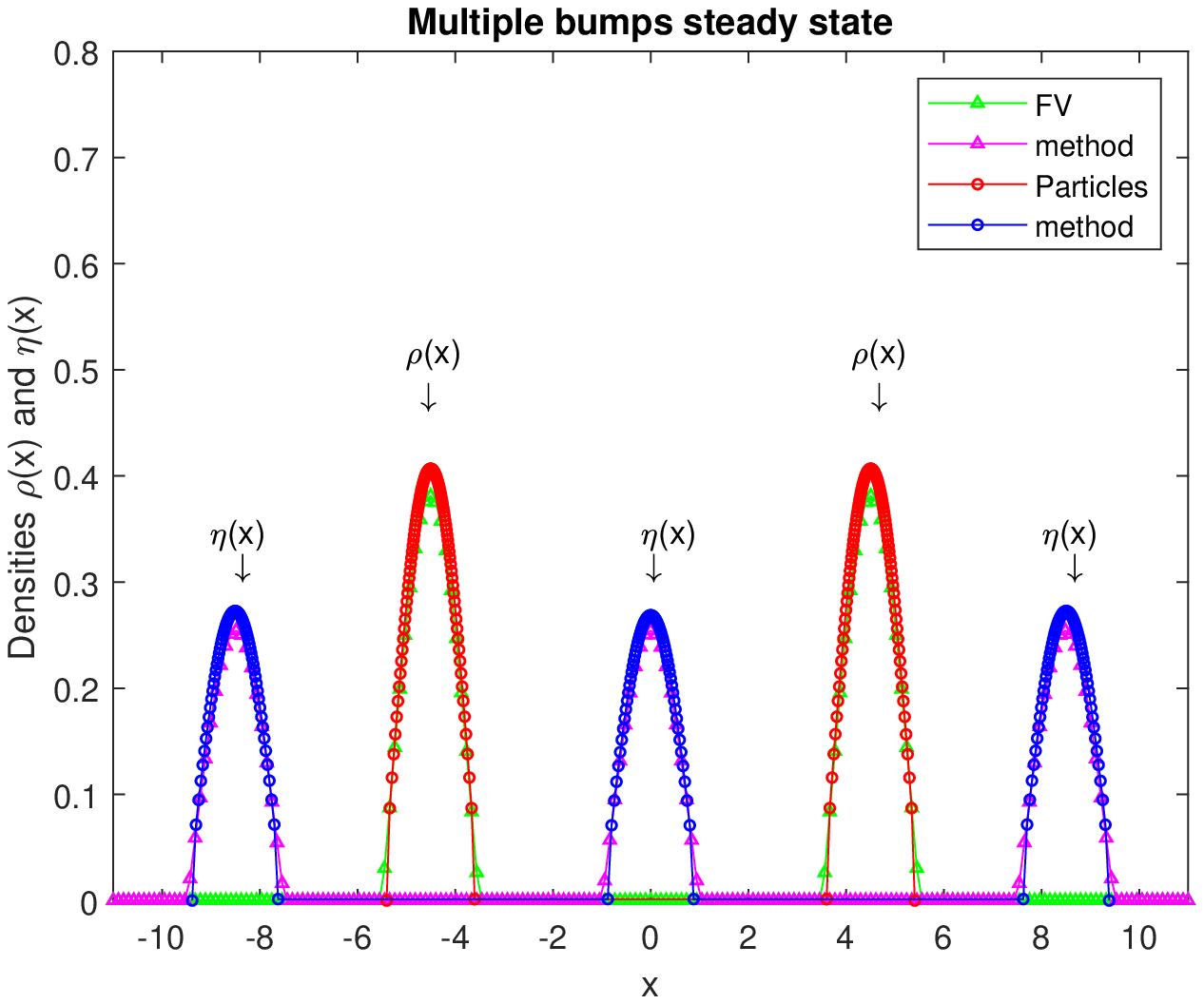}\par 
    \includegraphics[width=8cm,height=5cm]{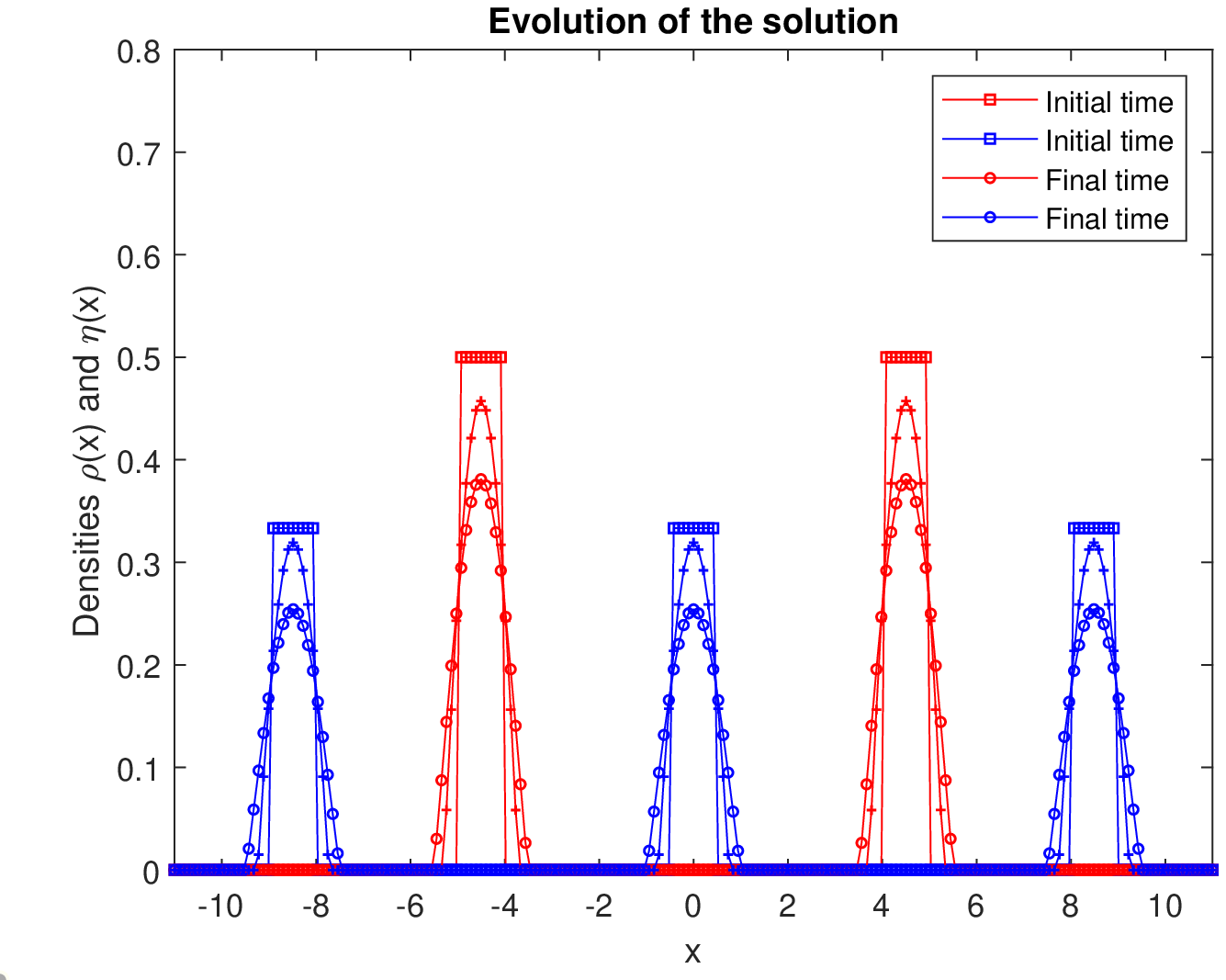}\par 
\end{multicols}
\begin{multicols}{2}
    \includegraphics[width=8cm,height=5cm]{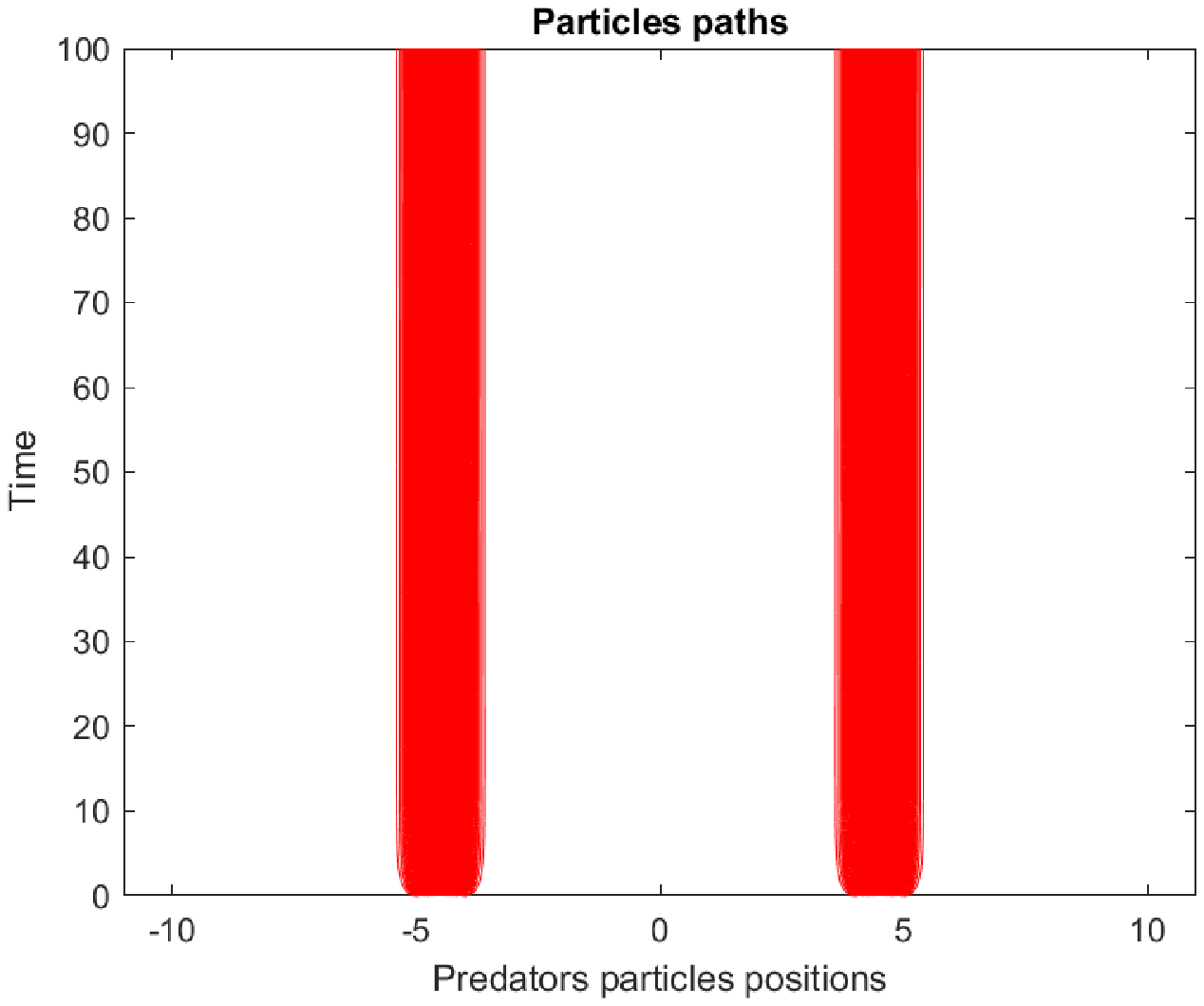}\par
    \includegraphics[width=8cm,height=5cm]{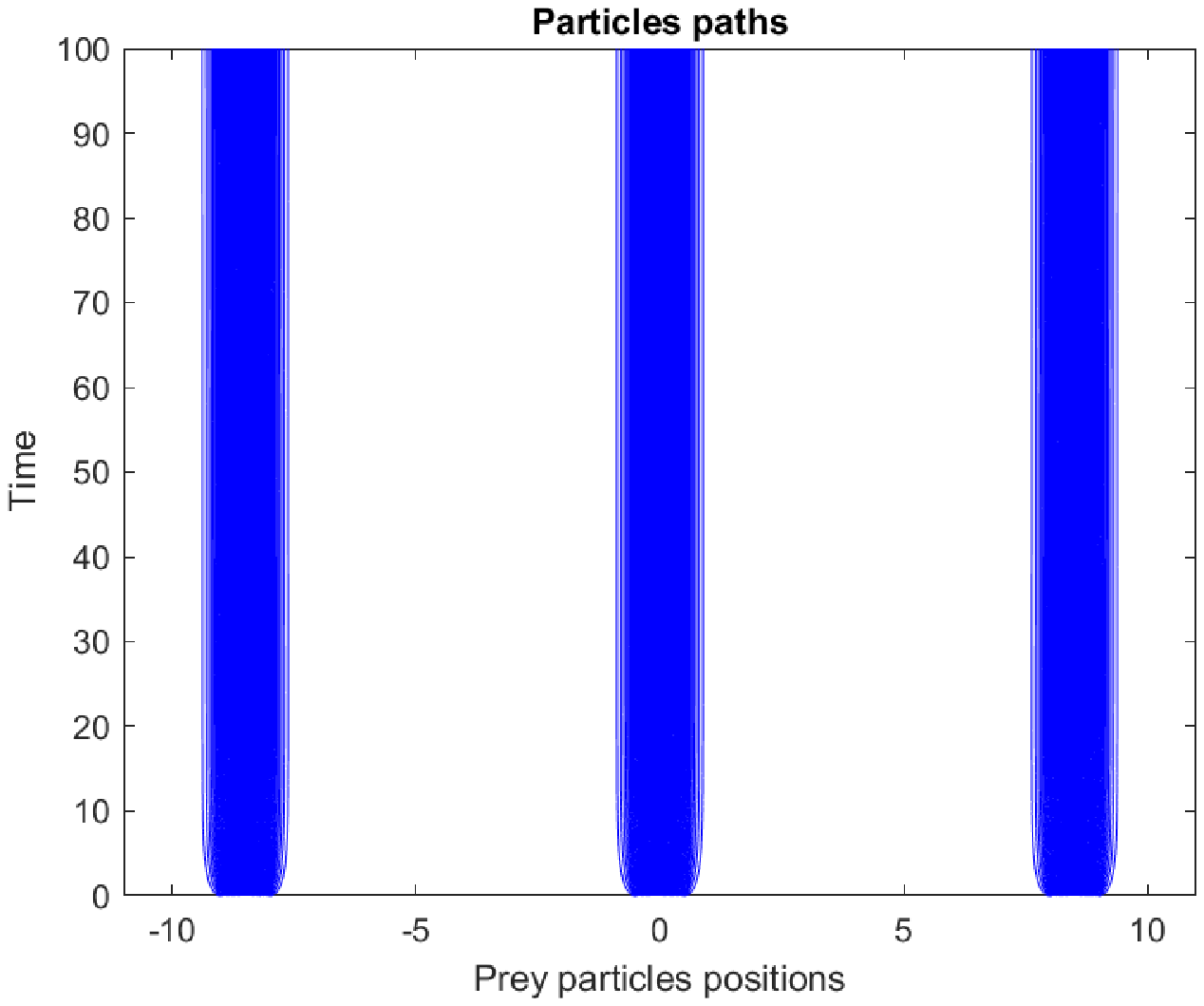}\par
\end{multicols}
\begin{multicols}{2}
    \includegraphics[width=8cm,height=5cm]{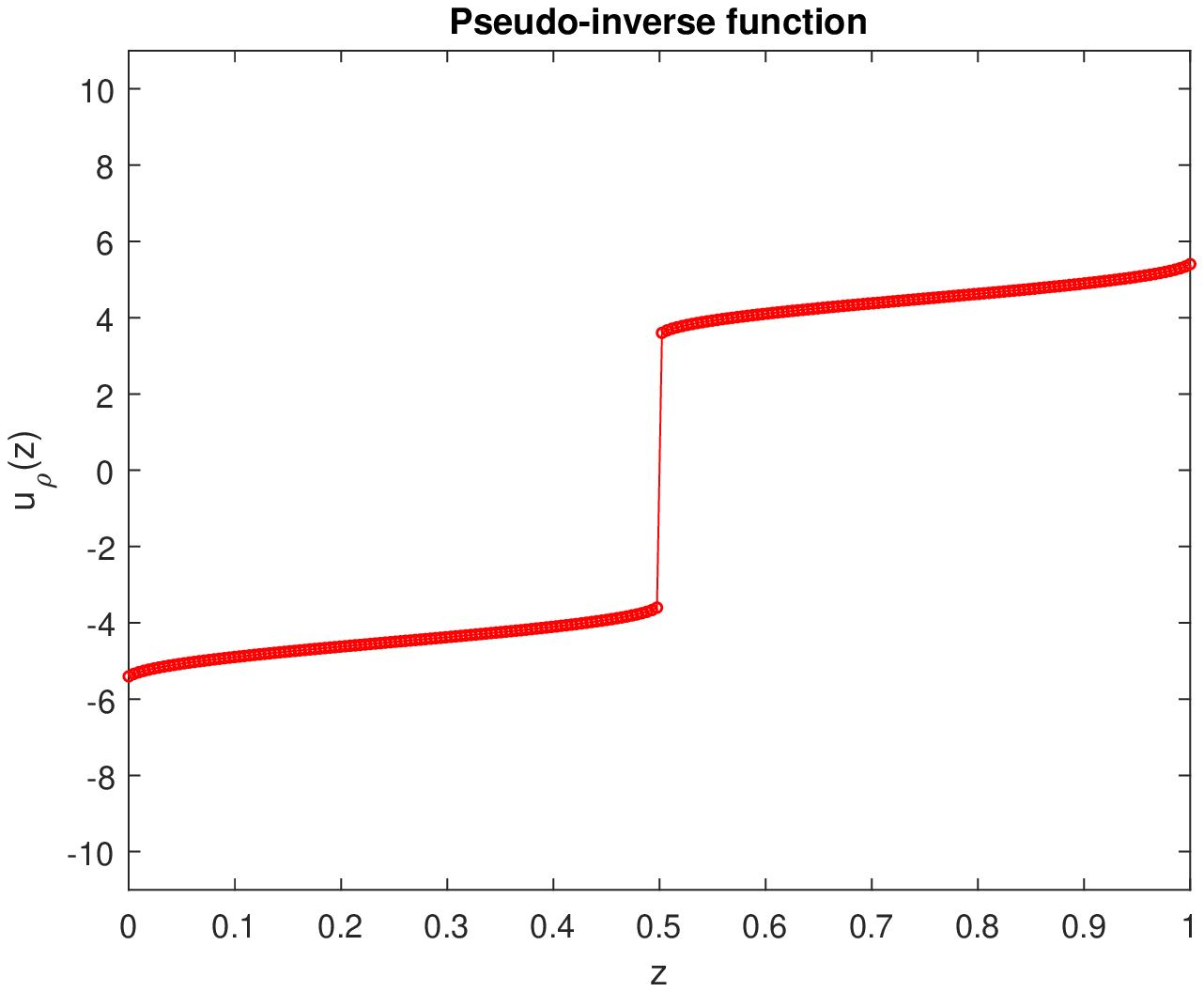}\par
    \includegraphics[width=8cm,height=5cm]{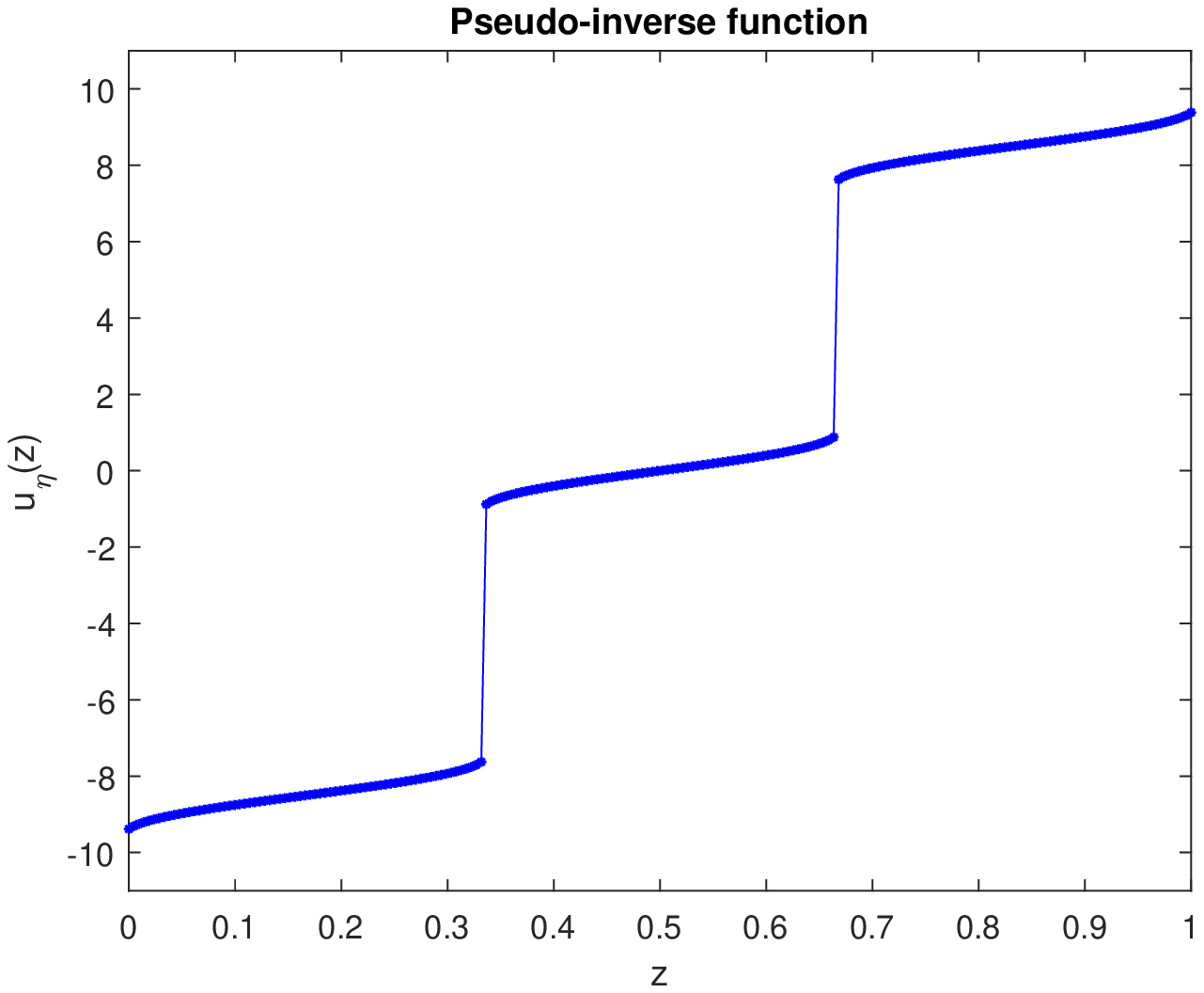}\par
\end{multicols}

\caption{Starting from initial data as in \eqref{initial4} with $\alpha=1$ and $d=0.3$, we get a five bumps steady state.}
\label{Multi SS1}
\end{center}
\end{figure}

\begin{figure}[htbp]
\begin{center}
\includegraphics[width=10cm,height=5cm]{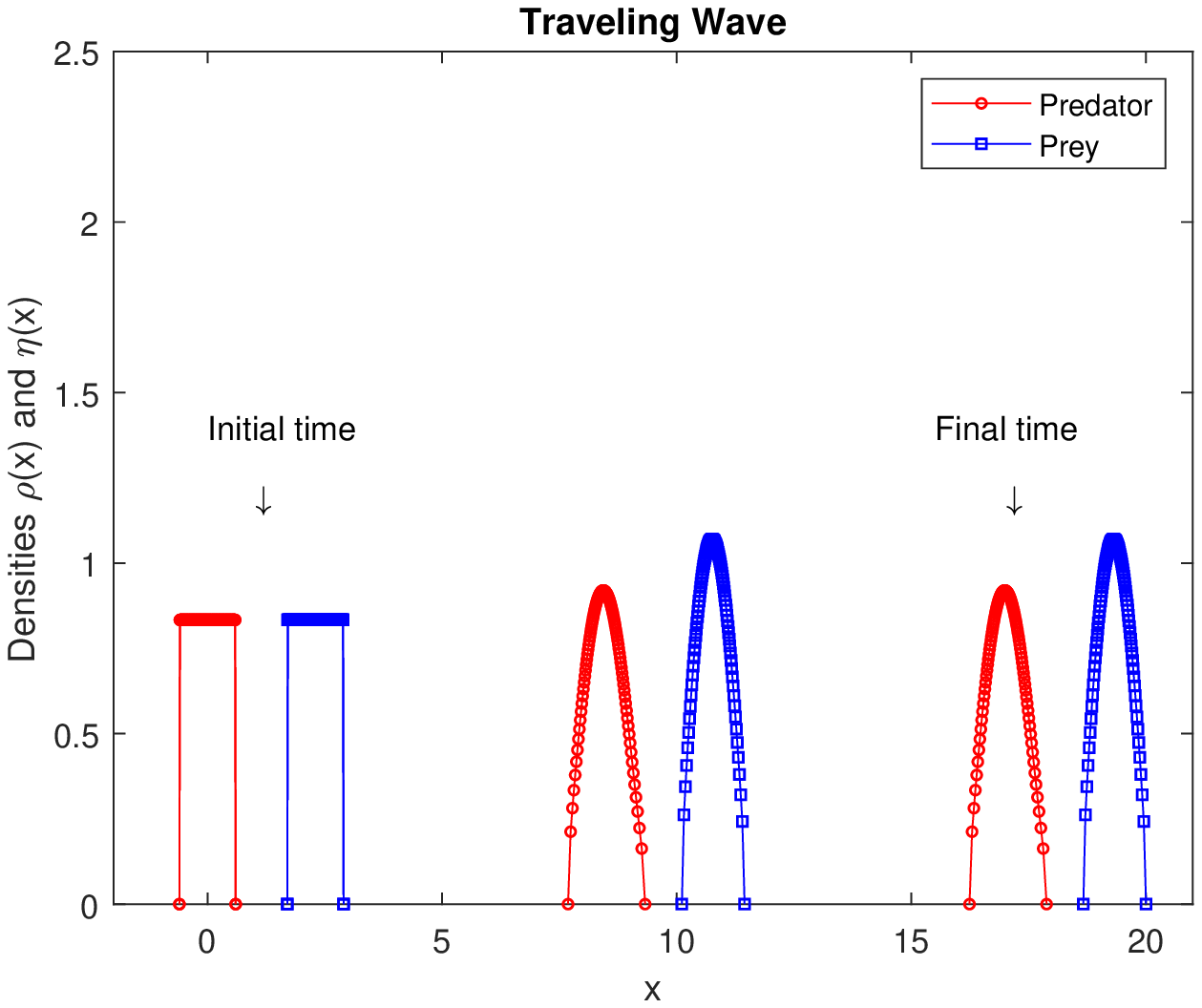}

\includegraphics[width=10cm,height=5cm]{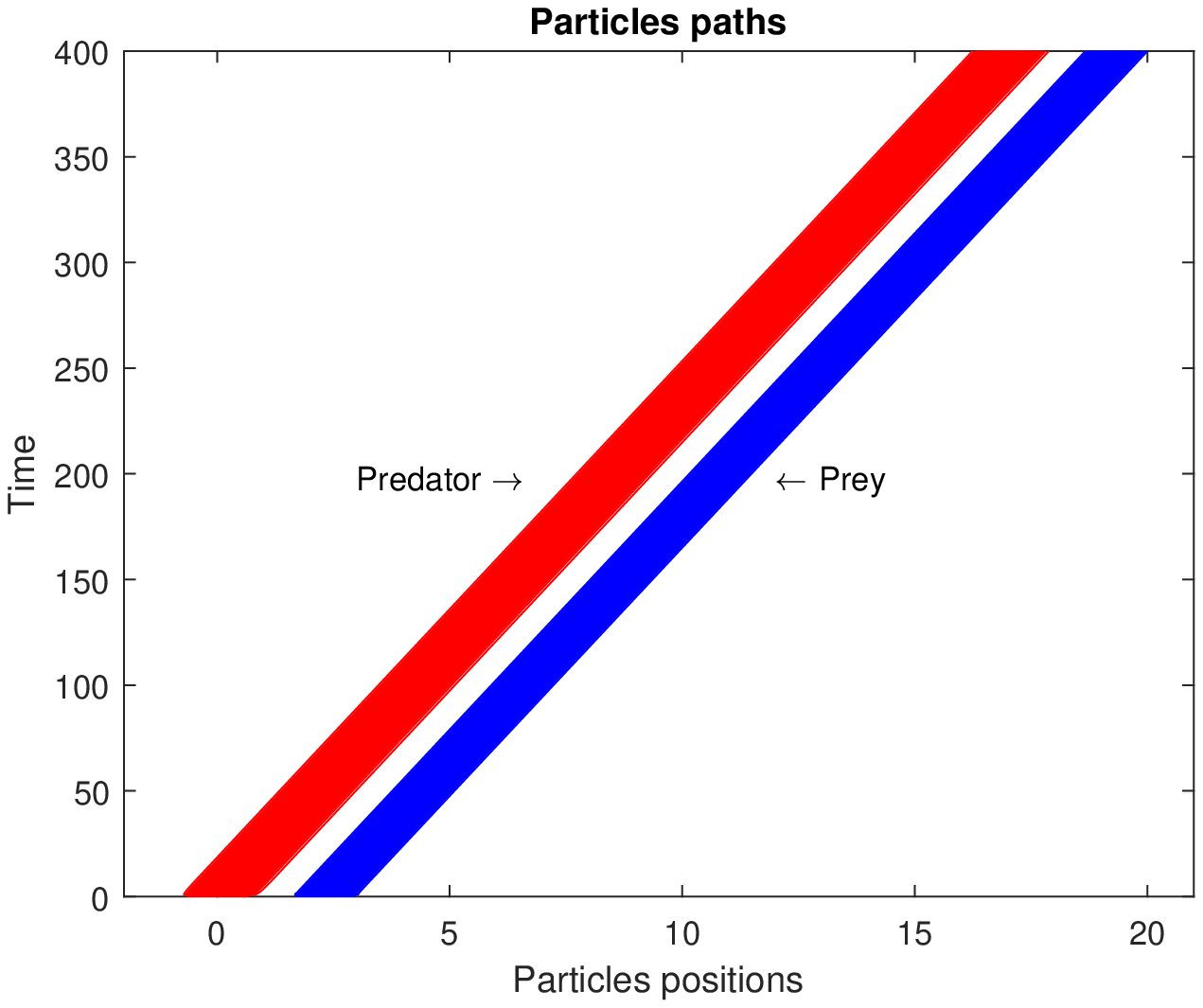}

\includegraphics[width=10cm,height=5cm]{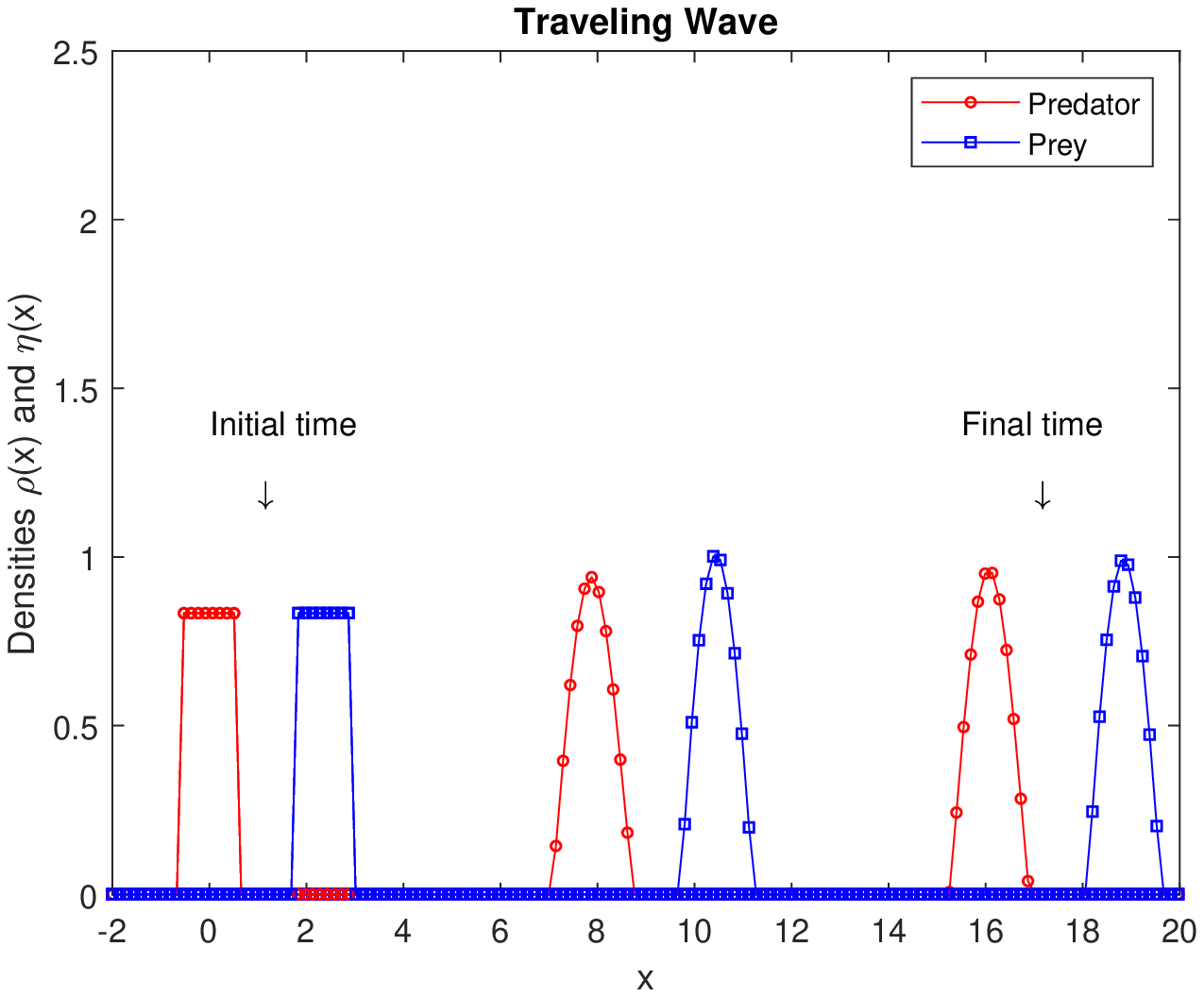}

\caption{This last figure shows a possible existence of traveling waves by choosing initial data as in \eqref{initial5}, $\alpha=1$ and $d=0.2$. The first two plots are performed by particles method, while the third one is done by finite volume method. Here we fix $N=101$. }
\label{Wave SS}
\end{center}
\end{figure}

In all the simulations below, we will fix the kernels  as a normalised Gaussian potentials
$$
S_{\rho}(x) = S_{\eta}(x) = K(x)
 = \frac{1}{\sqrt{\pi}} e^{-x^2},
$$
that are under the assumptions on the kernels (A1), (A2) and (A3). This choice helps us in better understanding the variation in the behavior of the solutions w.r.t. the change in the initial data and the parameter $\alpha$. In the first five examples, see Figures \ref{Mixed SS}, \ref{Separated SS}, \ref{From 1st to 2nd SS}, \ref{Multi SS2} and  \ref{Multi SS1} we show: 
\begin{itemize}
    \item in the first row  steady states are plotted at the level of density, on the l.h.s.  we compare the two methods illustrated above, while on the r.h.s. we show the evolution by the finite volume method,
    \item in the second row we plot the particles paths for both species obtained with the particles method,
    \item  in the last row we show the pseudo inverse functions corresponding to the steady state densities.
\end{itemize}
The last example we present shows an interesting \emph{travelling waves}-type evolution.

The first example is devoted to validating existence of mixed steady state and separated steady state.  By choosing the initial data $(\rho_0,\eta_0)$ as
\begin{equation}\label{initial1}
\rho_0(x) = \eta_0(x) = \frac{10}{14} \mathds{1}_{[-0.7,0.7]}(x),
\end{equation}
and fixing $\alpha=0.1$ and $d=0.4$, we obtain a mixed steady state as plotted in Figure \ref{Mixed SS}. Note that, the small value of $\alpha$ allows the predators to dominate the prey which results in the shape shown in Figure \ref{Mixed SS}. Next, we take the initial data as 
\begin{equation}\label{initial2}
\rho_0(x) = 0.5 \mathds{1}_{[-1,1]}(x), \qquad \eta_0(x) = 0.5 \mathds{1}_{[-4,-3]\cup[3,4]}(x),
\end{equation}
with the same $d$ as above and $\alpha=0.2$. This choice of the initial data, actually, introduce two equal attractive forces on the right and left hand sides of the predators which, in turn, fix the predators at the centre and gives the required shape of the separated steady state as shown in Figure \ref{Separated SS}. Finally, a sort of  separated steady state can also be obtained starting from the same initial data as in \eqref{initial1} and same diffusion parameter $d$. If we choose $\alpha=6$, the prey $\eta$ will have enough speed to get out of the predators region, producing a transition between the mixed state to the separated one, this is illustrated in Figure \ref{From 1st to 2nd SS}. 

We then, test two cases where we validate the existence theory of multiple--bumps steady states done in Section \ref{sec:implicit}. Let us fix $d=0.3$. Then, a four-bumps steady state is performed and plotted in Figure \ref{Multi SS2}, where we consider 
\begin{equation}\label{initial3}
\rho_0(x) =  \frac{10}{14}\mathds{1}_{[-0.7,0.7]}(x), \qquad \eta_0(x) = \frac{5}{21} \mathds{1}_{[-0.7,0.7]}(x) + \frac{1}{3} \mathds{1}_{[-6,-5]\cup[5,6]}(x),
\end{equation}
as initial data and we take $\alpha=0.05$. This way of choosing initial data produces a balanced attractive forces which in turn form the required steady state. Similarly, we can obtain a steady state of five bumps by using initial data
\begin{equation}\label{initial4}
\rho_0(x) =  \frac{1}{2} \mathds{1}_{[-5,-4]\cup[4,5]}(x), \qquad \eta_0(x) =  \frac{1}{3} \mathds{1}_{[-9,-8]\cup[-0.5,0.5]\cup[8,9]}(x),
\end{equation}
and $\alpha=1$, see Figure \ref{Multi SS1}. 

 In this example, we finally  detect existence of traveling waves. Indeed, by choosing  initial data as
\begin{equation}\label{initial5}
\rho_0(x) = \frac{10}{12}\mathds{1}_{[-0.6,0.6]}(x), \qquad \eta_0(x) = \frac{10}{12}\mathds{1}_{[1.7,2.9]}(x),
\end{equation}
$\alpha = 1$ and $d=0.2$, we obtain a traveling wave that is shown in Figure \ref{Wave SS}. Once the initial data and the value of $d$ are fixed, if $\alpha$ is taken small enough, then we will come out with a mixed steady state. If we take a larger value for $\alpha$ then the prey will be fast escaping from the predators. Therefore, the proper value for $\alpha$ will produce a situation where the \textit{attack speed} of the predators is equal to the \textit{escape speed} of the prey, which results in a travelling wave of both the densities $\rho$ and $\eta$. All the simulations above motivate further


\end{document}